\documentclass[english,12pt,leqno]{article}
\usepackage[T1]{fontenc}
\usepackage{float}
\usepackage{amsmath,amsthm,amssymb,mathtools} 
\usepackage{pxfonts}
\usepackage{graphicx}
\usepackage{enumitem,epigraph}
\usepackage{csquotes}
\usepackage[style=oxyear,dashed, giveninits=false,maxnames=10,uniquename=false,doi=false,useprefix=true,sortcites=false,mergedate=basic,natbib]{biblatex}
\usepackage{linguex}

\usepackage{nicefrac,stmaryrd}
\usepackage[letterpaper]{geometry} 
\usepackage[title]{appendix}
\usepackage[tiny,pagestyles]{titlesec}
\usepackage{xspace}
\usepackage{accents}
\usepackage{todonotes}
\setuptodonotes{size=footnotesize}
\newcommand\disj{{\textstyle\bigvee}}

\usepackage{xcolor}
\usepackage{pdfcolfoot}

\definecolor{darkgreen}{rgb}{0.0, 0.7, 0.0}
\DeclareSymbolFont{symbolsC}{U}{txsyc}{m}{n}
\DeclareMathSymbol{\boxright}{\mathrel}{symbolsC}{128}
\usepackage{thmtools}
\usepackage{stmaryrd}  
\usepackage{comment}
\usepackage[breaklinks]{hyperref}
\hypersetup{colorlinks=true, citecolor=blue, linkcolor=blue}  
\usepackage[noabbrev,capitalise]{cleveref}
\AtBeginEnvironment{appendices}{\crefalias{section}{appendix}}
\crefformat{equation}{#2(#1)#3}
\crefformat{section}{#2\S #1#3}
\crefformat{subsection}{#2\S #1#3}

\usepackage{thm-restate}

\declaretheorem[style=definition,name=Theorem,numberwithin=section]{theorem}

\usepackage{tikz}
\usepackage{pgf} 
\usetikzlibrary{patterns,automata,arrows,shapes,topaths,trees,positioning,through,calc}
\usetikzlibrary{lindenmayersystems,decorations.pathreplacing,decorations.markings,arrows.meta}
\usepackage{setspace}
\usepackage{multicol}
\usepackage{graphicx}

\declaretheorem[style=definition,sibling=theorem]{corollary,definition,lemma,fact,example}

\usepackage{mdframed}

\newcommand{\bi}{\begin{itemize}}
\newcommand{\ei}{\end{itemize}}
\newcommand{\cond}{>}
\newcommand{\stal}{\textsf{C2}\xspace}
\newcommand{\stalp}{\textsf{C2}\xspace}
\newcommand{\vanf}{\textsf{C2.FS}\xspace}
\newcommand{\vanfp}{\textsf{C2.FS}\xspace}
\newcommand{\vflat}{\textsf{C2.F}\xspace}
\newcommand{\vflatp}{\textsf{C2.F}\xspace}
\usepackage{cleveref}
\usepackage[draft,inline,nomargin,marginclue]{fixme}
\definecolor{cobalt}{RGB}{0,71,171}
\definecolor{brickred}{RGB}{203, 65, 84}

\newcommand{\seq}[1]{\langle {#1} \rangle}

\newcommand{\set}[1]{\{ {#1} \}}

\FXRegisterAuthor{mm}{amm}{\color{cobalt}MM}
\FXRegisterAuthor{cd}{acd}{\color{brickred}CD}
\FXRegisterAuthor{wh}{awh}{\color{brickred}WH}
\fxsetup{inlineface=\sffamily}
\fxsetup{marginface=\sffamily}
\usepackage{harpoon,pifont}

\newcommand{\sem}[2][]{\mbox{$\llbracket{#2}\rrbracket^{#1}$}}

\makeatother
\usepackage{babel}  

\newcommand{\negate}[1]{\overline{#1}}
\newcommand{\snegate}[1]{\kern1pt\overline{#1}}

\newcommand{\makeseq}[1]{\mathord{\uparrow}#1}
\newcommand{\makelongseq}[1]{\mathord{\uparrow}^+#1}
\newcommand{\makesentence}[1]{\underline{#1}}

\newcommand{\ourtag}[1]{\tag*{\textbf{#1}}}

\newcommand{\restrict}[2]{#1\mathord\upharpoonright#2}
\newcommand{\slice}[1]{^{[#1]}}
\newcommand{\elem}[1]{^{#1}}

\newcommand{\boldtag}[1]{\tag*{\textbf{#1}}}

\newcommand{\firstdeg}{(\mathcal{B}>\mathcal{B})}
\newcommand{\firstdegconj}{\seq{\mathcal{B}>\mathcal{B}}_{\wedge}}

\pgfdeclaredecoration{arrows}{draw}{
\state{draw}[width=\pgfdecoratedinputsegmentlength]{%
  \path [every arrow subpath/.try] \pgfextra{%
    \pgfpathmoveto{\pgfpointdecoratedinputsegmentfirst}%
    \pgfpathlineto{\pgfpointdecoratedinputsegmentlast}%
   };
}}
\tikzset{every arrow subpath/.style={{Circle[length=0.6pt,sep=-0.4pt,black]}-, draw, very thin, gray}}

\pgfdeclarelindenmayersystem{mytree}{
  \symbol{F}{\pgflsystemdrawforward}
  \symbol{S}{\pgflsystemstep=0.6\pgflsystemstep}
  \rule{A->FS[-A][A]}
}
\pgfdeclarelindenmayersystem{mynewtree}{
  \symbol{F}{%
    \pgflsystemstep=2.5\pgflsystemstep%
    \pgflsystemdrawforward%
    \pgflsystemturnleft\pgflsystemturnleft%
    \pgflsystemmoveforward%
    \pgflsystemturnleft\pgflsystemturnleft%
    \pgflsystemstep=0.4\pgflsystemstep%
    \pgflsystemmoveforward%
    }
  \symbol{G}{\pgflsystemmoveforward}
  \symbol{S}{\pgflsystemstep=0.6\pgflsystemstep}
  \rule{A->FS[+A][B]}
  \rule{B->GS[+A][B]}
}

\newcommand{\tails}[1]{\operatorname{tails}(#1)}
\newcommand{\eglist}{\tau}
\newcommand{\chooselist}{\Phi}
\newcommand{\pq}{pq}
\newcommand{\length}[1]{\ell(#1)}
\newcommand{\append}[1]{\mathbin{::}#1}
\renewcommand{\lozenge}{\Diamond}

\mathcode`>=\inteval{\mathcode`>-"1000}

\renewcommand{\gg}
{\mathbin{\vcenter{\hbox{\scalebox{0.5}{$\mathrel{\boldsymbol{<}}$}}}\mkern-11mu\mathrel{>}}}
\newcommand{\setcond}{\mathbin{\boldsymbol{>}}}

\newcommand{\omicron}{o}
\newcommand{\shift}[1]{s^*_{#1}}
\renewcommand{\complement}[1]{#1^{c}}
\newcommand{\prob}[2][\pi]{#1(#2)}
\newcommand{\cprob}[3][\pi]{#1(#2\mid #3)}

\title{Sequence and Consequence}

\author{Cian Dorr and Matt Mandelkern\thanks{Thanks to two anonymous referees for this journal for extensive and meticulous comments; to Melissa Fusco, Wesley Holliday,  Calum McNamara,  Nick Ramsey, Robert Stalnaker, and especially Snow Zhang, and to audiences at UConn, UC-Irvine, Stanford and Berkeley for discussion and invaluable and patient help.}}

\begin{document} 

\maketitle 

\begin{abstract}
In the course of proving a tenability result about the probabilities of conditionals, \citet{Fraassen:1976} introduced a semantics for conditionals based on $\omega$-sequences of worlds, which amounts to a particularly simple special case of ordering semantics for conditionals. On that semantics, `If p, then q' is true at an $\omega$-sequence just in case q is true at the first tail of the sequence where p is true (if such a tail exists). This approach has become increasingly popular in recent years. However, its logic has never been explored. We axiomatize the logic of $\omega$-sequence semantics, showing that it is the result of adding two new axioms to Stalnaker's logic \stalp: one, \textsf{Flattening}, which is prima facie attractive, and a second, \textsf{Sequentiality}, which is complex and difficult to assess, but, we argue, likely invalid.  But we show that when sequence semantics is generalized from $\omega$-sequences to arbitrary (transfinite) ordinal sequences, the result is a more attractive logic that adds only \textsf{Flattening} to \stalp. 
We identify two further interesting classes of order frame that are sound and complete for this logic: a class of \emph{flat} frames, which model-theoretically corresponds to it, and a class of \emph{linear} frames where the orderings induced by particular worlds are all restrictions of one grand ordering. 
We  explore the logics of a few interesting restrictions of ordinal sequence semantics. Finally, we address the question of whether sequence semantics is motivated by probabilistic considerations, answering, pace van Fraassen, in the negative.
\end{abstract}

\section{Introduction}

Stalnaker's \citeyearpar{Stalnaker:1968} `A theory of conditionals' launched the modern study of the conditional with a semantics for natural-language conditionals and a description of the corresponding logic \stalp, laying the groundwork for the rich subsequent literature on conditionals in philosophy, linguistics, and logic.
One subject of lively debate in this literature concerns the probabilities of conditionals. Our topic is the logic corresponding to an intriguing semantics for conditionals, based on $\omega$-sequences,  developed by \citet{Fraassen:1976} in the course of that debate.

\citet{Stalnaker:1970a} developed a theory of probability---treated as a property of sentences, and equated with ``degree of rational belief''---for a language with a binary conditional connective $>$ standing for `if$\dots$then' (on an indicative interpretation). The theory includes the following characteristic principle:
\begin{quote} 
    \emph{Stalnaker's Thesis}: $Pr(p > q)=Pr(q\mid p)$ (provided the right-hand-side is defined).\footnote{See \cite[p. 75]{Stalnaker:1970a}.  Like \citet{Popper:1959}, Stalnaker himself sets things up in such a way that the right-hand-side is always defined---e.g., $Pr(q\mid p\wedge\neg p) = 1$ for all $p$ and $q$.}
\end{quote}
There is robust empirical motivation for thinking that there is some important truth in the vicinity of Stalnaker's Thesis (see \citealt{DouvenVerbrugge} and many citations therein). For instance, if David is holding a fair die, the probability that if he rolls even, then he'll roll two is intuitively 1/3---equal to the probability that he rolls a two, conditional on rolling even.  

\citet{lewis:1976} showed that Stalnaker's Thesis was in tension with the assumption that we update our credences by conditionalization: in particular, no non-trivial class of probability functions closed under conditionalization satisfies Stalnaker's Thesis for a single interpretation of $>$. This left it open, however, that there is some way of co-ordinating different interpretations of the conditional with different interpretations of `probability'---e.g., whether we are talking about credences before or after some update---in such a way that Stalnaker's Thesis always holds \emph{within a single context}. However, a striking result in \citealt{StalnakerBas} showed that even this is impossible given the background logic \stal, except in certain trivial cases.

At the same time, however, \citet{Fraassen:1976} showed that models of \stal \emph{can} be equipped with non-trivial probability functions satisfying a restricted form of Stalnaker's Thesis.\footnote{Publication dates in this literature are confusing. To our knowledge, Lewis's was the first triviality result. Stalnaker's 1974 letter \citep{StalnakerBas} was a response to a draft of van Fraassen's paper, which, in turn, was a response (in part) to Lewis's result. Van Fraassen's published 1976 paper  appears to leave open whether his construction validates Stalnaker's Thesis for the whole language, a possibility which Stalnaker's letter  rules out; our understanding is that the published version was in fact the one Stalnaker's letter was responding to, despite the later publication date. Thanks to Bas van Fraassen for correspondence about this.} 
In van Fraassen's models, the probability function on non-conditional sentences can be freely specified, and Stalnaker's Thesis holds for conditionals whose antecedents do not themselves contain conditionals.%
\footnote{It also holds for some conditionals whose antecedents \emph{do} contain conditionals: see \cref{probsagain} for details.}

Ours is not, however, a paper about the probabilities of conditionals (a topic we return to only briefly, in \cref{probsagain}), but rather about the construction which van Fraassen developed in the course of modeling the restricted version of Stalnaker's Thesis. In that construction, van Fraassen used a semantics for conditionals with the following form. Start with a set $W$ of ``worlds'' and a valuation that specifies which atomic sentences are true at elements of $W$. Now consider the set of $\omega$-\emph{sequences} over $W$: that is, functions from the natural numbers to $W$. These sequences will serve as indices in a model for a language containing the conditional connective $>$.  In this model, an atom is true at a sequence $\sigma=\seq{w_0,w_1, w_2, \dots}$ just in case it is true at $w_0$ according to the old valuation.  The clauses for negation and conjunction are classical. Finally, a conditional $p>q$ is true at a sequence $\sigma$ just in case either $\sigma$ has a tail at which $p$ is true, and $q$ is true at the first such tail, or $\sigma$ has no tail at which $p$ is true (the \emph{tails} of $\seq{w_0,w_1, w_2, w_3, w_4 \dots}$ are $\seq{w_0,w_1,w_2,w_3,w_4\dots}, \seq{w_1,w_2,w_3,w_4\dots}, \seq{w_2,w_3,w_4,\dots}$, and so on).

Sequence semantics has become increasingly popular in recent years.\footnote{See e.g. 
\cite{Kaufmann:2009b,Kaufmann:2015a, bacon:2015,Schultheis:2020,Santorio:2021,Santorio:2021a,KhooBook}.} But, surprisingly, some basic questions about the semantics have never been answered, including what its logic is. The goal of this paper is to axiomatize the logic of $\omega$-sequence models, as well as some interesting generalizations and restrictions of that approach that base the same semantics on different classes of ordinal sequences (that is, functions from arbitrary ordinals, possibly larger or smaller than $\omega$, to an underlying set). Along the way, we will explore some other, closely related semantic systems.

We have a few motivations for this project. One is its intrinsic interest: $\omega$-sequence semantics is an intriguing, and in some ways very simple, semantics for conditionals. So we should understand it, and part of understanding the semantics is knowing the logic it gives rise to. 
Besides being of intrinsic interest, this will help us assess the viability of sequence semantics for modeling conditionals in natural language.\footnote{See \citealt{HollidayIcard2018} on the methodological importance of axiomatization in semantics.}

 Another motivation comes from particularities of the logic that arises from sequence semantics. As we will show, sequence semantics can be viewed as a special case of Stalnaker's semantics for conditionals---a special case which, it turns out, strictly strengthens Stalnaker's logic \stalp. This is of special interest to both authors, who believe that all the principles of \stal are plausible as far as natural language conditionals go.
 
This is a controversial position. The commitment of \stal to the validity of  Conditional Excluded Middle (CEM) has historically been rather unpopular, due to influential criticism by Lewis.  In fact, essentially every commitment of \stal has been rejected somewhere in the subsequent literature. However, our commitment to the correctness of a logic at least as strong as \stal makes us particularly interested in strengthenings of \stalp. (We will not do anything here to defend \stalp, but see \citealt{DorrBook} for extensive discussion.)
 
To our knowledge, however, no logics which are stronger than \stal but weaker than Materialism have ever been explored. Materialism is the logic which collapses the natural language conditional $p>q$ to the material conditional $p\to q$, that is, the logic which simply adds to classical logic the principle $(p>q)\leftrightarrow (p\to q)$.  There exist powerful arguments against this equivalence \citep{Edgington:1995}. 
Famously, however, \citet{Dale:1974,Dale:1979,Gibbard:1981,McGee:1985} showed that the gap between  \stal and Materialism is surprisingly small: in particular, it is fully closed by the \textsf{Import-Export} principle (which we discuss in \cref{evalflat}). To our knowledge, no logics residing in the gap between \stal and Materialism have ever been studied, perhaps because of these famous results.
But it will turn out that the logic of $\omega$-sequences is strictly intermediate between \stal and Materialism. In fact, in the course of exposition, we will explore two such logics. We will show that the logic of $\omega$-sequences is the logic we call \vanf, comprising  \stal plus every instance of the following two axiom schemas.
\begin{gather*}
    \ourtag{Flattening}
    (p > ((p\wedge q) > r )) \leftrightarrow ((p\wedge q) > r)
\\
    \ourtag{Sequentiality}
    \begin{multlined}[t][0.7\textwidth]
    \Box(p \to (\neg p > r)) \wedge \Box(q \to (\neg q > r)) \\\to ((p\vee q) \to (\neg(p \vee q) > r))
    \end{multlined}
\end{gather*}
where $\Box p$ is defined as $\neg p > p$.  
We will argue that Flattening is at least prima facie appealing for conditionals in natural language, while Sequentiality is, at best, too complex to reasonably assess, and, at worst, invalid. This suggests that $\omega$-sequence semantics is not a strong contender for a logic of the natural language conditional. But  we will show that the logic of a semantics based on  \emph{ordinal sequences} in general, rather than just $\omega$-sequences, is the more attractive logic \vflat comprising \stal together with just Flattening.
We show that this same logic is also sound and complete for two different classes of order frames: first, a class we call the \emph{flat} frames, which corresponds to \vflat in the frame-theoretic sense; second, a class we call  \emph{linear} order frames, where the world-sensitive ordering induced at any given world $w$ can be derived from one grand ordering on all worlds, by   finding $w$ in the grand ordering and then restricting it to the worlds that follow $w$ in that ordering.
We also explore the logics of a few other interesting restrictions of ordinal sequence semantics, and, finally, argue that, \emph{pace} van Fraassen, sequence semantics cannot be motivated by indirect considerations about the probabilities of conditionals.

\section{\stal and its semantics}
We will begin with some important background, reviewing Stalnaker's \citeyearpar{Stalnaker:1968} conditional logic \stal, and a class of models corresponding to that logic. (Cognoscenti may wish to skip to the next section.)

The language of \stal and all the logics we will be considering is a standard propositional language $\mathcal{L}_>$ equipped with a binary conditional connective $>$.  So, where $At = \{p_0, p_1,\ldots\}$ is a countably infinite set of atomic sentences:\footnote{\citet{StalnakerThomasonSACL} extend \stal to a language with quantifiers, but here we are concerned only with the propositional fragment from \citealt{Stalnaker:1968}.} 

\[  
    p   ::= p_k\in At \mid \neg p \mid (p \wedge p) \mid( p>p)
\]
We use $\to$, $\leftrightarrow$, and $\vee$ as abbreviations for the material conditional, material biconditional, and disjunction defined as usual.
For brevity, we sometimes use a compact notation for negation and conjunction applied to atoms and metavariables: $\negate p$ for $\neg p$ and $pq$ for $(p\wedge q)$ (and $\neg pq$ is $\neg(p\wedge q)$). We also sometimes omit parentheses: where we do, the order of operations is negation, then $>$, then $\wedge$ and $\vee$, and finally $\to$ and $\leftrightarrow$, so for instance $p\cond q \to \neg r\cond s\wedge t$ is to be read as $(p\cond q) \to ((\neg r\cond s)\wedge t)$.%
\footnote{The defined unary operators $\Box$ and $\lozenge$ which we introduce below also take highest priority in the same way as $\neg$.} 

\stal is the closure of the following set of axiom schemas:\footnote{Reciprocity is often called CSO, but the source of that name is lost (to us, at least), so we will use the more mnemonic name.}
\begin{align*} 
    \ourtag{PC} &\text{Every theorem of classical propositional logic} \\
    \ourtag{Identity} &p>p \\
    \ourtag{Reciprocity} 
    &(p\cond q)\wedge (q\cond p)\wedge (p\cond r)\rightarrow q\cond r \\
    \ourtag{MP} &p > q\rightarrow (p\rightarrow q) \\
    \ourtag{CEM} &p\cond q\vee p\cond\neg q \\
\intertext{under the following two inference rules:}
    \ourtag{Detachment} &\vdash p\rightarrow q\text{ and }\vdash p\text{ together imply }\vdash q \\
    \ourtag{Normality} &\vdash(p \wedge q)\rightarrow r\text{ implies }\vdash(( s\cond  p) \wedge (s\cond q))\rightarrow s\cond r
\end{align*}
When $p$ is a theorem of \stal we write $\vdash_{\stalp}p$. 

Stalnaker's own axiomatization of \stal is somewhat different, and uses two further abbreviations, which will be useful in what follows and hence worth noting here: 
\begin{itemize}
    \item[] $\Box$, defined by $\Box p :=  \neg p\cond p$; and
\item[]  $\lozenge$, defined by $\lozenge p :=  \neg\Box\neg p$.%
\footnote{$\lozenge p$ could equivalently be defined as $\neg(p\cond \neg p)$.}
\end{itemize}
Stalnaker then defines \stal with the axioms PC, MP, and Reciprocity plus four further axioms (the first of which is just the familiar \textsf{K} axiom for $\Box$):
\begin{align*} 
    &\Box(p\to q)\to (\Box p\to \Box q) \\
    &\Box(p\to q)\to p\cond q \\
    &\lozenge p\to (p\cond q\to \neg (p\cond\neg q)) \\
    & p\cond (q\vee r)\to (p\cond q\vee p\cond r)
\end{align*}
and closing the result under Detachment together with the rule:
\begin{align*} 
    \ourtag{Necessitation} &\vdash p \text{ implies }\vdash \Box p
\end{align*}
It is a good exercise to show that these two axiomatizations of \stal  are indeed equivalent. Deriving our axiomatization from Stalnaker's is easy, using the definition of $\Box$ via $>$. For the other direction, the key principle to derive is:
\[
    \ourtag{MOD} \Box p\to q\cond p
\]

Stalnaker's axiomatization brings out the fact that there is a sense in which \stal \emph{contains} (though is not reducible to) a  modal logic. The modal logic \textsf{KT} is the logic containing every instance of the PC and K schemas as above, along with the further axiom schema:
\[
    \ourtag{T} \Box p\to p
\]
closed under the rules of Necessitation and Detachment. We will show below that the theorems of \stal expressible using atoms, Boolean connectives, and $\Box$ (that is, the theorems in \emph{the modal fragment} of $\mathcal{L}_>$) are exactly the theorems of \textsf{KT}. 

While $\Box$ in the formal language is simply a shorthand, one might also think there are indeed close connections between necessity and conditionals, in particular between, on the one hand, epistemic modals and indicative conditionals; and, on the other, circumstantial modals and subjunctive conditionals.%
\footnote{See \citealt{DorrBook} for one defense of such a position.} 
Such connections would make the modal logics of various conditional logics especially interesting. Even in the absence of such connections, however, the modal fragments of our conditional logics are well worth studying on a purely logical basis.

\subsection{Order models for \stal}

There are many model-theoretic semantics for conditional connectives in the literature, generalizing  Kripke's possible-worlds semantics for modal logic to conditional languages. Our focus in this paper will be on \emph{order models} for conditionals, which were introduced by \citet{Lewis:1973} as models for his logic (which is strictly weaker than \stal). Order models turn out to be particularly intuitive for the study of \stal and its strengthenings.
In particular, we can model \stal with order models where each world is associated with a \emph{well-ordering} of worlds. As we will see,  sequence models can be  naturally viewed as a special case of well-order models, making it easy to see that their logic includes \stal.

A Kripke model equips a set of possible worlds with a binary accessibility relation $R$, representing relative necessity and possibility: $p$ is necessary at $w$ just in case $p$ is true throughout the worlds accessible from $w$, which we write $R(w)$, and $p$ is possible at $w$ just in case $p$ is true somewhere in $R(w)$. An order model is like a Kripke model but with additional structure: in addition to a set of worlds $R(w)$, an order model associates each world $w$ with an \emph{ordering} $<_w$ of $R(w)$.  We pronounce $u <_w v$ as `$u$ is closer to $w$ than $v$'.  In such a model, assuming there are some \emph{closest} $p$-worlds to $w$, $p>q$ is true at $w$ just in case all of these worlds are $q$-worlds. $p>q$ is also (`vacuously') true at $w$ when there aren't any $p$-worlds in $R(w)$.\footnote{In Lewis's generalization of order models, we also need to say something about the case where $R(w)$ contains some $p$-worlds, but for each one of them, there is another that is closer.  However this case will conveniently not arise in the models we deal with.}

Lewis conceived of closeness in terms of \emph{similarity}: $x<_w y$ means that $x$ is more similar to $w$ than $y$ is (in whatever respects turn out to be relevant).  But using order models does not commit us to a similarity-based interpretation of the order functions, any more than Kripke semantics for the modal operators commits us to any particular theory of necessity and possibility.  Thus, skeptics of similarity-based approaches to conditionals (a group in which we include ourselves) have no special reason to object to the use of order models.  Indeed, in modelling strengthenings of \stalp, we will need to impose conditions on order models which would be completely implausible if closeness had to be interpreted as similarity, so insofar as the strengthenings are well-motivated, they will add to the already strong case against similarity-based approaches (see \cref{evalflat}).

In the order models characteristic of \stal, each ordering $<_w$ is a \emph{well-order}: that is, a transitive, connected, asymmetric, well-founded relation on $R(w)$. This means that we can restate the order semantics for conditionals in terms of the \emph{unique} closest antecedent world, if there is one: $p>q$ is true at $w$ just in case $q$ is true at the first $p$-world in $<_w$, or there are no $p$-worlds in $R(w)$. This uniqueness assumption guarantees that the controversial CEM axiom holds in the model; logics without CEM (like Lewis's) can be obtained by relaxing this assumption. However, we will not further discuss such models in this paper, so by `order model' we will always mean a well-order model.\footnote{\label{fn:selection}This order semantics for conditionals is equivalent to the semantics based on \emph{selection functions} given in \citealt{Stalnaker:1968}. A selection function is a function $f$ which takes a set of worlds $\varphi$ and world $w$ and returns a set $f(\varphi,w)$.  We can use a selection function to evaluate conditionals, by defining $\sem{p>q}$ to be $\{w:f(\sem{p},w)\subseteq\sem{q}\}$.  
In the case of interest for \stal, $f$ is required to obey constraints corresponding to MP, Reciprocity, Identity, and CEM:
\begin{enumerate}
    \item If $w\in \varphi$, $w\in f(\varphi,w)$.
    \item If $f(\varphi,w) \subseteq \psi$ and $f(\psi,w) \subseteq \varphi$, $f(\varphi,w) = f(\psi,w)$.  
    \item $f(\varphi,w) \subseteq \varphi$. 
    \item $f(\varphi,w)$ has at most one element.
\end{enumerate}
Given these constraints, we can move freely between selection functions and order functions.  Given a selection function $f$, we define an isomorphic order function $<$ by saying that $x<_w y$ just in case $f(\{x,y\},w)=\{x\}$ and $f(\{y\},w)=\{y\}$. 
Conversely, given an order function $<$, we define an isomorphic selection function $f$ by saying that $f(\varphi,w)$ is $\{w\}$ when $w\in \varphi$; otherwise the singleton of the first $\varphi$-world ordered by $<_w$, if there is one; and otherwise the empty set. (Reflection on the respective properties of order and selection functions show that these defined constructions are, indeed, order and selection functions, respectively.) 
So there is no deep difference between these two kinds of model. However, order models lend themselves more naturally to the study of sequence semantics, as we shall see.}

Let us lay this all out more formally, and introduce some standard terminology which will be helpful:
\begin{definition}
    An \emph{order frame} is a pair $\seq{W,<}$, where $W$ is a non-empty set and $<$ is a  function which takes any $w \in W$ to a strict well-order $<_w$ on some subset of $W$ such that whenever $x <_w y$, $w=x$ or $w<_wx$.
\end{definition}
\noindent We can read off an accessibility relation from an order frame: the worlds accessible from $w$ are those that $w$ strictly precedes in the ordering induced at $w$, together with $w$ itself:\footnote{Accessibility relations are sometimes specified as independent parameters, but as this shows, they needn't be.} 
\[R(w)=\{w\}\cup\{v: w<_w v\}\] As we will see, $R$ plays the same role with respect to the defined $\Box$ as accessibility relations usually do in Kripke models. 
When it is clear which order frame is under discussion, we will always use $R$ for the corresponding accessibility relation. 

\begin{definition} An \emph{order model} is a triple $\seq{W,<, V}$ where $\seq{W,<}$ is an order frame and $V: At\times W \to \{0,1\}$ is a valuation function. 
\end{definition}
We will write $V(p_k)$ for $\{w:V(p_k,w) = 1\}$ and $V(w)$ for $\{p_k:V(p_k,w) = 1\}$.  
\begin{definition}
    When $\seq{W,<,V}$ is an order model, its denotation function is the function $\sem{\seq{W,<,V}}{\cdot} : \mathcal{L}_>\rightarrow\mathcal{P}(W)$ such that for any atom $p_k$ and sentences $p,q$:
    \begin{align*}
    \sem{p_k} &= V(p_k) \\
    \sem{\neg p} &= W\setminus \sem{p} \\
    \sem{p\wedge q} &= \sem{p} \cap \sem{ q } \\
    \sem{p>q} &= \parbox[t]{0.65\textwidth}{$\{w\in W:$ either $R(w)\cap\sem{p}=\varnothing$, or for some $y\in R(w)\cap \sem{p} \cap \sem{q}$, $x\notin \sem{p}$ whenever $x<_w y\}$}
    \end{align*}
\end{definition}
For readability, relativization of \sem{\cdot} to a model is usually left implicit. 

As usual, we can define a \emph{pointed} order model as a pair of an order model with a world from its set of worlds, i.e. a pair $\seq{w,\seq{W,<,V}}$ such that $w\in W$ and $\seq{W,<,V}$ is an order model; when $\seq{w,\seq{W,<,V}}$ is a pointed order model, we say that it is \emph{based on} $\seq{W,<,V}$.
$p$ is \emph{true in} a pointed order model $\seq{w,\mathcal{M}}$  just in case $w\in \sem{p}^\mathcal{M}$. 
When $\Gamma\subseteq\mathcal{L}_>$, we can also speak of $\Gamma$ being true at a pointed model to mean that all its elements are. 
We also write $w,\mathcal{M}\Vdash p$ when $p$ is true in $\seq{w,\mathcal{M}}$; when $\mathcal{M}$ is implicit from the context, we write simply $w\Vdash p$. For brevity we sometimes talk about $p$ being true at every model in a given class; by this we mean true at every pointed model based on a model in that class. 
Two pointed models are \emph{equivalent}
just in case they verify exactly the same sentences of $\mathcal{L}_>$. 

A standard induction on formulae shows:
\begin{theorem} \label{C2oundness}
    \stal is sound for order models: that is, $\vdash_{\stal} p$ implies that $p$ is true in every pointed order model.
\end{theorem}
We also have a corresponding completeness result: every sentence that is true in every pointed order model is a theorem of \stal.  Equivalently, whenever $p$ is \emph{consistent} in \stal (that is, $\nvdash_\stalp \neg p$) it is true in some pointed order model.  In fact, we can show something stronger: \stal is  complete with respect to the class of \emph{finite} order models:

\begin{restatable}[]{theorem}{golcompleteness}
\label{thm:golcompleteness}
  If $ p $ is true in every  finite pointed order model,  then $\vdash_{\stalp} p $.
\end{restatable}

One corollary is that \stal is \emph{decidable}.  Since every non-theorem is false in some finite pointed order model, and we can effectively enumerate all the finite pointed order models (up to isomorphism), we can test for non-theoremhood by searching through the finite pointed order models for a countermodel; this provides an effective decision-procedure when run in parallel with a proof search.  

This soundness and completeness theorem also allows us to prove our earlier assertion about the modal fragment of \stalp:
\begin{theorem} 
    When $ p $ is a sentence in the modal fragment of $\mathcal{L}_>$, $\vdash_{\stal} p $ iff $\vdash_{\textsf{KT}} p $.  
\end{theorem}
\begin{proof}
We rely on the well-known fact that \textsf{KT} is sound and complete with respect to modal modals with a reflexive accessibility relation.   

\bi\item[$\Leftarrow$]  Normality for $>$ gives the K axiom for $\Box$, and MP for $>$ gives T for $\Box$.  

\item[$\Rightarrow$] Suppose we have a modal model $W$ with a reflexive accessibility relation $R$; we can extend this to an order model that respects $R$ by fixing a strict well-ordering $<$ on $W$, and for $x\neq y$, let $x <_w y$ iff $wRx$ and $wRy$ and either $x=w$, or $x\neq w$ and $y\neq w$ and $x<y$.  If  $\vdash_{\stal} p $ then by soundness $p$ is true in every order model, and hence in every reflexive modal model, and so by completeness of \textsf{KT},  $\vdash_{\textsf{KT}} p $.

\qedhere

\ei
\end{proof}

\subsection{The failure of strong completeness\label{notecomp}}
\cref{thm:golcompleteness} is about individual sentences: it is equivalent to the claim that whenever $p$ is \emph{consistent} in \stalp (meaning that $\neg p$ is not a theorem of \stalp), $p$ is true in some pointed order model.  This kind of result is sometimes called a \emph{weak} completeness theorem.  By contrast, a \emph{strong} completeness theorem would say that for every \emph{set} of sentences that is consistent in a certain logic, there is a model based on a frame in the relevant class in which \emph{every} sentence in that set is true. 
We define consistency for sets of sentences as usual: $\Gamma$ is consistent relative to $\vdash$ when there is no finite conjunction $p$ of elements of $\Gamma$ such that $\vdash p\to \bot$, where $\bot$ abbreviates $p_0\wedge \neg p_0$.
Somewhat surprisingly, we do not have a strong completeness theorem analogous to \cref{thm:golcompleteness}.\footnote{This point is noted very briefly by
\citet[Footnote 12]{Fine:2012}, and developed in much more detail in
\citealt{Kocurek:2025}. Thanks to Arc for drawing our attention to his paper.}  
\begin{theorem} 
    There are \stalp-consistent sets $\Gamma\subseteq \mathcal{L}_>$ which are not true in any pointed order model.\label{notstrong}
\end{theorem}
\begin{proof}
   One such $\Gamma$ is the following:
    \[
        \Gamma = \{\neg((p_i\vee p_{i+1})>p_{i}) \mid i\in\mathbb{N}\}
    \]
    Suppose all the members of $\Gamma$ were true in some pointed order model $\seq{w,\seq{W,<,V}}$.  Consider the set of worlds $\varphi$ which verify some atom in this model, i.e., $\varphi=\bigcup_i V(p_i)$. We must have $\varphi\cap R(w)\neq \varnothing$, for otherwise the elements of $\Gamma$ would all be false at $w$ (e.g., if $p_1$ and $p_2$ are nowhere true in $R(w)$, then $(p_1\vee p_2)>p_1$ is vacuously true at $w$). Since $<_w$ is a well-order, $\varphi$ must have a least element $x$ in $<_w$. Some atom $p_k$ is true at $x$, by definition of $\varphi$. Now consider $(p_k\vee p_{k+1})>p_k$. This is true at $w$, since the first world in $<_w$ where $p_k\vee p_{k+1}$ is true must be $x$ (since $x$ is the first world in $<_w$ where \emph{any} atom is true), where $p_k$ is true. Hence its negation is false at $w$, contrary to the assumption that $w$ verifies all the elements of $\Gamma$.

    Nevertheless, $\Gamma$ is consistent in \stalp. If it were not, then by definition, it would have some  inconsistent finite subset. But every non-empty finite subset $\Delta\subset \Gamma$ is true in a pointed order model, which given soundness (\cref{C2oundness}) shows that every finite subset of $\Gamma$ is consistent. To see this, let $p_k$ be the atom with the highest index of any atom which appears in $\Delta$. Consider any set $U$ with $k+1$ members, which we label $w_0, w_{1}, \dots w_k$. 
    $\Delta$ is true at the pointed order model $\seq{w_k,\seq{U,<,V}}$, where $<$ is any order function with $w_n<_{w_k}w_{m}$ just in case $n > m$, and $V$ any valuation such that $V(p_i)=\{w_i\}$ for $i\leq k$. 
\end{proof}
The same reasoning shows that no extension of \stal in which $\Gamma$ remains consistent is strongly complete for any class of order models.

It is possible, however, to formulate a notion of a ``general'' order frame, and hence order model, relative to which we do have strong completeness (cf.\ \citealt{Segerberg:1989}). 
The idea is to add to our frames an extra ``propositional domain'' parameter---a set of subsets of worlds, representing the allowable denotations for sentences---and only require that our orders are well-founded  relative to the elements of that parameter, rather than relative to all subsets of the set of worlds.

In more detail, let a generalized order frame be a triple $\seq{W,\mathcal{B},<}$, where $W$ is any non-empty set; $\mathcal{B}$ is a set of subsets of $W$; and $<$ is a function which takes any $w\in W$ to a  linear order $<_w$ on a subset of $W$, such that whenever $<_w$ orders any element of $\varphi\in \mathcal{B}$, $\varphi$ has a  first element in $<_w$, and where $<$ and $\mathcal{B}$ are coordinated so that $\mathcal{B}$ is closed under the set-theoretic operations corresponding to $\wedge, \neg$, and $>$ in order semantics (relative to $<$). A generalized order model is a generalized order frame $\seq{W,\mathcal{B},<}$ equipped with a valuation $V:At\to \mathcal{B}$.  The definition of $\sem{\cdot}$ that worked for order models still works in generalized models, and yields a function from $\mathcal{L}_>$ to $\mathcal{B}$.  \stal is sound and \emph{strongly} complete with respect to generalized order models.  For example, the consistent set $\Gamma$, which is not true in any pointed order model, can be true in a generalized pointed order model where the set $\bigcup_i V(p_i)$ of worlds that verify some atom is not in the propositional domain, and does not have a minimal element in $<_w$.%
\footnote{Consider the generalized pointed order model where $W$ is the set $[0,\infty)$ of nonnegative reals, with designated world $0$; $V(p_i)=[1/(i+1),\infty)$; $x<_w y$ iff $w\leq x < y$; and $\mathcal{B}$ the set of finite unions of intervals of the form $[x,y)$ or $[x,\infty)$.}  
Completeness can be shown with a standard canonical model construction, as in \cite{Lewis:1971}.%
\footnote{Instead of a primitive propositional domain, Lewis get the same result by treating the denotations of sentences as primitive rather than defining them inductively in the customary way, so that the set of all denotations of sentences plays the role of the propositional domain.  Note too that while Lewis only states weak completeness theorems, his method immediately establishes strong completeness.}

A generalized order frame is \emph{full} if $\mathcal{B}=\wp(W)$; full generalized frames are equivalent to order frames.%
\footnote{The selection function models described in \cref{fn:selection} can be ``generalized'' in an analogous way to order models.  In a generalized order model, the selection function $f$ is only defined on pairs $w$, $\varphi$ where $\varphi$ belongs to the propositional domain; as with order models, we also require the propositional domain to be closed under all the operations on sets corresponding to the semantic clauses. 

In both generalized order models and generalized selection models, elements of the propositional domain that happen not to be denoted by any sentence in $\mathcal{L}_>$ are logically irrelevant: restricting the propositional domain of a model to the sets that are in fact denoted by sentences of $\mathcal{L}_>$ (in the original model) will not change the truth value of any sentence at any world.  This is worth noting because it brings us back to the kind of models developed by Stalnaker and Thomason \citep{Stalnaker:1968,StalnakerThomasonSACL} in which the selection functions are defined not on pairs of worlds and sets of worlds, but for pairs of worlds and sentences. Given the constraints they place on such selection functions, such models are equivalent to generalized order models as we have defined them here.}

Presumably because of facts along these lines, \citet{Segerberg:1989} writes that 
`in modal logic it seems quite natural to restrict one's interest---at least
initially---to full frames. In conditional logic, studied in the present vein, this is
not so.' 
Nevertheless, our interest in this paper will be  full order frames, since our main goal is to identify the logics of various kinds of sequence frames, which (as we will shortly see) can be viewed as special cases of (full) order frames.  Hence the completeness results we prove will be weak rather than strong completeness results.

\section{Omega-sequence semantics}

With this set-up in hand, 
we are now in a  position to give a more rigorous presentation of van Fraassen's \parencite*{Fraassen:1976} $\omega$-sequence models. We will present these models in a way which brings out the fact that these are in fact a special case of order models. This will let us show immediately that the logic of $\omega$-sequence models is at least as strong as \stal.

First, we will introduce some general terminology for talking about sequences.  Although for van Fraassen's models the sequences of interest are $\omega$-sequences, which can be understood as functions from the natural numbers to an underlying set, for the sake of later generalizations we will consider these as a special case of ``ordinal sequences'', whose domain can be any arbitrary ordinal.  
\begin{definition}  
    Given a non-empty set $P$ and an ordinal $\alpha$, an \emph{$\alpha$-sequence} over $P$ is a function $\sigma:\alpha\rightarrow P$. A function is an \emph{ordinal sequence} just in case it is an $\alpha$-sequence for some ordinal $\alpha$. 

    When $\sigma$ is an $\alpha$-sequence and $\beta<\alpha$, we write $\sigma\elem{\beta}$ for the value of $\sigma$ at $\beta$, that is, $\sigma(\beta)$. When $\beta\leq \alpha$, we write $\sigma\slice{\beta:}$ for the  \emph{$\beta$th tail} of $\sigma$, i.e., the length $\alpha-\beta$ sequence $\seq{\sigma\elem{\beta},\sigma\elem{\beta+1},\sigma\elem{\beta+2},\ldots}$.
    
    We write $\tails{\sigma}$ for the set of all nonempty tails of $\sigma$, i.e.\ $\{\sigma\slice{\beta:}:\beta<\alpha\}$.  When $\tau \in \tails{\sigma}$, the \emph{rank} of $\tau$ in $\sigma$ is the least $\beta$ such that $\tau = \sigma\slice{\beta:}$. 
\end{definition} 
Any set of ordinal sequences can be endowed with an order function in a natural way:
\begin{definition}\label{seqdef}
 The \emph{tail order function} $\prec^S$ on any set $S$ of sequences has $\tau\prec^S_\sigma\rho$ iff  $\tau,\rho,\sigma\in S$, $\tau$ and $\rho$ are tails of $\sigma$, and the rank of $\tau$ in $\sigma$ is less than the rank of $\rho$ in $\sigma$.
\end{definition}

\begin{definition} 
    An order frame $\seq{W,<}$ is an \emph{$\omega$-sequence frame} iff $W$ is a set of $\omega$-sequences on some underlying set $P$ (which we call the ``protoworlds''); $W$ is \emph{closed under tailhood} (that is, if $\sigma \in W$ then $\tails{\sigma} \subseteq W$); and $<$ is the tail order function $\prec^W$ as in \cref{seqdef}.  A [pointed] \emph{$\omega$-sequence model} is a [pointed] order model based on an $\omega$-sequence frame.  
\end{definition}
\noindent For brevity, we write $\seq{\sigma, W,V}$ for the pointed $\omega$-sequence model $\seq{\sigma, \seq{W,\prec^W,V}}$, since the tail order function is determined by $W$. For even more brevity we sometimes will simply specify a sequence and a valuation $\seq{\sigma, V}$; in that case, $W$ is (implicitly) the set of all and only $\sigma$'s non-empty tails.

It should be clear how this notion of model, together with the semantics given above for arbitrary order models, is equivalent to the (more standard) presentation of van Fraassen's models we gave in the introduction: on this semantics, $p>q$ is true at a sequence $\sigma$ just in case either $\sigma$ has a tail at which $p$ is true and $q$ is true at the first such tail, or else $\sigma$ has no tails at which $p$ is true.  

Van Fraassen's models have two further particular properties: they are \emph{full} and \emph{categorical}:
\begin{definition}
    An $\omega$-sequence frame $\seq{W,\prec^W}$ is \emph{full} when $W$ is the set of \emph{all and only} $\omega$-sequences over some non-empty set $P$.  
\end{definition}
\begin{definition}
    An $\omega$-sequence model $\seq{W,\prec^W,V}$ is \emph{categorical} when for any sequences $\sigma$ and $\rho$ with the same first element, $V(\sigma) = V(\rho)$. 
\end{definition}
As we will see, however, these two restrictions on models are logically immaterial.

Van Fraassen called the members of the underlying set $P$ `worlds'.  We call them `protoworlds' to avoid confusion---after all, it is not elements of $P$, but sequences over $P$, that play the standard model-theoretic role of worlds in assigning truth values to sentences.  The choice to call them ``worlds'' might go along with a metaphysically ambitious take on the significance of the models, on which the contrast between subsets of $W$ that do not divide sequences with the same first element and those that do
is taken to model a non-model-relative contrast between ``factual''/``objective''/``heavyweight'' questions on the one hand and ``non-factual''/``subjective''/``lightweight'' questions on the other, with categorical (conditional-free) propositions depending only on the answers to the former kind of question, and conditionals typically depending on the latter kind of questions.  The former questions might be supposed to be settled by \emph{how things are}, whereas the latter are in some sense mere expressions of the way we think, or artifacts of the way we talk (see \citealt{KhooBook} for a recent picture along these lines).  A proponent of this metaphysical distinction might think of ``worlds'' as things that merely answer all \emph{factual} questions; in that case, `world' will seem a good name for elements of $P$, since, given van Fraassen's assumptions, a single element of $P$ is enough to determine the truth-value of any conditional-free sentence.  But we will set aside these questions of metaphysical interpretation here, since they are irrelevant to our logical concerns. Even if you hold the contrast between  categorical and  conditional to be metaphysically chimerical, you could still  accept the logic of sequence models; conversely, you could reject that logic while still maintaining that there is an important metaphysical distinction between the categorical and conditional.

\subsection{Some variations on omega-sequence models}
\label{incon}
In the previous section we introduced both $\omega$-sequence models and the more specific class of full, categorical models.  \emph{Prima facie}, since not every $\omega$-sequence model is a full and categorical model, one might expect that some sentences that hold in all full and categorical models do not hold in all $\omega$-sequence models.  But it turns out that this is not the case: the restriction to full and categorical models makes no difference to the logic.  To see why this is the case, the following fact will be useful:
\begin{fact} \label{canon} 
    Pointed $\omega$-sequence models $\seq{\sigma,W,V}$ and $\seq{\tau,W',V'}$ are equivalent      whenever $V(\sigma\slice{k:}) = V'(\tau\slice{k:})$ for all $k\in \mathbb{N}$.
\end{fact}
The proof is a routine induction on the complexity of formulae. 
The intuition is that all that matters in assessing the truth of a sentence at the distinguished sequence in a pointed $\omega$-sequence model is how the tails of that sequence are valued. Just as the actual identity of worlds doesn't matter in Kripke semantics, likewise the actual identity of sequences doesn't matter in sequence semantics.


As an immediate consequence of \cref{canon}, we have:
\begin{fact} \label{canon1}
    Any pointed $\omega$-sequence model $\mathcal{M}=\seq{\sigma,W,V}$ is equivalent to the pointed $\omega$-sequence model $\mathcal{M}_\omega=\seq{\vec{\omega},\tails{\vec{\omega}}, V'}$ where $\vec{\omega}$ is the sequence $\seq{0,1,2,3\dots}$ of the natural numbers in their standard order (i.e.\ $\operatorname{id}_{\mathbb{N}}$), and $V'(\vec{\omega}^{[k:]}) = V(\sigma^{[k:]})$.
\end{fact}
\noindent We can think of $\mathcal{M}_\omega$ as a kind of minimal representation of $\mathcal{M}$.  Its underlying frame is obviously isomorphic to the order frame with domain $\mathbb{N}$ in which $j<_ik$ iff $i\leq j < k$.  Thus the logic of that order frame, the logic of $\seq{\tails{\vec{\omega}}, \prec}$, and the logic of the class of all $\omega$-sequence frames are all the same.

The minimal representation $\mathcal{M}_\omega$ is always categorical, since no two sequences in its domain share an initial element.  \Cref{canon1} thus shows that requiring categoricity makes no difference to the logic.  

The underlying frame of the minimal representation $\mathcal{M}_\omega$ also has another noteworthy property that is worth giving a name to:
\begin{definition} \label{lineardef}
   An order frame $\seq{W,<}$ (or model $\seq{W,<,V}$) is \emph{linear} iff there is a strict partial order $\prec$ on $W$ such that for all $w,x,y: x<_w y$ iff $w\preceq x \prec y$.
\end{definition}
Given the definition of order frames, any such $\prec$ will also be a strict well-order on any principal upper set of $W$ (any set of elements $\succeq w$ for any $w\in W$).

$\omega$-sequence models need not be linear: consider, for example, the $\omega$-sequence model with domain $\{\seq{0,1,0,1,\ldots},\seq{1,0,1,0,\ldots}\}$.  But \cref{canon1} implies that the logic of $\omega$-sequence frames is at least as strong as the logic of linear order frames, since $\mathcal{M}_\omega$ is linear.  We will see later that this inclusion is proper.%
\footnote{Linear order models are strikingly similar to standard models of temporal logic, a connection to which a reviewer has helpfully drawn our attention.
In particular, \citet{Kamp:1968} considers a language with a binary operator $U$, with $pUq$ standing intuitively for `$p$ will be true until $q$ is true'  (see \citealt{Goranko:2025} for a helpful overview).
The standard semantics for $U$ uses frames $\seq{W,\prec}$ where $\prec$ is a transitive, irreflexive relation on $W$, with $t\Vdash p Uq$ iff there exists $s \succ t$ such that $s\Vdash q $ and $u\Vdash p$ whenever $t\prec u\prec s$. 

This semantics naturally suggests a definition of $>$ in terms of $U$: 
\begin{equation*}
    p>q \coloneqq pq \vee (\negate{p} \wedge \neg ( \negate{p} U p\negate{q}))
\end{equation*}
That is, either $pq$ is actually true, or there is no $p\negate{q}$-world such that all worlds between it and actuality, as well as the actual world, are $\negate{p}$-worlds.  If $\prec$ strictly well-orders every set of the form $\{x:x\succ w\}$, so that we can turn $\seq{W,\prec}$ into a linear order model, the $>$ defined in terms of $U$ will be interchangeable with the primitive one.  However, an embedding in the other direction isn't possible. Even setting aside past-directed operators, this temporal logic can express `at the next world, $\varphi$' as $\bot U \varphi$, but there is no way to express this in terms of $>$ in linear order models. To see this, consider the sequence model $\sigma$ where $p_0$ is true at $\sigma$ and at its first tail, and false at every tail thereafter, and every other atom is false at every tail. A simple induction shows that $\sigma$ and its first tail verify exactly the same sentences of the conditional language; by contrast, $\bot U \neg p$ is false at $\sigma$ but true at the first tail of $\sigma$.

Some of the literature in temporal logic (especially in computer science) uses a different semantics for $U$, where $t\Vdash pUq$ iff there exists $s\succeq t$ such that $s\Vdash q$ and $u\Vdash p$ whenever $t \preceq u \prec s$.
  Given this  clause for $U$, $>$ and $U$ are interdefinable over linear order models (again, assuming that  $\prec$ strictly well-orders every set of the form $\{x:x\succ w\}$):
\begin{align*}
    p>q &\coloneqq \neg (\negate{p} \mathbin{U} p\negate{q}) \\
    p\mathbin{U}q &\coloneqq \neg ((q\vee \negate{p})>\negate{q})
\end{align*}
However, the most common logical system using the reflexive $U$, the linear temporal logic LTL  \citep{Pnueli:1977}, also includes a primitive ``at the next time'' operator $X$, which cannot be defined in terms of $>$.}

From \cref{canon1} a number of interesting invariance facts immediately follow. First, we can extend any pointed $\omega$-sequence model to a \emph{full} pointed $\omega$-sequence model in which $W$ includes all $\omega$-sequences over $P$, extending the valuation to the new sequences however we please, without making any difference to what's true in the model.  The logic of full $\omega$-sequence frames is thus the same as the logic of all $\omega$-sequence frames. From the other end, we can prune any pointed $\omega$-sequence model back to the \emph{generated} model in which $W$ is just the set of all non-empty tails of the designated sequence without making a difference to what's true in the model.  Thus the logic of $\omega$-sequence frames whose domain is $\tails{\sigma}$ for some $\sigma$ is also the same as the logic of all $\omega$-sequence frames.  

We can also use \cref{canon1} to identify some further conditions which we could, if we wished, impose on $\omega$-sequence models without making a difference to the logic.  In our models, we allow the domain to include \emph{eventually cyclic} sequences some of whose tails are identical: for instance, $\seq{1,2,1,2,\ldots}$ is the first, third, fifth, \ldots tail of $\seq{3,1,2,1,2,\dots}$.%
\footnote{$\sigma$ is eventually cyclic iff for some $n$ and for some $m>0$, for all $j: \sigma^{[n+jm:]} = \sigma^{[n:]}$.}  
However, it would not matter if we ruled out eventually cyclic sequences, since none of the sequences in the minimal representation are eventually repeating.%
\footnote{It turns out that we also get the same logic if we \emph{require} the sequences to be eventually cyclic.  This follows from our completeness theorem for \vanf, which works by generating models all of whose sequences are eventually cyclic.}
For the same reason, it would make no difference if we ruled out \emph{all} repetition of protoworlds within a sequence, so that (e.g.) we cannot have a sequence beginning $\seq{1,2,1,\dots}$.

With all this in hand, it is worth noting from the opposite direction that some approaches which bear a close resemblance to $\omega$-sequence semantics have logics that are very different from the logics we consider, and indeed are orthogonal to \stal, rather than strengthening \stal. Two noteworthy recent examples are the approach of \citet{bacon:2015}, who develops a version of sequence semantics which gives up Reciprocity; and that of \citet{Santorio:2021a}, which marries finite sequence semantics with the domain semantics from \citet{Yalcin:2007}. We will set aside these approaches, as well as other variants whose logic is orthogonal to \stal, focusing instead on semantics corresponding to logics that extend \stal.

\section{Flattening\label{flattening1}}

We are now in a position to begin answering our  main question: what is the logic of $\omega$-sequence models?   

Since we were able to present $\omega$-sequence models as a special case of order models, it is immediate from the soundness of \stal with respect to order models that the logic of $\omega$-sequence models includes \stalp.  But it includes more as well.  For an obvious example of how it goes beyond \stalp, consider the following modal schema:
\begin{equation}
    \ourtag{4}
    \Box p \to \Box \Box p
\end{equation}
4 is valid on a modal frame just in case its accessibility relation is transitive.  It is consequently valid on $\omega$-sequence frames: $\rho$ is accessible from $\sigma$ just in case $\rho$ is a tail of $\sigma$, and any tail of a tail of $\sigma$ is a tail of $\sigma$.  But 4 is not part of \stal, whose frames need not be transitive; again, the  modal logic of \stal is \textsf{KT}, which does not include 4.

Another modal schema that is not part of \stal is
\begin{equation}
    \ourtag{H} 
    (\lozenge p \wedge\lozenge q) \to (\lozenge (p\wedge\lozenge q)\vee \lozenge (q\wedge\lozenge p))
\end{equation}
H is valid on a modal frame just in case it's \emph{connected}: whenever $wRv$ and $wRu$, either $vRu$ or $uRv$.  The accessibility relations of $\omega$-sequence frames are connected: if $\tau$ and $\rho$ are distinct tails of $\sigma$, whichever of them has greater rank is a tail of the other and hence accessible from it.  But H is not part of \stal or even of \stal extended with the 4 axiom, since order frames with transitive accessibility relations need not have connected accessibility relations.  

In addition to 4 and H, the logic of $\omega$-sequence frames also includes schemas which essentially involve the conditional. We will discuss two such schemas, which together yield an axiomatization of the logic of $\omega$-sequence models. We begin with a schema we call \textsf{Flattening}: 
\[
    \ourtag{Flattening} 
    p \cond (pq \cond r )\leftrightarrow pq \cond r
\]
Let \vflat be the result of adding Flattening to \stal (i.e., the smallest extension of \stal including every instance of Flattening and closed under Detachment and Normality).  

It is easy to see that Flattening is valid on $\omega$-sequence frames.  The right-hand side is false at a sequence just in case it has a $pq$-tail and $r$ is false at its first $pq$-tail.  The left-hand side is false just in case it has a $p$-tail with a $pq$-tail, and $r$ is false at the first such $pq$-tail.  But any sequence with a $pq$-tail thereby has a $p$-tail, and in any sequence with a $pq$-tail, its first $pq$-tail is identical to the first $pq$-tail of its first $p$-tail. So the two sides of Flattening  have the same truth-value in any pointed $\omega$-sequence model. 

Indeed, to foreshadow a bit, note that this reasoning depends just on the structure of the tailhood relation, not on the ordinal structure of $\omega$-sequences. That means that Flattening is valid on sequence semantics \emph{whatever the domain of the underlying sequence}. Indeed, we will see that \vflat is the logic of ordinal sequence semantics---a variant of $\omega$-sequence semantics where the underlying sequences can take any ordinal as their domain---whereas the logic of $\omega$-sequences is strictly stronger than \vflat.

On the other hand, there are pointed order models in which instances of Flattening are false, such as in the model in \cref{notflat}.

\begin{figure}
\begin{center}
\begin{tikzpicture}[->,>=stealth',shorten >=1pt,shorten <=1pt, auto,node
distance=2.5cm,semithick]
\tikzstyle{every state}=[fill=gray!20,draw=none,text=black]
\node[circle,draw=black!100, fill=black!100, label=below:$1$,inner sep=0pt,minimum size=.1cm] (1) at (0,0) {{}};
\node[circle,draw=black!100, fill=white!100, label=below:$2_{p}$,inner sep=0pt,minimum size=.1cm] (2) at (2,0) {{}};
\node[circle,draw=black!100, fill=white!100, label=below:$3_{pq}$,inner sep=0pt,minimum size=.1cm] (3) at (4,0) {{}};

\node[circle,draw=black!100, fill=black!100, label=below:$2_p$,inner sep=0pt,minimum size=.1cm] (1) at (0,-1.3) {{}};
\node[circle,draw=black!100, fill=white!100, label=below:$4_{pqr}$,inner sep=0pt,minimum size=.1cm] (3) at (2,-1.3) {{}};
\node[circle,draw=black!100, fill=black!100, label=below:$3_{pq}$,inner sep=0pt,minimum size=.1cm] (1) at (0,-2.6) {{}};
\node[circle,draw=black!100, fill=black!100, label=below:$4_{pqr}$,inner sep=0pt,minimum size=.1cm] (3) at (0,-3.9) {{}};

\end{tikzpicture} 
\end{center}
\caption{\small A countermodel to Flattening. Each horizontal line represents the order induced at the left-most (shaded) world, with a world appearing to the left of another just in case the first precedes the second in the relevant ordering, so, e.g., $<_3$ is the empty order, while $<_2$ is the order $\{\seq{2,4}\}$. Subscripts indicate atomic valuations. Thus $1\nVdash pq>r $ while $1\Vdash p>(pq>r)$. For a counterexample to the opposite direction of the Flattening biconditional, make $r$ true at 3 but false at 4.}\label{notflat}
\end{figure}
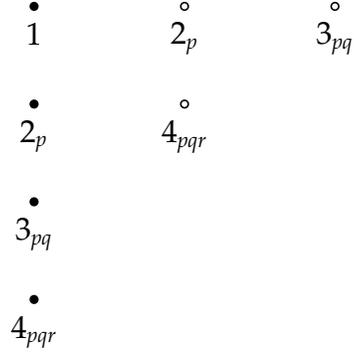

It is worth noting a few alternative axiomatizations of this logic. 

 First, \vflat is equivalent to the logic obtained by instead adding to \stal a corresponding schema, $\gg$-Flattening, formulated not with $>$ but with the strong conditional connective $\gg$ defined as:
 \[ p\gg q :=  \neg(p > \neg q))\]
 (or equivalently, $p\gg q := \lozenge p\wedge (p > q)$). That schema is:
\begin{equation*}
    \ourtag{$\gg$-Flattening}
    p \gg (pq \gg r) \leftrightarrow pq \gg r
\end{equation*}
$\gg$-Flattening is valid in \vflat, since   $p \gg (pq \gg r) \leftrightarrow pq \gg r$ follows from the Flattening-instance  $p > (pq > \neg r) \leftrightarrow pq > \neg r$ by negating both sides. 
Conversely, the logic comprising \stal plus every instance of $\gg$-Flattening proves every instance of Flattening. 
We reason by cases. First assume $\lozenge(pq)$. Note that the instance $p \gg (pq \gg \top) \leftrightarrow pq \gg \top$ of $\gg$-Flattening is equivalent to $ p\gg \lozenge(pq)\leftrightarrow \lozenge(pq)$, so we also have $ p\gg \lozenge (pq)$. By $\gg$-Flattening we have   $pq\gg r \leftrightarrow p\gg (pq \gg r)$, which together with  $\lozenge(pq)\wedge p\gg \lozenge (pq)$ entails  $pq>r\leftrightarrow p>(pq>r)$.
Next assume $\neg \lozenge(pq)$, so we also have $\neg (p\gg \lozenge (pq))$. But 
$\neg \lozenge(pq) \wedge \neg (p\gg \lozenge (pq))$ \stal-entails  $(pq>r) \wedge( p>(pq>r))$ and hence $(pq>r) \leftrightarrow( p>(pq>r))$ .

Second, we can break up Flattening into its two directions
\begin{align*}
    \ourtag{Cautious Importation} 
    p>(pq\cond r)&\to pq \cond r \\
    \ourtag{Cautious Exportation} (pq\cond r) &\to p> (pq \cond r)
\end{align*}
The special cases of these principles where $r$ is a contradiction are interesting; these can be written using $\lozenge$ as
\begin{align*}
    \ourtag{Crashing Cautious Importation} 
    p>\neg\lozenge \pq &\to \neg\lozenge \pq \\
    \ourtag{Crashing Cautious Exportation} \neg\lozenge \pq &\to p> \neg\lozenge \pq
\end{align*}
It turns out that given the validity of either one of Cautious Importation and Cautious Exportation, we only need the `Crashing' restriction of the other to get back the full strength of Flattening.  For example, to derive Cautious Exportation from Cautious Importation plus Crashing Cautious Exportation, suppose $pq > r$.  If $\neg\lozenge \pq$, then $p > \neg\lozenge \pq$ and hence $p > (pq >r)$.  Otherwise, $\neg(pq > \neg r)$, so by Cautious Importation $\neg(p > (pq > \neg r))$, so by CEM $p > \neg(pq > \neg r)$, so by CEM and Normality, $p > (pq > r)$.  The other derivation is analogous.  

Third, it is often convenient (especially in working with natural language examples) to use ``rule'' forms of these axioms, where the conjunction $pq$ is replaced by any $q$ for which we have $\vdash q \to p$.  For example, we can also characterise \vflat as the result of closing \stal under the following rule:
\begin{equation*}
    \ourtag{Flattening Rule}
    \text{If }\vdash q\to p \text{ then }\vdash p > (q > r) \leftrightarrow q > r
\end{equation*}
This implies Flattening since $\vdash pq \to p$, and follows from Flattening by the substitution of logical equivalents (since from $\vdash q \to p$ it follows  that $q$ is logically equivalent to $pq$). 

\vflat turns out to include many---though not all---of the distinctive principles that hold in $\omega$-sequence models but not all order models.  For example, it includes both the modal axioms 4 and H.  4 is actually equivalent to Crashing Cautious Exportation in \stal. From the latter we derive 4 by applying the rule form of Crashing Cautious Exportation to the \stal-theorem $\neg p \to \neg\Box p$ to get
\begin{equation*}
    \neg\lozenge\neg p \to \neg\Box p > \neg\lozenge\neg p
\end{equation*}
which simplifies to $\Box p \to \Box\Box p$. For the other direction, start with $\Box \neg (pq) \to \Box\Box \neg (pq)$, which is an instance of 4; since $\Box q$ entails $p>q$ in \stal, this entails $\Box \neg (pq) \to p > \Box \neg (pq)$ as needed.

For H, we use Crashing Cautious Importation (in rule form, applied to the tautology $p \to (p\vee q)$) to get:
\begin{equation*}
    (p\vee q) > \neg\lozenge p \to \neg\lozenge p
\end{equation*}
Contraposing and applying CEM, this implies $\lozenge p \to (p\vee q) > \lozenge p$.  
By parallel reasoning we also have $\lozenge q \to (p\vee q) > \lozenge q$, and hence by Normality
\begin{equation*}
    \lozenge p \wedge \lozenge q \to 
    (p\vee q) > (\lozenge p \wedge \lozenge q)
\end{equation*}
But CEM, Normality, and Reciprocity together yield the entailment from $(p\vee q)>r$ to $p>r \vee q>r$; applying this to the above, we have:
\begin{equation*}
    \lozenge p \wedge \lozenge q \to 
    (p > (\lozenge p \wedge \lozenge q)) \vee (q > (\lozenge p \wedge \lozenge q))
\end{equation*}
Since $\lozenge p$ and $p>q$ entail $\lozenge(p\wedge q)$ in \stalp, this implies H:
\begin{equation*}
    \lozenge p \wedge \lozenge q \to 
    \lozenge(p \wedge \lozenge q) \vee \lozenge(q \wedge \lozenge p)
\end{equation*}

\subsection{Evaluating Flattening\label{evalflat}}

As we mentioned above, one way to interpret order models is as representing relative similarity between worlds. From the point of view of that interpretation, it is no accident that Flattening fails in those models: Flattening is in clear tension with that interpretation of order models.  

Schematically, similarity-based theories of the conditional predict Flattening can fail because the most similar $pq$-world(s) to actuality need not be the most similar $pq$-world(s) to the $p$-world(s) most similar to actuality.
For a simple concrete example of this, consider a variation on a  toy example of Lewis's involving a line L. Lewis's similarity-based intuition was that, given a world $w$ where L has length $n$, if $x$ and $y$ are otherwise exactly alike, except that in $x$ the length of L is  closer to $n$ than it is in $y$, then $x$ is more similar to $w$ than $y$ is.
Now suppose that L is in fact 10 inches long, and compare \ref{linea} and \ref{lineb}:

\ex. 
    \a. If L hadn't been strictly between 8--11 inches, then if it hadn't been strictly between 8--13 inches, it would have been 13 inches.\label{linea}
    \b. If L hadn't been strictly between 8--13 inches, then it would have been 13 inches.\label{lineb}

\ref{linea} and \ref{lineb} instantiate the two sides of Flattening (in the rule-based formulation from the last section) since not being strictly between 8--13 inches entails not being strictly between 8--11 inches. But, if we interpret conditionals via similarity, in particular with Lewis's simple assumption above, then \ref{linea} should be true while \ref{lineb} is false. The world $x$ most similar to actuality where the line isn't strictly between 8 and 11 inches is one where it's  11 inches. The world $y$ most similar to $x$ where the line isn't strictly between 8--13 inches is one where it's 13 inches. So \ref{linea} is true. By contrast, the world most like actuality where the line isn't between 8 and 13 inches is one where it's 8 inches, so \ref{lineb} is false. (For a counterexample in the opposite direction, change `it would have been 13 inches' to `it would have been 8 inches'.)

Is this counterexample convincing? We find it difficult to detect a clear divergence between \ref{linea} and \ref{lineb}, except by doggedly holding in mind the Lewisian interpretation of `if $p$\ldots' as a proxy for `in the world most similar to actuality where $p$$\dots$'. 
Of course, a defender of a similarity-based view could claim that we simply fail to clearly see a contrast which does exist here. But they would need a story about why we make an error here, whereas we have clear intuitions about many other subtle judgments recorded in the literature on conditionals. Barring such a theory, the apparent validity of Flattening might provide a new argument in the battery of well-known arguments against  similarity theories of conditionals.

Setting aside the baggage of similarity, we can try to evaluate Flattening on its own terms, by considering pairs of sentences that would be logically equivalent according to Flattening and seeing whether they in fact seem equivalent.  It seems to us that the results of this exercise speak in favor of Flattening. For instance, compare these pairs:
\ex. 
\a. If Mark and Sue are both at the party, there will be a conflagration.
\b. If Mark is at the party, then if Mark and Sue are both at the party, there will be a conflagration. 

\ex. \a. If he had gotten an espresso and it had been overextracted, he would have had a fit.  
\b. If he had gotten an espresso, then if he had gotten an espresso and it had been overextracted, he would have had a fit.  

These feel pairwise equivalent. At best, the (b)-variants feel \emph{redundant}; the first antecedent feels like it's doing nothing. This is not explained by order semantics, according to which the two variants have logically orthogonal meanings. But this intuition is explained (given standard theories of redundancy) if Flattening is valid, since then the (b)-variants are equivalent to their consequents. 

By relying on the rule form of Flattening, we can formulate test pairs that feel somewhat less clunky: 

\ex. \a. If he had gotten an espresso and it had been overextracted, he would have had a fit.  
\b. If he had gotten an espresso, then he would have had a fit if he'd gotten an overextracted one.

\ex. \a. If he had been in  the south of France, he'd have had a great time.  
\b. If he had been in France, he'd have had a great time if he had been in the south of France. 

Again, these feel pairwise equivalent. 
We have checked many instances of Flattening, in both the indicative and subjunctive mood, and have not found clear counterexamples.  

To be sure, there are superficial counterexamples to Flattening involving tense and anaphora:

\ex. \a. If John wins, then if John and Sue win, John will have won twice.
\b. If John and Sue win, John will have won twice. 

\ex. \a. If a man came in, then if a man came in and a man came in, then three men came in.
\b. If a man came in and a man came in, then three men came in. 

But it seems implausible that these are  counterexamples to Flattening, and more plausible that the felt inequivalence in these pairs arises from different indexing of tense/anaphora in the two pairs. This is a somewhat delicate issue, involving questions about the representation of context-sensitivity that are beyond our scope. But it is worth noting that if we accept these as counterexamples to Flattening, then we also have to accept that there are counterexamples to the very widely accepted `Contraction' principle $p\cond (p\cond q) \leftrightarrow p\cond q$.  For the following also feel pairwise inequivalent:

\ex. \a. If a man came in, then if a man came in, then two men came in.
\b. If a man came in, then two men came in. 

Thus Flattening seems, from the point of view of natural language, in at least as good \emph{prima facie} standing as Contraction (a theorem of \stalp as well as many weaker  logics of conditionals, though one that must notoriously be rejected by non-classical logicians attempting to maintain a na\"ive theory of truth).  This is a strong position to be in.

However, there are reasons to be wary about these appearances.  Flattening is a ``cautious'' cousin of the following well-known axiom schema: 
\[
    \ourtag{Import-Export (IE)} 
    p \cond (q \cond r )\leftrightarrow pq \cond r
\]
The only difference between Flattening and Import-Export is that, in Flattening, $p$ recurs in the antecedent of the conditional consequent on the left-hand side, so we have $p\cond (pq\cond r)$, rather than $p\cond (q\cond r)$ as in IE. From a logical point of view, this small difference is crucial, for, as \citet{Dale:1974,Dale:1979,Gibbard:1981}, showed,  adding  IE to \stal (or, indeed, to many weaker conditional logics) collapses  $>$ to the material conditional, that is, results in a logic that validates: 
\[
    \ourtag{Materialism} 
    p\cond q \leftrightarrow (p\to q)
\]
Materialism is, however, widely rejected in the literature on conditionals, as we noted above; see \citealt{Edgington:1995} for many arguments against it. For one brief argument, consider the claim that no tree is deciduous if it keeps its leaves through the winter. According to materialism this entails that every tree keeps its leaves through the winter, since the negation of $p\to q$ entails $p$. But this is obviously wrong.\footnote{A reviewer notes that one could respond to this argument with a decompositional approach to `no', where it decomposes into $\forall $ and $\neg$ with the `if'-clause coming in between. While this is an interesting possibility, such a decomposition oviously threatens overgeneration, so something would have to be said about how to constrain this syntactic possibility.} For another argument, note that $(p\to q) \vee (\neg p\to r)$ is a theorem, but disjunctions of conditionals don't always strike us that way; there is little temptation to think that \ref{earth} is true:

\ex.\label{earth} The earth is flat  if it's Tuesday or the earth is flat if it's not Tuesday. 

To be sure, there are some defenders of Materialism \parencite[e.g.][]{jackson:1979,Williamson:2019}. But most view it as an untenable hypothesis. That makes the Dale/Gibbard collapse result a challenge to the co-tenability of Import-Export with \stal, and it is important to ask  whether this challenge extends to Flattening.

The answer is that it does not. For we have already seen that Flattening is valid in $\omega$-sequence models, while Materialism is not. For instance, in the $\omega$-sequence model generated from $\sigma=\seq{1,2,1,2,\dots}$, where $p$ is false at $\sigma$ and true at $\seq{2,1,2,1,\dots}$, while $q$ is false at both sequences, the material implication $p\to q$ is true at $\sigma$ while the conditional $p\cond q$ is false. 
Nor does Flattening lead to any other troubling form of triviality, as $\omega$-sequence models show. 
Moreover, at least one error theory of the apparent validity of IE precisely relies, in part, on the validity of Flattening \citep{Bounds}.  So validating Flattening may turn out to be a key stepping stone towards explaining the \emph{apparent} validity of IE.

A final relevant observation is that there are compelling counterexamples to IE in the case of subjunctive conditionals: e.g. the sentences in \ref{matcha}  can intuitively diverge in meaning \citep{Etlin:2008}.

\ex.\label{matcha}\a. If the match had lit and it had been soaked in water, then it would have lit.
\b. If the match had lit, then it would have lit if it had been soaked in water. 

A standard desideratum in the theory of conditionals is to give a unified theory of indicative and subjunctive conditionals: there is one word `if' which can express both conditionals, depending on the mood of the rest of the sentence. So we have positive reason \emph{not} to validate IE as a matter of the logic of `if', and instead to explain its apparent validity for indicatives, and lack thereof for subjunctives, as arising from the interaction of the meaning of `if' with mood. But matters are different for Flattening, which appears valid for both indicatives and subjunctives.  Indeed, the pairs which look like counterexamples to IE for subjunctive conditionals still look equivalent when we change them to instantiate Flattening:

\ex.\label{matchb}\a. If the match had lit and it had been soaked in water, then it would have lit.
\b. If the match had lit, then it would have lit if it had been soaked in water and it had lit.

In sum, despite their superficial similarity, Flattening and IE have very different statuses vis-\`a-vis the theory of conditionals; reasonable arguments against IE do not extend to Flattening. 

Stepping back, 
we think the relationship of Flattening to Import-Export is reminiscent of the relationship of the following pairs of principles:
\begin{align*}
    \ourtag{Cautious Transitivity}
    (p > q) \wedge (pq > r) &\to (p \cond r)
\\
\ourtag{Transitivity} 
   ( p\cond q) \wedge (q\cond r) &\to (p\cond r)
   \end{align*}
\begin{align*}   
    \ourtag{Cautious Monotonicity}
    (p \cond qr) &\to (pr \cond q)
\\
    \ourtag{Monotonicity} 
   ( p\cond q) &\to (pr\cond q)
\end{align*}
The ``incautious'' principles Transitivity and Monotonicity are widely rejected \citep{Stalnaker:1968}; and, just as for Import-Export, adding either of them to \stal results in Materialism.  By contrast, the ``cautious'' principles are widely accepted, and are theorems of \stal and of many other popular conditional logics.  In this light, we might naturally regard Flattening as a plausible \emph{cautious} cousin of the implausible, incautious Import-Export principle.

Nevertheless, we do not want to suggest that the case for the validity of Flattening is anything like watertight.  The strongest reason for worry we see involves the fact that, as we noted, Flattening implies the 4 and H principles for the $\Box$ defined in terms of $>$.  This is a potential warning sign, since there are well known arguments against the \textsf{4} principle for many seemingly relevant interpretations of $\Box$, and many of these arguments also extend to the \textsf{H} principle.  \citet{Williamson2000}  and \citet{DorrGoodHaw} argue against 4 on an interpretation where $\Box$ means `$a$ is in a position to know that\ldots'.  This suggests that 4 might also fail for the epistemic `must' if its meaning is related to that of `know'; and the arguments can also be adapted to directly use epistemic `must'.   Insofar as the $\Box$ defined in terms of $>$ interpreted as an indicative conditional is equivalent to, or otherwise intimately connected to, the epistemic modal, these considerations may also threaten the 4 axiom for that $\Box$.  Meanwhile, when we turn to the notion of \emph{nomic} necessity---which might be thought to be identical to the defined $\Box$ on some counterfactual interpretations of $>$---we find influential forms of Humeanism which motivate rejection of at least H and perhaps also 4.  On the `best system' theory of laws \citep{Lewis:1994}, being a nomic necessity is being entailed by whatever collection of true axioms achieves the best balance of simplicity and strength.  On this picture, there could be a complex world where there are two jointly inconsistent simple propositions both of which are false but nomically possible, and (because of their simplicity) such that necessarily, if they are true, they are nomically necessary.  This is inconsistent with the connectedness of nomic accessibility, and thus with \textsf{H}.  (It is also arguable that Humeans should reject 4 for nomic necessity, though we will not go into that here.)  While it is not so plausible that nomically necessary truths are counterfactually necessary on \emph{every} interpretation of counterfactuals, one might think that there are \emph{some} salient interpretations of counterfactuals that involve ``holding the laws fixed'' in such a way that these Humean worries would carry over to 4 and H for the defined $\Box$.  There are also potential reasons for doubting both 4 and H for \emph{metaphysical} modality, which many take to be equivalent to $\Box$ defined in terms of counterfactuals: \citet{SalmonRE2} rejects 4 for metaphysical necessity in order to solve certain puzzles of Tolerance (though see \cite{DorrHawthorneYliVakkuri} for an alternative approach to those puzzles which preserves 4); meanwhile, \citet{BaconLC} and \citet{BaconDorrClassicism} endorse 4, but are sympathetic to versions of ``combinatorialism'' on which metaphysical modality would not obey H.%
\footnote{On these views, we can have propositions $p$ and $q$---e.g., the results of predicating two different fundamental properties of some fundamental individual---such that it is metaphysically possible both that $p=q$ and that $p=\neg q$.  But $p=q$ entails $\neg\lozenge(p=\neg q)$ and $p=\neg q$ entails $\neg\lozenge(p=q)$, so this violates H.}

We will not here undertake to evaluate these arguments against 4 and H for various familiar interpretations of necessity, or  adjudicate the question to what extent they carry over to the $\Box$ defined in terms of `if'.  If we accept these arguments (for that defined $\Box$), we will have to reject Flattening, and develop an error theory of its apparent validity.  But the appearances favoring Flattening are quite strong, so the difficulty of this task should not be underestimated. In any case, we think it's clear that the logic \vflat has strong  prima facie appeal as (at least part of) the logic of the natural language conditional.
 
If we do accept \vflat, then we might naturally hope to develop some intuition about its semantics which might help to \emph{explain} its validity. As we saw above, similarity-based approaches will not work here.  Indeed, even the more abstract terminology of ``closeness'' may be rather misleading, since it suggests an underlying \emph{metric} structure, and no ordinary metric space has the property that when $X\subseteq Y$, the closest element of $X$ to a given point is the closest element of $X$ to the closest element of $Y$ to that point.  Of course, if ``closeness'' is just a label for whatever plays the role of the order function in order semantics, a closeness-based semantics is consistent with the validity of \vflat; but this does not suggest anything we could think of as an \emph{explanation}.   
We think this is an important challenge, to which we do not have a satisfactory answer; we leave it as an open problem for proponents of \vflat.%
    \footnote{Thanks to Robert Stalnaker and an anonymous referee for pushing us on this question.}

\subsection{The logic \vflat} 
As we have already asserted, the logic of $\omega$-sequence frames is not exhausted by \vflat.  To see why this is the case, and to get some intuition for what is missing, it will be useful to introduce a different class of order frames with respect to which \vflat is complete as well as sound: the \emph{flat} order frames.  
\begin{definition} Order frame $\seq{W,<}$ is 
\begin{itemize}
 \item
 \emph{semi-flat} iff whenever $x<_w y$, $y\in R(x)$ and for all $z$ such that $y <_w z$ or $z \in R(x)\setminus R(w)$, $y <_x z$.
 \item  
 \emph{flat} iff it is  semi-flat and transitive.
\end{itemize}
\end{definition}
(Note that in a transitive frame, the case  where $x <_w y$ and $z\in R(x)\setminus R(w)$ cannot arise, since $x <_w y$ guarantees $x\in R(w)$ and hence $R(x)\subseteq R(w)$.) 
In other words: a flat order frame is a transitive frame where, whenever $x$ is accessible from $w$, $<_x$ orders all the worlds that come \emph{after} $x$ in $<_w$ in just the same way they were ordered in $<_w$.  By contrast, worlds that came \emph{before} $x$ in $<_w$ may occur at any position in $<_x$, or not be accessible from $x$ at all.  See \cref{flatorders1} for an illustration of a flat and non-flat order frame.

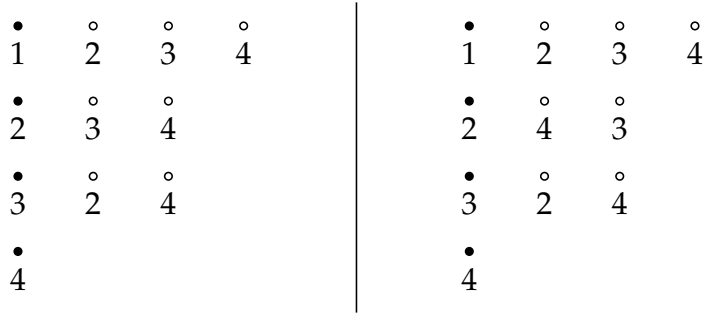
\begin{figure}
\begin{center}
\tikzset{every loop/.style={min distance=10mm,looseness=10}}
\begin{tikzpicture}[->,>=stealth',shorten >=1pt,shorten <=1pt, auto,node
distance=2.5cm,semithick]
\tikzstyle{every state}=[fill=gray!20,draw=none,text=black]

\node[circle,draw=black!100, fill=black!100, label=below:$1$,inner sep=0pt,minimum size=.1cm] (1) at (0,0) {{}};
\node[circle,draw=black!100, fill=white!100, label=below:$2$,inner sep=0pt,minimum size=.1cm] (2) at (1,0) {{}};
\node[circle,draw=black!100, fill=white!100, label=below:$3$,inner sep=0pt,minimum size=.1cm] (3) at (2,0) {{}};
\node[circle,draw=black!100, fill=white!100, label=below:$4$,inner sep=0pt,minimum size=.1cm] (4) at (3,0) {{}};

\node[circle,draw=black!100, fill=black!100, label=below:$2$,inner sep=0pt,minimum size=.1cm] (1) at (0,-1) {{}};
\node[circle,draw=black!100, fill=white!100, label=below:$3$,inner sep=0pt,minimum size=.1cm] (2) at (1,-1) {{}};
\node[circle,draw=black!100, fill=white!100, label=below:$4$,inner sep=0pt,minimum size=.1cm] (3) at (2,-1) {{}};

\node[circle,draw=black!100, fill=black!100, label=below:$3$,inner sep=0pt,minimum size=.1cm] (1) at (0,-2) {{}};
\node[circle,draw=black!100, fill=white!100, label=below:$2$,inner sep=0pt,minimum size=.1cm] (2) at (1,-2) {{}};
\node[circle,draw=black!100, fill=white!100, label=below:$4$,inner sep=0pt,minimum size=.1cm] (3) at (2,-2) {{}};

\node[circle,draw=black!100, fill=black!100, label=below:$4$,inner sep=0pt,minimum size=.1cm] (1) at (0,-3) {{}};

\node (10) at (4.5,0.5) {{}};
\node (11) at (4.5,-4) {{}};
\path (10) edge[-] node {{}} (11);

\node[circle,draw=black!100, fill=black!100, label=below:$1$,inner sep=0pt,minimum size=.1cm] (1) at (6,0) {{}};
\node[circle,draw=black!100, fill=white!100, label=below:$2$,inner sep=0pt,minimum size=.1cm] (2) at (7,0) {{}};
\node[circle,draw=black!100, fill=white!100, label=below:$3$,inner sep=0pt,minimum size=.1cm] (3) at (8,0) {{}};
\node[circle,draw=black!100, fill=white!100, label=below:$4$,inner sep=0pt,minimum size=.1cm] (3) at (9,0) {{}};

\node[circle,draw=black!100, fill=black!100, label=below:$2$,inner sep=0pt,minimum size=.1cm] (1) at (6,-1) {{}};
\node[circle,draw=black!100, fill=white!100, label=below:$4$,inner sep=0pt,minimum size=.1cm] (2) at (7,-1) {{}};
\node[circle,draw=black!100, fill=white!100, label=below:$3$,inner sep=0pt,minimum size=.1cm] (3) at (8,-1) {{}};

\node[circle,draw=black!100, fill=black!100, label=below:$3$,inner sep=0pt,minimum size=.1cm] (1) at (6,-2) {{}};
\node[circle,draw=black!100, fill=white!100, label=below:$2$,inner sep=0pt,minimum size=.1cm] (2) at (7,-2) {{}};
\node[circle,draw=black!100, fill=white!100, label=below:$4$,inner sep=0pt,minimum size=.1cm] (3) at (8,-2) {{}};

\node[circle,draw=black!100, fill=black!100, label=below:$4$,inner sep=0pt,minimum size=.1cm] (1) at (6,-3) {{}};
\end{tikzpicture} 

\end{center}
\caption{\small  Illustrations of a flat order function (left) and a non-flat (because not semi-flat) order function (right) in frames with worlds \{1,2,3,4\}.}\label{flatorders1}
\end{figure}

\emph{Linear} order frames, whose order function can be derived from some unrelativized $\prec$ via $x <_w y$ iff $w\preceq x \prec y$, are always flat: they are obviously transitive, and if $x<_w y<_w z$, then $w\preceq x \prec y \prec z$ and so $y\in R(x)$ and $y <_x z$.  But not all flat frames are linear: for example, the one on the left of \cref{flatorders1} is not.

We can show that Flattening is valid on all flat order frames.  In fact, flatness \emph{characterizes} Flattening, in the sense that the order frames on which Flattening is valid are exactly the flat ones.  Since we already know that Flattening is equivalent to the combination of Cautious Importation with Crashing Cautious Exportation, and that the latter is equivalent to \textsf{4} (which is characterized by transitivity), it suffices to prove the following lemma:

\begin{lemma}\label{absorpcon} 
    Cautious Importation is valid on an order frame $\seq{W,<}$ iff it is semi-flat.
\end{lemma}
\begin{proof}\leavevmode
\bi \item[$\Rightarrow$] Suppose $x<_w y$.  
\bi
 \item[(a)] If $y \notin R(x)$, let $V(p)=\{x,y\}, V(q)=\{y\}, V(r)=\varnothing$.
 \item[(b)] If $y \in R(x)$ and for some $z$, $y <_w z$ but $y \nless_x z$, then:
 \bi
\item if $z\notin R(x)$, let $V(p)=\{x,z\}, V(q)=\{z\}, V(r)=\varnothing$; 
\item if $z<_x y$, let $V(p)=\{x,y,z\}, V(q)=\{y,z\}, V(r)=\{z\}$.
\ei
\item[(c)] if  $y \in R(x)$ and for some $z$, $z \in R(x)\setminus R(w)$ but $y\nless_x z$, let $V(p)=\{x,z,y\}, V(q)=\{y,z\}, V(r)=\{z\}$. 
\ei
In each case, Cautious Importation will fail at $w$.
\item[$\Leftarrow$] 
Suppose $\seq{W,<}$ is semi-flat, and consider any $V$ and $w\in W$ such that $w\Vdash p>(pq\cond r)$.  
If there is no $pq$-world in $R(w)$, or the first $p$-world in $<_w$ is a $pq$-world, then $w\Vdash pq > r$ and we are done. Otherwise, let $x$ be the first $p$-world in $<_w$ and let $y$ be the first $pq$-world in $<_w$.  Since $x<_w y$, $y\in R(x)$, so $R(x)$ contains a $pq$-world; let $u$ be the first $pq$-world according to $<_x$. We know $u\Vdash r$, so it suffices to show $y=u$.  We cannot have $u<_w y$ since then $y$ would not be the first $pq$-world in $<_w$.  And we also cannot have $y<_wu$ or $u\notin R(w)$, since in either case semi-flatness would yield $y<_xu$, meaning that $u$ would not be the first $pq$-world in $<_x$.  The only remaining possibility is that $y=u$.  
\qedhere
\ei
\end{proof}

Since Flattening is equivalent (modulo \stal) to 4 together with Cautious Importation, it follows that Flattening is characterized by flatness:
\begin{theorem}\label{flatt}
Flattening is valid on an order frame $\seq{W,<}$ iff $\seq{W,<}$ is flat. 
\end{theorem}

We can also formulate frame conditions that characterize Cautious Exportation and Crashing Cautious Importation; the combination of these conditions is also equivalent to flatness, since the conjunction of the axioms is equivalent to Flattening.%
\footnote{Cautious Exportation is valid on  $\seq{W,<}$ iff it is transitive and for all $w,x,y,z$, if $x<_w y<_w z$ and $z\in R(x)$, then $ y<_x z$.
Crashing Cautious Importation is valid on $\seq{W,<}$ iff whenever $x <_w y$, $y\in R(x)$.}

With this characterization result in hand, we can turn to soundness and completeness results for \vflat.
The right-to-left direction of \cref{flatt} says that \vflat is sound for flat order frames: that is, all the theorems of \vflat are valid on every flat order frame. However, completeness is another matter: a characterization result like \cref{flatt} does not entail a completeness result. Abstractly, a characterization result for a logic \textsf{L} against a background class of frames $\mathcal{F}$ specifies the  subset $\mathcal{F}_\textsf{L}$ of $\mathcal{F}$ such that for all $F\in \mathcal{F}$, $F\in \mathcal{F}_\textsf{L}$ iff \textsf{L} is valid on $F$. However, it is possible that \textsf{L} is not complete with respect to $\mathcal{F}_\textsf{L}$, when there is some sentence $p$ which is valid on every $F\in \mathcal{F}_\textsf{L}$ but is not a theorem of \textsf{L}: i.e., when the set $\mathcal{F}_\textsf{L}$ is ``too small'' to find a countermodel to every non-theorem of \textsf{L}.

And indeed, it is well known that there are normal modal logics which are not complete with respect to the class of modal frames that characterize them: in other words, certain formulae that are not theorems of the logic are nevertheless valid on every Kripke frame (\citealt{Fine:1974,Thomason:1974,vanBenthem1978-VANTSI-4}; see \citealt{Holliday:2019} for a helpful recent discussion). So we cannot simply assume that a characterization result yields a corresponding completeness result. Moreover, the fact that \vflat, like many extensions of \stal, is not strongly complete for \emph{any} class of order frames 
(which we can show analogously to \cref{notstrong}) means that the standard canonical model method for proving completeness will not work for these logics. 

With that said, we do in fact have a completeness result for \vflat: 
in the appendix, we show that \vflat is weakly complete for flat order frames. Indeed, we show that it is weakly complete for \emph{finite} such frames: 
\begin{theorem}\label{flattcom}
    \vflat is weakly complete with respect to finite flat order frames. 
\end{theorem}
Because of its restriction to finite frames (and the fact that flatness is a decidable property of finite frames), this result has the corollary that \vflat is decidable.

Since all linear order frames are flat, the fact that \vflat is sound for flat order frames immediately implies that it is sound for linear order frames.  Since not all flat order frames are linear, one might guess that the logic of linear order frames would be stronger than \vflat.  But in fact the logics turn out to be the same:
\begin{restatable}[]{theorem}{linearcom} \label{linearcom}
    \vflat is sound and weakly complete with respect to linear order frames.
\end{restatable} \noindent
Both \cref{flattcom} and \cref{linearcom} follow from a third completeness result for \vflat which we come to in \cref{generalized} (namely, \cref{ordseqcompleteness}).
That result is based on a generalization of $\omega$-sequence frames which replaces $\omega$-sequences with sequences of arbitrary length, while keeping the other definitions fixed. 

That generalization is natural  to explore when we start from $\omega$-sequence semantics. However, for fans of \vflat but not \vanf, the resulting semantics might be a hard sell: it is at least somewhat difficult to motivate the explanatory force of a highly theoretical notion like transfinite sequences for modeling the meaning of `if'. Similarly, our characterization of \vflat yielded a complex frame condition (flatness) which is not terribly easy to get an independent grip on.
 By contrast, the definition of linear order frames is extremely simple, and, we think, provides the clearest semantic picture of the difference between \vflat and \stal. 
 \stal is sound and complete for order frames whose order functions endow every world with a separate well-orders which need have no relation to that of any other world. By contrast, \vflat is sound and complete for a semantics based on a single master order: each world's order is just the result of deleting all its predecessors from the master order. The less constrained order semantics for \stal has, in our view, more prima facie philosophical plausibility, so this new perspective does not yet explain in some deeper sense why \vflat would be the correct logic for the conditional.  But it may help get a grip on what that question amounts to.

One noteworthy corollary of the soundness direction of \cref{flatt} is that the two modal schemas 4 and H we mentioned in \cref{flattening1} in fact \emph{exhaust} the purely modal content of \vflat.  Here are the two axioms again:
\begin{align*}
    \ourtag{4}  
    \Box p &\to \Box\Box p 
    \\
    \ourtag{H} 
    (\lozenge p \wedge\lozenge q)&\to (\lozenge (p\wedge\lozenge q)\vee \lozenge (q\wedge\lozenge p))
\end{align*}
The modal logic that adds these two axiom schemas to \textsf{KT} is called \textsf{S4.3}, and it is sound and complete (indeed strongly complete) for modal frames with a reflexive, transitive, and connected accessibility relation.  It is also weakly complete for \emph{finite} such frames \citep{Bull:1966}, which  we use to show: 

\begin{theorem}
     When $ p $ belongs to the modal fragment of $\mathcal{L}_>$, $\vdash_{\text{\vflat}} p $ iff $\vdash_{\text{\textsf{S4.3}}} p $.  
\end{theorem}
\begin{proof}
We already established the right-to-left direction in \cref{flattening1}.%
\footnote{By the completeness theorem \ref{flattcom}, we can also give a simple model-theoretic proof: it is easy to see that every flat order model has an accessibility relation that is reflexive, transitive, and connected; thus if $p$ is true in every reflexive transitive connected modal model, it is also true in every flat order model, and thus a theorem of \vflat by completeness.}
For the left-to-right direction,
suppose $p$ is false in a finite reflexive, transitive, and connected pointed model $\seq{v,\seq{W,R,V}}$. Then (by a standard induction, relying on the transitivity of $R$) $p$ is also false in the generated submodel $\seq{v,\seq{R(v),R_v,V_v}}$ (where $R_v$ and $V_v$ are respectively $R$ and $V$ restricted to $R(v)$). Next we give a recipe to transform $\seq{v,\seq{R(v),R_v,V_v}}$ into an equivalent pointed flat order model $\seq{v,\seq{R(v), <,V_v}}$, yielding a flat order model where $p$ is also false.

To construct $<$,  choose a strict linear order $\prec$ on $R(v)$ such that $x\prec y$ implies  $xRy$ (the existence of which is guaranteed by the connectedness of $R$). Then set  $x<_w y$ iff $y\neq x$, $y\neq w$, $wRxRy$, and either  $x=w$ or $x\prec y$.

The accessibility relation induced by $<$ obviously agrees with $R_v$ (relying on the transitivity of $R$ and the fact that $x\prec y$ implies $xRy$). It is obviously irreflexive. It is well-founded by the finitude of $R(v)$.  Transitivity of $<_w$ follows from the transitivity of $\prec$ and $R$. 
And it is connected: when $w\in R(v)$: for any distinct $x,y\in R(w)$, if $x=w$ or $y=w$ then $x<_w y $ or $y<_w x$ respectively; otherwise by the totality of $\prec$ either $x\prec y$ or $y\prec x$ and hence $x<_w y$ or $y<_w x$, respectively. 

Finally, $<$ is also semi-flat.  We know that $xRy$ whenever $x<_w y$, since either $x=w$ and $wRy$, or $x\prec y$.  So we just need to show that whenever $x <_w y <_w z$, $y <_x z$. (Since $R$ is transitive, the case where $z \in R(x)\setminus R(w)$ cannot arise.)  If $x=w$ this holds trivially.  Otherwise, 
 $x<_wy<_w z$ implies $x\prec y \prec z$ and hence  $y<_x z$.

\end{proof}
A corollary of this result is that \vflat cannot be axiomatized by adding any purely modal principles (such as \textsf{4} or \textsf{H}) to \stal.  For it is easy to see that there are non-flat order frames that are reflexive, transitive and connected: the non-flat frame in \cref{flatorders1} is an example.  Such frames do not validate all of \vflat, but do validate all its purely modal theorems.  

\section{Sequentiality}
\label{sect:sequentiality}
\vflat does not include all the sentences that are valid on every $\omega$-sequence frame.  To see this, we can appeal to the fact that \vflat is sound for flat order frames (the right-to-left direction of \cref{flatt}).  Consider the  flat order model in \cref{nonsequential}.
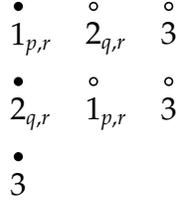
\begin{figure}[t]
\begin{center}\tikzset{every loop/.style={min distance=10mm,looseness=10}}
\begin{tikzpicture}[->,>=stealth',shorten >=1pt,shorten <=1pt, auto,node
distance=2.5cm,semithick]
\tikzstyle{every state}=[fill=gray!20,draw=none,text=black]

\node[circle,draw=black!100, fill=black!100, label=below:$1\rlap{$_{p,r}$}$,inner sep=0pt,minimum size=.1cm] (1) at (0,0) {{}};
\node[circle,draw=black!100, fill=white!100, label=below:$2\rlap{$_{q,r}$}$,inner sep=0pt,minimum size=.1cm] (2) at (1,0) {{}};
\node[circle,draw=black!100, fill=white!100, label=below:$3$,inner sep=0pt,minimum size=.1cm] (3) at (2,0) {{}};

\node[circle,draw=black!100, fill=black!100, label=below:$2\rlap{$_{q,r}$}$,inner sep=0pt,minimum size=.1cm] (1) at (0,-1) {{}};
\node[circle,draw=black!100, fill=white!100, label=below:$1\rlap{$_{p,r}$}$,inner sep=0pt,minimum size=.1cm] (2) at (1,-1) {{}};
\node[circle,draw=black!100, fill=white!100, label=below:$3$,inner sep=0pt,minimum size=.1cm] (3) at (2,-1) {{}};

\node[circle,draw=black!100, fill=black!100, label=below:$3$,inner sep=0pt,minimum size=.1cm] (1) at (0,-2) {{}};

\end{tikzpicture} 
\caption{\small A flat order model which is not equivalent to any $\omega$-sequence model, verifying (a)--(c) but not (d).}\label{nonsequential}
\end{center}
\end{figure}
It is easy to see that (a)--(c) are true at 1 while (d) is false at 1:
\begin{align*}
\text{(a)}\quad &\Box(p \to (\neg p > r)) \\
   \text{(b)}\quad &\Box(q \to (\neg q > r)) \\
   \text{(c)}\quad &p \vee q \\
   \text{(d)}\quad &\neg(p\vee q) > r
\end{align*}
But in any $\omega$-sequence model where (a)--(c) are true, (d) must be true as well.  Suppose for contradiction that (a)--(c) are true at $\sigma$ but (d) is false there. Then $\sigma$ has a first $\neg(p\vee q)$-tail, at which $r$ is false; let $n$ be the rank of that tail. Since (c) is true at $\sigma$, $n$ cannot be 0, so $\sigma^{[n-1:]}$ exists.  Either $p$ or $q$ is true at $\sigma^{[n-1:]}$.  If $p$ is, then $\sigma^{[n:]}$ is the first $\neg p$-tail of $\sigma^{[n-1:]}$, so $\neg p > r$ is false at $\sigma^{[n-1:]}$, contradicting (a).  Similarly, if $q$ is true at $\sigma^{[n-1:]}$, $\sigma^{[n:]}$ is the first $\neg q$ tail of $\sigma^{[n-1:]}$, so $\neg q > r$ is false at $\sigma^{[n-1:]}$, contradicting (b). 

The following axiom schema, which says that (d) follows from (a)--(c), is thus valid on $\omega$-sequence models:
\begin{equation}
    \ourtag{Sequentiality}
    \Box(p \to (\neg p > r)) \wedge \Box(q \to (\neg q > r)) \to ((p\vee q) \to \neg(p \vee q) > r)
\end{equation}
Let \vanf be the result of adding Sequentiality to \vflat (i.e., the smallest extension of \vflat including every instance of Sequentiality and closed under Detachment and Normality). We have just shown that \vanf is sound for $\omega$-sequence frames. It will also turn out to be complete for such frames. Hence \vanf is our sought-after logic of $\omega$-sequences.

Note that  given the validity of the 4 and T axioms in \vflat, the necessitation of Sequentiality is equivalent to the following variant which turns the final material conditional into a strict conditional:
\begin{equation*}
    \Box(p \to (\neg p > r)) \wedge \Box(q \to (\neg q > r)) \to \Box((p\vee q) \to \neg(p \vee q) > r)
\end{equation*}
This principle hence gives an equivalent axiomatization of \vflat, and shows that Sequentiality can be seen as saying that a certain property of propositions---the property of being a proposition $p$ such that $\Box(p \to (\neg p > r))$---is closed under finite disjunction.  

Since Sequentiality will play an important role in what follows, it is instructive to consider a special case of it,  where $r$ is simply the disjunction $p\vee q$.  Given our definition of $\Box$ and Normality, this is equivalent to:

\begin{equation}
    \ourtag{Restricted Sequentiality}
    \Box(p \to (\neg p > q)) \wedge \Box(q \to (\neg q > p)) \to ((p\vee q) \to \Box(p \vee q))
\end{equation}
Surprisingly, Restricted Sequentiality turns out to be equivalent to the full strength Sequentiality schema against the background of \vflat: see \cref{app:restrictedseq} for a proof.  Restricted Sequentiality can be false in flat models for the same reason as Sequentiality: in fact since $r$ is necessarily equivalent to $p\vee q$ in the model from \cref{nonsequential}, this model also demonstrates the invalidity of Restricted Sequentiality in \vflat.  And seeing why Restricted Sequentiality holds in $\omega$-sequence models  brings out  in a simple way the intuition for why these axiom schemas are valid in these models. Suppose the antecedent and $p\vee q$ are both true at $\sigma$; without loss of generality,  suppose in particular that $p$ is true at $\sigma$.  Then as we pop worlds off the beginning of $\sigma$ one at a time, if we ever reach a tail where $p$ is false, $q$ will have to be true there; as we continue popping, if we ever reach a tail where $q$ is false, $p$ will have to be true \emph{there}, and so on.  Since all the tails of $\sigma$ will eventually be reached in this process, either $p$ or $q$ (or both) is true at all of them, meaning that $\Box(p \vee q)$ is true at $\sigma$.

\subsection{Ancestral order frames}
To better understand what property of $\omega$-sequence models is responsible for the validity of Sequentiality, it is helpful to characterize the class of order frames on which \vanf is valid.  In the reasoning we just gave explaining why Restricted Sequentiality is valid over $\omega$-sequence frames, the only fact about these frames (apart from their flatness) that we relied on is  that every sequence accessible from a given sequence can be reached from it in finitely many steps, where at each step we go from a sequence to its first tail. Likewise, in the reasoning we gave explaining why Sequentiality is valid, the only fact we relied on is that every sequence accessible from a given sequence has an immediate predecessor, which, in light of the well-foundedness of ordinals, is equivalent to the first property. 
In other words, the key property of $\omega$-sequence frames which (together with flatness) explains why they validate Sequentiality, is that every accessible world from any given world can be reached from the starting world via a succession of single steps along each world's closeness ordering. More carefully:
\begin{definition}
Given an order frame $\seq{W,<}$:
\begin{itemize}
    \item
    The \emph{successor} of $w$
    is $w$ itself if $R(w)=\{w\}$, and otherwise the first world after $w$ in $<_w$.
     
\item 
$x$ is \emph{reachable} from $w$ iff $w$ is related to $x$ by the transitive, reflexive closure of the successor relation.  

\item Finally, the frame is \emph{ancestral} iff every world that is accessible from a world is also reachable from that world.

\end{itemize}
\end{definition}
It is easy to see that every $\omega$-sequence frame is ancestral as well as flat.  And while of course not every flat ancestral order frame is an $\omega$-sequence frame, we can show that every flat ancestral frame is isomorphic to an $\omega$-sequence frame, so the logic of $\omega$-sequence frames is the logic of flat ancestral frames.  

To construct an $\omega$-sequence frame isomorphic to a given flat ancestral frame $\seq{W,<}$, we will replace each world in $W$ with its \emph{successor-sequence}, defined as follows:
\begin{definition}
In an order frame $\seq{W,<}$, the \emph{successor-sequence} $\alpha_w$ of $w$ is the $\omega$-sequence starting with $w$ where each element after the first is the successor of the previous element.  
\end{definition}
The set $\{\alpha_w \mid w\in W\}$ is obviously closed under non-empty tailhood, and thus is an $\omega$-sequence frame.  A few examples should give an intuitive sense for why, when $\seq{W,<}$ is flat and ancestral, the function mapping each world to its successor-sequence is an isomorphism from the starting frame to the corresponding $\omega$-sequence frame.  Consider the frame $\seq{\{1,2,3\}, <^2}$ illustrated in the center of \cref{flatnvinduct}: replacing each member with its successor-sequence yields an $\omega$-sequence frame comprising the three sequences $\seq{1,2,3,3,3\dots}$, $\seq{2,3,3,3,3\dots}$, $\seq{3,3,3\dots}$, which is indeed isomorphic to the one we started with.  Likewise, if we start with the frame $\seq{\{1,2,3\}, <^3}$ on the right of \cref{flatnvinduct}, we get the isomorphic $\omega$-sequence frame with sequences $\seq{1,2,3,1,2,3\dots}$, $\seq{2,3,1,2,3,1\dots}$, $\seq{3,1,2,3,1,2\dots}$. 

By contrast, this procedure \emph{won't} work if we start with a non-ancestral order frame. To get a feel for why not, consider the flat but non-ancestral order frame $<^1$ in \cref{flatnvinduct}. The successor-sequence of 1 in this order frame is $\seq{1,2,1,2\dots}$, since 1 and 2 are each other's successors. The successor sequence of 3 is $\seq{3,3,3,3\dots}$, since 3 is its own successor. $\alpha_3$ is hence not accessible from (i.e., not a tail of) $\alpha_1$, even though 3 \emph{is} accessible from 1. Hence the sequence frame constructed this way is not isomorphic to the starting frame under the successor sequence function. More generally, whenever $v$ is accessible  but not reachable from $w$, $\alpha_v$ will not be accessible from $\alpha_w$.

To show that flat ancestral frames are always isomorphic to the corresponding successor-sequence frame, consider any flat ancestral frame $F=\seq{W,<}$ and $w\in W$. First we show that $u\in R(w)$ iff $\alpha_u\in R(\alpha_w)$. The fact that $F$ is ancestral means that every world accessible from $w$ is reachable from it. And the fact that $F$ is flat, and hence transitive, means every world reachable from $w$ is accessible from it. So $u\in R(w)$ just in case $u$ is reachable from $w$. But now note that, likewise, $\alpha_u$ is a tail of $\alpha_w$ iff $u$ is reachable from $w$: if $u$ is reachable from $w$, $u$ will appear in $\alpha_w$, and the truncation of $\alpha_w$ at $u$ is obviously $\alpha_u$; and if $u$ is not reachable from $w$, $u$ will never appear in $\alpha_w$, but $\alpha_u$ always begins with $u$.

Now we need to show that $x<_w y$ iff $\alpha_x\prec_{\alpha_w}\alpha_y$.
To show this, it suffices to show that $x<_w y$ iff the rank of $x$ from $w$ is less than the rank of $y$ from $w$, where we define the \emph{rank} of $x$ from $w$ as the least $n$ such that $x$ is the $n$th element of $\alpha_w$.
\bi \item[$\Rightarrow$] Suppose $x <_w y$. If $x=w$, the rank of $x$ from $w$ is 0, so the claim holds trivially.  If $x\ne w$, then by flatness $x$ also precedes $y$ according to $w$'s successor, and so on until we reach $x$.  Since $x$ cannot precede $y$ according to $<_y$, that means that as we take successor steps from $w$, we must reach $x$ before we reach $y$. 
\item[$\Leftarrow$] 
Suppose $x \nless_w y$. Then by connectedness, $y<_w x$, so by the first part of the proof, the rank of $y$ from $w$ is less than the rank of $x$ from $w$.  Thus the rank of $x$ from $w$ is not less than the rank of $y$ from $w$.
\ei

\begin{figure}[t]
\begin{center}
\tikzset{every loop/.style={min distance=10mm,looseness=10}}
\begin{tikzpicture}[->,>=stealth',shorten >=1pt,shorten <=1pt, auto,node
distance=2.5cm,semithick]
\tikzstyle{every state}=[fill=gray!20,draw=none,text=black]

\node[circle,draw=black!100, fill=black!100, label=below:$1_{}$,inner sep=0pt,minimum size=.1cm] (1) at (-4,0) {{}};
\node[circle,draw=black!100, fill=white!100, label=below:$2$,inner sep=0pt,minimum size=.1cm] (3) at (-3,0) {{}};
\node[circle,draw=black!100, fill=white!100, label=below:$3$,inner sep=0pt,minimum size=.1cm] (5) at (-2,0) {{}};

\node[circle,draw=black!100, fill=black!100, label=below:$2$,inner sep=0pt,minimum size=.1cm] (1) at (-4,-1) {{}};
\node[circle,draw=black!100, fill=white!100, label=below:$1$,inner sep=0pt,minimum size=.1cm] (3) at (-3,-1) {{}};
\node[circle,draw=black!100, fill=white!100, label=below:$3$,inner sep=0pt,minimum size=.1cm] (5) at (-2,-1) {{}};

\node[circle,draw=black!100, fill=black!100, label=below:$3$,inner sep=0pt,minimum size=.1cm] (3) at (-4,-2) {{}};

\node (4) at (-3,-3.5) {{\color{blue}{$<^1$}}};

\node (5) at (-1,-3) {{}};
\node (6) at (-1,0) {{}};
\path (5) edge[-] node {{}} (6);

\node[circle,draw=black!100, fill=black!100, label=below:$1_{}$,inner sep=0pt,minimum size=.1cm] (1) at (0,0) {{}};
\node[circle,draw=black!100, fill=white!100, label=below:$2$,inner sep=0pt,minimum size=.1cm] (3) at (1,0) {{}};
\node[circle,draw=black!100, fill=white!100, label=below:$3$,inner sep=0pt,minimum size=.1cm] (5) at (2,0) {{}};

\node[circle,draw=black!100, fill=black!100, label=below:$2$,inner sep=0pt,minimum size=.1cm] (1) at (0,-1) {{}};
\node[circle,draw=black!100, fill=white!100, label=below:$3$,inner sep=0pt,minimum size=.1cm] (5) at (1,-1) {{}};

\node[circle,draw=black!100, fill=black!100, label=below:$3$,inner sep=0pt,minimum size=.1cm] (3) at (0,-2) {{}};

\node (4) at (1,-3.5) {{\color{blue}{$<^2$}}};

\node (5) at (3,0) {{}};
\node (6) at (3,-3) {{}};
\path (5) edge[-] node {{}} (6);

\node[circle,draw=black!100, fill=black!100, label=below:$1_{}$,inner sep=0pt,minimum size=.1cm] (1) at (4,0) {{}};
\node[circle,draw=black!100, fill=white!100, label=below:$2$,inner sep=0pt,minimum size=.1cm] (3) at (5,0) {{}};
\node[circle,draw=black!100, fill=white!100, label=below:$3$,inner sep=0pt,minimum size=.1cm] (5) at (6,0) {{}};

\node[circle,draw=black!100, fill=black!100, label=below:$2$,inner sep=0pt,minimum size=.1cm] (1) at (4,-1) {{}};
\node[circle,draw=black!100, fill=white!100, label=below:$3$,inner sep=0pt,minimum size=.1cm] (3) at (5,-1) {{}};
\node[circle,draw=black!100, fill=white!100, label=below:$1$,inner sep=0pt,minimum size=.1cm] (5) at (6,-1) {{}};

\node[circle,draw=black!100, fill=black!100, label=below:$3$,inner sep=0pt,minimum size=.1cm] (3) at (4,-2) {{}};
\node[circle,draw=black!100, fill=white!100, label=below:$1$,inner sep=0pt,minimum size=.1cm] (3) at (5,-2) {{}};
\node[circle,draw=black!100, fill=white!100, label=below:$2$,inner sep=0pt,minimum size=.1cm] (5) at (6,-2) {{}};

\node (4) at (5,-3.5) {{\color{blue}{$<^3$}}};

\end{tikzpicture} 

\end{center}
\caption{Illustrations of a flat but not ancestral  frame $<^1$, and two ancestral frames $<^2, <^3$. In $<^1$ there is no way to get from 1 to 3 by taking successor steps, since you end up stuck in a loop between 1 and 2. 
}\label{flatnvinduct}
\end{figure}	

This isomorphism result both shows the interest of, and puts us in a position to prove, the following characterization result:

\begin{theorem} \vanf is valid on an order frame iff it is flat and ancestral.\label{seqchar}\end{theorem}

\begin{proof}
Given the characterization result for \vflat (\cref{flatt}), it suffices to show that a flat order frame validates
Sequentiality iff it is also ancestral. 
\bi 
\item[$\Rightarrow$] Suppose $\seq{W,<}$ is flat but not ancestral, so for some $w$ and $v$, $w$ can access but not reach $v$. Set $p$ true at exactly the worlds whose rank from $w$ is even, $q$ true at exactly the worlds whose rank from $w$ is odd, and $r$ true at exactly the worlds reachable from $w$. Then:

\bi 
\item[(a)] $w$ verifies $\Box(p\to \neg p>r)$. Consider any $p$-world $y$ accessible from $w$; by our choice of valuation, $y$ is reachable from $w$. We claim that some world reachable from $y$ has odd rank from $w$. There are  two options. 

\bi
\item 
$y$ is not the highest-ranked world reachable from $w$, so $y$'s successor is also the next world in $\alpha_w$ and hence has odd rank. 

 \item 
$y$ is the highest-ranked world reachable from $w$. Since $v$ is accessible but unreachable from $w$, $v$ is accessible but unreachable from $y$, so $y$ has a successor $u$ other than itself.  $u$'s rank from $w$ is strictly less than $y$'s, and hence $u$ must be able to reach a world $t$ with odd rank from $w$, which must therefore be reachable from $y$.

\ei 
Since $\neg p$ is true at all worlds with odd rank from $w$, it follows that $y$ can reach a $\neg p$-world.  By flatness, all worlds reachable from $y$ precede in $<_y$ any worlds unreachable from $w$.  So the first $\neg p$ world in $<_y$ is reachable from $w$, meaning that $r$ is true at it.  Thus $y$ verifies $\neg p>r$.

\item[(b)]  $w$ verifies $\Box(q\to \neg q>r)$, by parallel reasoning.

\item[(c)] $w$ verifies $p\vee q$, since $w$ verifies $p$

\item[(d)] $w$ falsifies $\neg(p\vee q)>r$, since the first accessible $\neg(p\vee q)$ worlds in $<_w$ are not reachable from $w$ and hence falsify $r$.
\ei

\item[$\Leftarrow$] Suppose $\seq{W,<}$ is flat and ancestral. Then it is isomorphic to the corresponding successor sequence frame, and we have already shown that every $\omega$-sequence frame validates Sequentiality. \qedhere
\ei
\end{proof}

Since we showed above that an order frame is isomorphic to an $\omega$-sequence frame if it is flat and ancestral, \cref{seqchar} implies that any order frame which is not flat and ancestral is not isomorphic to an $\omega$-sequence frame. So a corollary is that an order frame validates \vanf iff it is isomorphic to an $\omega$-sequence frame.\footnote{Another corollary of this result is that \vanf is sound over frames $\seq{W,\prec}$ where $W$ is a set of $\omega$-sequences \emph{not necessarily closed under tailhood}; such frames are obviously still flat and ancestral, and hence by this result validate \vanf. This shows that the assumption that $\omega$-sequence frames are closed under tailhood, while convenient in our discussion and faithful to van Fraassen's exposition, is logically irrelevant.}

Turning to soundness and completeness:
the right-to-left direction of \cref{seqchar} is equivalent to the claim that \vanf is sound for $\omega$-sequence frames. As we emphasized in the discussion before \cref{flattcom}, characterization results do not always yield corresponding completeness results. However, once more, our characterization result does indeed point the way towards a completeness result.  In \cref{app:seqcompleteness}, we prove a stronger claim which implies that \vanf is complete for $\omega$-sequence frames:
\begin{restatable}[]{theorem}{goweakcomp}
    \label{weak}
    \vanf is weakly complete for finite $\omega$-sequence frames.
\end{restatable}
This immediately implies:
\begin{restatable}[]{theorem}{golancestralcomp}
\label{thm:golancestralcomp}
    \vanf is weakly complete for finite flat ancestral order frames.
\end{restatable}

As before, the restriction to finite frames (and the fact that the question whether a finite order frame is flat and ancestral is decidable) gives us the decidability of \vanf as a corollary.

One noteworthy application of these results is  that they let us exactly pin down the purely modal fragment of \vanf.  The  modal logic \textsf{S4.3.1} is the result of adding every instance of Dum  to \textsf{S4.3}:
\[
    \ourtag{Dum}     \Box(\Box(p \to \Box p) \to p) \to (\lozenge\Box p  \to p)
\]
\textsf{S4.3.1} is sound and weakly complete for (the singleton of) the Kripke frame $\seq{\mathbb{N},\leq}$ \citep{Segerberg:1970}.
We saw in \cref{incon} that every pointed $\omega$-sequence model is equivalent to a pointed model based on the linear frame whose worlds are $\mathbb{N}$ with the order function $j<_i k$ whenever $i\leq j<k$. This order function generates the accessibility relation $\leq$. Hence a modal sentence is true in some $\omega$-sequence model iff it is true in some model on $\mathbb{N}$ with accessibility relation $\leq$, which with our soundness and completeness results yields: 
\begin{theorem}
    When $ p $ is a modal sentence, $\vdash_{\vanf} p $ iff $\vdash_{\text{\textsf{S4.3}.1}} p $.  
\end{theorem}

\subsection{Evaluating Sequentiality}

As with Flattening, we'd like to know whether Sequentiality is in fact valid for the natural language conditional connective. Here we will tentatively argue it is not,  highlighting some methodological difficulties as we go.

We'll focus on Restricted Sequentiality (repeated below), which, recall, is equivalent to Sequentiality modulo \vflat:
\[
    \ourtag{Restricted Sequentiality} 
    \Box(p \to (\neg p > q)) \wedge \Box(q \to (\neg q > p)) \to ((p\vee q) \to \Box(p \vee q))
\]
To assess an axiom schema like this, the standard methodology would be to find sentences of natural language that instantiate the schema under a plausible translation, and assess whether they strike reflective speakers as valid. But this is difficult to do in the present case. There are two options. We could unpack the $\Box$'s in Restricted Sequentiality into the left-nested conditionals that they abbreviate, and then consider corresponding translations into natural language. However, humans generally struggle to evaluate complex left-nested conditionals, and indeed, in this particular case, the resulting sentences of English are little more than word salad.

The other option is to try to find natural language expressions that express the target $\Box$ defined out of $>$. But it is controversial whether there are words of natural language that express this $\Box$. Even if there are, there is plausibly context-sensitivity \emph{both} in the interpretation of `if' \emph{and} in the interpretation of natural language necessity modals. To assess Restricted Sequentiality, we must find a natural language modal that is interpreted  \emph{throughout the particular context of our counterexample} in the target way.

There are various options to try. Here is an attempt using the adverbial phrase `in every case' as our necessity operator, which at least makes for a readable example. In particular,  we gloss $\Box(p\to q)$ as `in every case where $p$, $q$'. 

\ex.\label{jfk} \emph{The Sequential Assassin's Guild (SAG)  comprises two seasoned assassins, $p$ and $q$. They have a strict back-up policy in place: whenever $p$ kills someone, they make sure that $q$ would otherwise have carried out the assassination, and vice versa. The SAG is deciding whether to kill Nero, and will flip a fair coin to determine whether they in fact do so. Consider:}
\a. In every case where $p$ kills Nero, if $p$ hadn't done it, $q$ would have. \hspace*{\fill} $ \Box (p\to (\neg p>q))$\label{jfka}
\b. In every case where $q$ kills Nero, if $q$ hadn't done it, $p$ would have.\hspace*{\fill} $\Box (q\to (\neg q>p))$\label{jfkb}
\c.  $p$ or $q$ will kill Nero.\label{jfkc} \hfill $p\vee q$
\d. In every case, $p$ or $q$ kills Nero.\label{jfkd} \hfill $\Box (p\vee q)$

Given SAG's iron-clad back-up policies, we can be certain of \ref{jfka} and \ref{jfkb}.  \ref{jfkc} has .5 chance: it will be true just in case the fair coin lands heads. But \ref{jfkd} is certainly false, since killing Nero is a contingent matter: the coin flip could come up either way. 

But observe that \ref{jfka}--\ref{jfkc} together entail \ref{jfkd} via Restricted Sequentiality, so if Restricted Sequentiality were valid, we could not rationally be sure of \ref{jfka} and \ref{jfkb}, assign credence .5 to \ref{jfkc}, and assign credence less than .5 to \ref{jfkd}.
Hence if `in every case' and `if' are coordinated in the right way (so that `in every case, $p$' is true iff `if not $p$, then $p$' is true),  and have the same interpretations throughout \ref{jfk}, then Restricted Sequentiality is not valid. 

Similar cases are easy to multiply (we encourage readers to experiment with other modal phrases to capture the target notion of necessity).\footnote{An interesting question is whether conditional mood matters in these cases. One obstacle to assessing Sequentiality for indicatives is that material conditionals with the form $p\to\neg p>q$---or the corresponding disjunctions $\neg p \vee \neg p > q$---are generally infelicitous when $>$ is indicative, making the premises difficult to assess.}

However, it is hard to be sure that the interpretation of  `in every case' is \emph{univocally} tied to `if' in the right way throughout \ref{jfk}. An alternative hypothesis is that, in \ref{jfka} and \ref{jfkb}, we interpret `in every case' so that Nero being killed \emph{is} a necessity, given that he in fact is killed (since, had the person who killed him not done it, the other one would have), while in \ref{jfkd} we have a more expansive notion of necessity in mind. It is difficult to rule out this possibility. 

So, while (Restricted) Sequentiality has apparent counterexamples in natural language, there is plenty of room to resist them. However, it is hard to see how a clear argument could be formulated \emph{for} the validity of Sequentiality in natural language. And in general, considerations of parsimony suggest that there is no reason to adopt a logic that validates principles for which we lack a positive case. There are countless logical principles that are too complex for humans to assess; no one would advocate simply adopting them all on the grounds that there is no clear case against them.

Given the difficulties of asessing Sequentiality directly, it is natural to wonder whether there might be an axiomatization of \vanf which is easier to assess. While we cannot rule this out, we can rule out one possibility. 
As we have seen, part of what makes Sequentiality especially hard to assess is the left-nested conditionals (or, equivalently, $\Box$'s) that it contains. Could we axiomatize \vanf without left-nesting? 
The answer is no: for in the fragment of our language without left-nesting, the logic of $\omega$-sequence models is in fact \vflat. More carefully:
\begin{definition}
The \emph{Boolean language} $\mathcal{B}$ is the standard language of propositional logic:
\[  
 q ::= p_k\in At \mid \neg q \mid (q \wedge q)
\]
The \emph{Boolean-antecedent} language $\mathcal{L}_\mathcal{B>}$ is the language which adds conditionals to $\mathcal{B}$ exactly when the antecedent is conditional-free

\[  
    p   ::= p_k\in At \mid \neg p \mid (p \wedge p) \mid( q>p): q\in \mathcal{B}
\]
\end{definition}
\begin{theorem}
    \label{noleftnest}
    When $p\in \mathcal{L}_{\mathcal{B}>}$, $\vdash_\vflatp p$ iff $\vdash_\vanf p$.
\end{theorem}
\cref{noleftnest}  is proved in \cref{appendix:lba}, and suggests that the difficulties for evaluating Sequentiality which we have just encountered are ineliminable: the characterization of \vanf essentially involves left-nested conditionals (or their abbreviation in terms of the defined $\Box$).

A different, less linguistic argument against Sequentiality turns on the observation that, given the characterization result above, any world in any model  for \vanf can access at most countably many  worlds under the accessibility relation relevant to the $\Box$ defined out of $>$. If that interpretation of $\Box$ is taken seriously as a model of metaphysical or epistemic modality, or indeed any ordinary interpretation of natural language necessity modals, then this limitation to countable infinity seems implausible. Clearly, a dart can be thrown in such a way that it could hit any point on a dartboard, in any relevant sense of `could' (epistemic, metaphysical, circumstantial); but there are uncountably many points on a dartboard. And this limitation to countably many accessible worlds isn't just an artifact of the model theory. In a suitable enrichment of the language (for example, with the resources of higher order logic), one could formalize, and use Sequentiality to prove, the claim that there are only countably many points that could (in the sense of $\lozenge$) have been hit: namely, the point actually hit, the point that would have been hit if the actual one hadn't been, the one that would have been hit if neither of those had been hit, and so on.  

Together with the natural language argument given above, and the methodological case against adopting complex logical principles that lack a positive motivation, this constitutes at least a prima facie case  that Sequentiality is an undesirable artifact of $\omega$-sequence semantics---unlike Flattening, which is prima facie attractive.

\section{Ordinal sequence frames\label{generalized}}

It may be surprising that a semantics  as apparently natural as $\omega$-sequence semantics should give rise to such a peculiar logic as \vanfp, especially when the first step towards our axiomatization, namely the addition of Flattening to \stalp, is on its face much more compelling. 

On reflection, however, the class of $\omega$-sequence frames is restricted in a somewhat arbitrary way: namely, to sequences of length $\omega$. There is a natural generalization of this approach which bases the same kind of semantics on \emph{arbitrary ordinal sequences}, which turns out to be sound and complete for \vflat rather than \vanfp.  
\begin{definition}
    An \emph{ordinal sequence frame} is an order frame $\seq{W,\prec^W}$ where $W$ is a set of  (possibly transfinite, possibly finite) sequences, and $\prec^W$ is the tail order function on $W$.\footnote{Note that we do not require $W$ to be closed under nonempty tailhood, or to exclude the empty sequence---though we could impose either of these conditions without affecting the resulting logic.}
\end{definition}
\begin{theorem} \label{ordseqcompleteness}
    \vflat is sound and weakly complete for ordinal sequence models, and in fact, for ordinal sequence models which are finite, closed under non-empty tailhood, and in which every sequence has length less than $\omega^\omega$.
\end{theorem}
For soundness, it suffices to show that all ordinal sequence models are flat, which is just as for the case of $\omega$-sequences.  The completeness result is proved in \cref{app:flatcompleteness}. In fact, it is by showing the completeness of \vflat for ordinal sequence models (and appealing to the fact that these models are flat) that we prove \cref{flattcom} (the completeness of \vflat for flat order models).

As an easy corollary of \cref{ordseqcompleteness}, we get the previously stated 
\linearcom* \noindent 
The soundness part follows, as discussed in \cref{flattening1}, from the fact that all linear order frames are flat.  To derive the completeness direction from \cref{ordseqcompleteness}, we can generalize the construction in \cref{incon} where we represented an arbitrary pointed $\omega$-sequence model by an equivalent linear one based on the frame $\seq{\vec{\omega},\tails{\vec{\omega}}}$.  
Given a pointed ordinal sequence model $\mathcal{M} = \seq{\sigma,W,V}$, where the domain (length) of $\sigma$ is $\alpha$, let $\mathcal{M}'$ be $\seq{\vec{\alpha},W', V'}$, where $\vec{\alpha}$ is the identity function on $\alpha$; $W'=\{\vec{\alpha}\slice{\beta:}: \sigma\slice{\beta:}\in W\}$; and $V'(\vec{\alpha}\slice{\beta:}) = V(\sigma\slice{\beta:})$ whenever $\sigma\slice{\beta:}\in W$.  $\mathcal{M}'$ is obviously linear, and a straightforward induction on complexity shows that it is equivalent to $\mathcal{M}$.

\subsection{Validity versus consequence}
Above we showed that \emph{ancestral} flat order frames are in fact isomorphic to $\omega$-sequence frames. This means that not only do these give rise to the same set of validities, but also to the same binary consequence relation (i.e., the same set of pairs $\Gamma,p$ such that $p$ is true in every pointed model in the relevant class where $\Gamma$ is). 
This is a substantive fact, since in general, two classes of model may agree on a set of validities while disagreeing on a binary consequence relation, in particular for pairs $\Gamma, p$ where $\Gamma$ is infinite.
And indeed, it turns out that flat order frames and ordinal sequence frames do \emph{not} give rise to the same binary consequence relation, despite giving rise to the same set of validities.

For an illustration of this, consider the frame in \cref{infiniteflat}.  The domain is $\mathbb{N}$, with $j <_i k$ iff (i) $i\leq j$, and $j<k$ or  $k < i $; or (ii) $k < j < i$.  This frame is flat but not isomorphic to any ordinal sequence frame.  And indeed, equipping it with a valuation that sets each atom $p_i$ true at exactly the world $i$ results in a set of sentences which is not true in any ordinal sequence model.

\begin{figure}
\begin{center}
\tikzset{every loop/.style={min distance=10mm,looseness=10}}
\begin{tikzpicture}[->,>=stealth',shorten >=1pt,shorten <=1pt, auto,node
distance=2.5cm,semithick]
\tikzstyle{every state}=[fill=gray!20,draw=none,text=black]

\node[circle,draw=black!100, fill=black!100, label=below:$0$,inner sep=0pt,minimum size=.1cm] (0) at (-1,0) {{}};
\node[circle,draw=black!100, fill=white!100, label=below:$1$,inner sep=0pt,minimum size=.1cm] (1) at (0,0) {{}};
\node[circle,draw=black!100, fill=white!100, label=below:$2$,inner sep=0pt,minimum size=.1cm] (2) at (1,0) {{}};
\node[circle,draw=black!100, fill=white!100, label=below:$3$,inner sep=0pt,minimum size=.1cm] (3) at (2,0) {{}};
\node[circle,draw=black!100, fill=white!100, label=below:$4$,inner sep=0pt,minimum size=.1cm] (4) at (3,0) {{}};
\node[label=center:$\dots$] (1dots) at (4.5,0) {{}};

\node[circle,draw=black!100, fill=black!100, label=below:$1$,inner sep=0pt,minimum size=.1cm] (1) at (-1,-1) {{}};
\node[circle,draw=black!100, fill=white!100, label=below:$2$,inner sep=0pt,minimum size=.1cm] (1) at (0,-1) {{}};
\node[circle,draw=black!100, fill=white!100, label=below:$3$,inner sep=0pt,minimum size=.1cm] (2) at (1,-1) {{}};
\node[circle,draw=black!100, fill=white!100, label=below:$4$,inner sep=0pt,minimum size=.1cm] (3) at (2,-1) {{}};
\node[label=center:$\dots$] (2dots) at (3.5,-1) {{}};
\node[circle,draw=black!100, fill=white!100, label=below:$0$,inner sep=0pt,minimum size=.1cm] (3) at (5,-1) {{}};

\node[circle,draw=black!100, fill=black!100, label=below:$2$,inner sep=0pt,minimum size=.1cm] (1) at (-1,-2) {{}};
\node[circle,draw=black!100, fill=white!100, label=below:$3$,inner sep=0pt,minimum size=.1cm] (2) at (0,-2) {{}};
\node[circle,draw=black!100, fill=white!100, label=below:$4$,inner sep=0pt,minimum size=.1cm] (3) at (1,-2) {{}};
\node[label=center:$\dots$] (2dots) at (2.5,-2) {{}};
\node[circle,draw=black!100, fill=white!100, label=below:$1$,inner sep=0pt,minimum size=.1cm] (3) at (4,-2) {{}};
\node[circle,draw=black!100, fill=white!100, label=below:$0$,inner sep=0pt,minimum size=.1cm] (3) at (5,-2) {{}};

\node[circle,draw=black!100, fill=black!100, label=below:$3$,inner sep=0pt,minimum size=.1cm] (1) at (-1,-3) {{}};
\node[circle,draw=black!100, fill=white!100, label=below:$4$,inner sep=0pt,minimum size=.1cm] (2) at (0,-3) {{}};
\node[label=center:$\dots$] (2dots) at (1.5,-3) {{}};
\node[circle,draw=black!100, fill=white!100, label=below:$2$,inner sep=0pt,minimum size=.1cm] (3) at (3,-3) {{}};
\node[circle,draw=black!100, fill=white!100, label=below:$1$,inner sep=0pt,minimum size=.1cm] (3) at (4,-3) {{}};
\node[circle,draw=black!100, fill=white!100, label=below:$0$,inner sep=0pt,minimum size=.1cm] (3) at (5,-3) {{}};

\node[circle,draw=black!100, fill=black!100, label=below:$4$,inner sep=0pt,minimum size=.1cm] (1) at (-1,-4) {{}};
\node[label=center:$\dots$] (2dots) at (0.5,-4) {{}};
\node[circle,draw=black!100, fill=white!100, label=below:$3$,inner sep=0pt,minimum size=.1cm] (3) at (2,-4) {{}};
\node[circle,draw=black!100, fill=white!100, label=below:$2$,inner sep=0pt,minimum size=.1cm] (3) at (3,-4) {{}};
\node[circle,draw=black!100, fill=white!100, label=below:$1$,inner sep=0pt,minimum size=.1cm] (3) at (4,-4) {{}};
\node[circle,draw=black!100, fill=white!100, label=below:$0$,inner sep=0pt,minimum size=.1cm] (3) at (5,-4) {{}};

\node[label=center:$\vdots$] (2dots) at (-1,-5) {{}};

\end{tikzpicture} 

\end{center}\caption{\small  A flat order frame not isomorphic to any ordinal sequence frame.}\label{infiniteflat}
\end{figure}
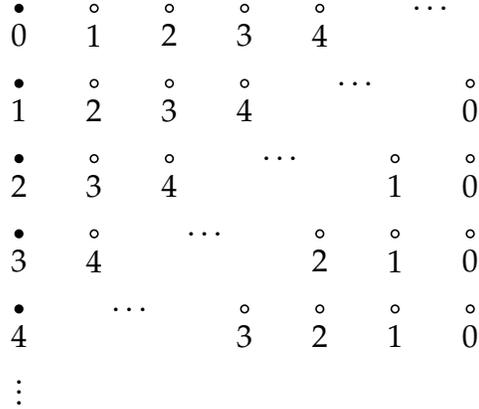
To see the obstacle that arises if we try to transform this frame into an ordinal sequence frame, think about how you would go about constructing a sequence corresponding to world 0. The natural idea would be to start with the natural numbers in order, followed by a copy of the natural numbers in reverse order. The problem, however, is that such a construction is not a sequence, since there is no ordinal with this order structure.

To make this thought precise, consider the following sets of sentences, whose union is true at 0 in this frame, given the valuation above:
\begin{align*}
\Gamma &\coloneqq  \{\neg((p_j\vee p_k)>p_k) \mid j < k\} 
\\
\Delta &\coloneqq \{p_i> \neg ((p_j \vee p_k) > p_j)\mid j < i< k \}
\\
K &\coloneqq \{p_i> \neg ((p_j \vee p_k) > p_j)\mid  j<k < i \}
\\
\Lambda &\coloneqq \{\neg\lozenge(p_i \wedge p_j)\mid i\neq j \}
\end{align*}
Now suppose for contradiction that  $\Gamma\cup\Delta\cup K\cup\Lambda$ is true at some ordinal sequence $\sigma$.
$\Lambda$ ensures no two atoms are ever true at the same tail of $\sigma$.
 $\Gamma$ ensures that some first $p_j$ tail precedes some first $p_k$ tail, whenever $j<k$.
  $\Delta$ ensures that when $i<k$, between the first $p_i$ tail and the first $p_k$ tail after it, there is no  $p_j$ tail when $j<i$. So, together, $\Gamma\cup\Delta\cup \Lambda$ ensure that $\sigma$ starts with a stretch of tails verifying just $p_0$, followed by a stretch verifying just $p_1$, followed by a stretch verifying just $p_2$, and so on (possibly interspersed with some stretches verifying no atoms at all). 
  
  Finally,  $K$ ensures that, after the first $p_i$ tail of $\sigma$, for any $j<k<i$, there is a first $p_k$ tail before any $p_j$ tail; in other words, we have a sequence that \emph{descends} towards $i$ of tails that first verify $p_{i-1}$, then $p_{i-2}$, then $p_{i-3}$, and so on, down to $p_1$ (again, possibly also with stretches verifying no atoms). But since $\Gamma\cup \Delta\cup\Lambda$ ensure that $\sigma$ starts with an \emph{ascending} sequence, so that the first $p_j$ tail always precedes the first $p_k$ tail when $j<k$, these descending sequences must come after every atom has appeared at least once.  Since the ordinals are well-founded, there is a least ordinal $\alpha$ such that some atom is true at $\sigma\slice{\alpha:}$ and every atom is true at some earlier tail in $\sigma$.  Let $p_k$ be the atom true at $\sigma\slice{\alpha:}$. Then there is no $p_{k+1}$-tail between the first $p_{k+2}$-tail and $\sigma\slice{\alpha:}$, so the conditional $p_{k+2} > \neg((p_{k}\vee p_{k+1})>p_k)$ is false.  But this is in $K$, and so our assumption that $\sigma$ verifies all these sentences leads to contradiction. 

In combination with our incompleteness result in \cref{notecomp} (which we proved for \stal, but which extends immediately to \vflat), this 
allows us to distinguish at least four, nested notions of consistency for sets of sentences, corresponding to four kinds of frame we have considered:
\begin{itemize}
    \item[(i)]
    Sets which hold in some finite ordinal sequence or flat order frame.  (A set like $\Gamma$ above is not consistent in sense (i), though it is consistent in senses (ii)--(iv).)
    \item[(ii)]
    Sets which hold in some ordinal sequence frame or linear order frame.
    \item[(iii)]
    Sets which hold in some flat order frame.
    \item[(iv)]
    Sets which hold in some \emph{generalized} flat order frame, or, equivalently, from which no contradiction can be derived in \vflat. 
\end{itemize}
All four notions of consistency correspond to \vflat and agree as regards finite sets.  

\subsection{List frames and successor-ordinal frames}

In this section we consider two interesting restrictions of ordinal sequence semantics: first, to ordinal sequence frames in which all sequences have finite length; second, to ordinal sequence frames whose sequences all have a \emph{final tail} in the domain (in a sense to be specified). 
The first class is of obvious interest for reasons of simplicity, and has been discussed in the recent literature \citep{KhooSantorio:2018,KhooBook}; the second, as we will see, is  related to the first as ordinal sequence frames are to $\omega$-sequence frames.

\begin{definition} 
    Given a non-empty set $P$, a \emph{list} over $P$ is a sequence over $P$ whose domain (or length) is a finite ordinal.  
\end{definition}
A \emph{list frame} is any ordinal sequence frame $\seq{W,\prec^W}$ in which $W$ is a set of lists closed under non-empty tailhood. 

The logic of list frames is at least as strong as \vanf, because any pointed list frame is equivalent to one whose domain is the set of tails of $\vec{n} = \seq{0,1,\ldots,n-1}$, and any such frame is 
equivalent to one whose domain is the set of tails of the $\omega$-sequence $\seq{0,1,\ldots,n-1,n-1,n-1,\ldots}$.  
In fact, the logic of lists is strictly stronger than \vanfp: it is the logic which strengthens \vanf with every instance of McKinsey:%
\footnote{Equivalently, we could add the Grzegorczyk Axiom (\emph{Grz}), $\Box(\Box(p\to\Box p)\to p)\to p$, 
a strengthening of Dum ($\Box(\Box(p\to\Box p)\to p)\to (\lozenge\Box p \to p)$), which as we pointed out is already in \vanf.  Any normal modal logic including McKinsey, Dum, and T includes Grz.   Suppose for contradiction that $\Box(\Box(p\to\Box p)\to p)$ but not $p$: then by Dum it must be that $\neg\lozenge\Box p$, i.e.\ $\Box\lozenge\neg p$, so by McKinsey $\lozenge\Box\neg p$, which by T entails
$\lozenge(\Box(p \to \Box p) \wedge \neg p)$, equivalent to the negation of $\Box(\Box(p\to\Box p)\to p)$.  Conversely, any normal modal logic including Grz and 4 includes McKinsey.  We will use the contrapositive scheme to Grz, Grz', namely $p\to \lozenge(p\wedge \Box(\neg p\to \Box\neg  p))$. 
Assume $p$ for proof by cases. Assume further $\Box\lozenge p$ for conditional proof, so by Grz' we have $\Box\lozenge p\wedge \lozenge\Box(\neg p\to\Box \neg p))$, which by 4 entails $\Box\Box \lozenge p\wedge\lozenge\Box (\neg p\to \Box\neg p)$, which entails $\lozenge(\Box\lozenge p\wedge\Box(\neg p\to \Box\neg p))$ and hence $\lozenge \Box(\lozenge p\wedge (\neg p\to \Box\neg p))$ and hence $\lozenge\Box p$.
For the second case  assume $\neg p$; by parallel reasoning we can derive  $\Box\lozenge\neg p\to \lozenge\Box\neg p$, which is equivalent to  $\Box\lozenge p\to \lozenge\Box  p$.}
\[
    \ourtag{McKinsey} 
    \Box\lozenge p\to \lozenge\Box p
\]
McKinsey is not valid in $\omega$-sequence frames: for instance, in the $\omega$-sequence model based on $\seq{1,2,1,2,\dots}$, where $p_0$ is true at $\seq{1,2,1,2,\dots}$ but false at $\seq{2,1,2,1\dots}$, $\Box\lozenge p_0$ is true at $\seq{1,2,1,2,\dots}$ while $\lozenge\Box p_0$ is false there. But McKinsey is valid in list frames, since for every list $\tau$ and set $X$ of lists containing $\tau$, $\tau$ has a shortest tail in $X$, which can have no tails in $X$ other than itself. If $\Box\lozenge p$ is true at $\tau$, every tail of $\tau$ in $X$ must have some $p$-tail in $X$, so the last  tail of $\tau$ which is in $X$ must be a $p$-tail which can only access itself. 

\begin{theorem} \label{FSMcomp}
Let \textsf{C2.FSM} be \vanf plus 
the McKinsey axiom schema. \textsf{C2.FSM} is sound and weakly complete with respect to list models.
\end{theorem}
For completeness, see \cref{app:mckinsey}.

The reasoning that shows that McKinsey is sound for list models will generalize to any ordinal sequence frame where all the sequences have a \emph{final tail} in the domain, no matter how long they are. Hence we can validate McKinsey without also validating Sequentiality.
Let a \emph{final} sequence frame be an ordinal sequence frame $\seq{W,\prec^W}$ such that whenever $\sigma\in W$, there exists a $\beta$ such that $\sigma^{[\beta:]}\in W$ and whenever $\alpha\geq \beta$ and $\sigma^{[\alpha:]}$ is defined and in $ W$, $\sigma^{[\alpha:]}$ = $\sigma^{[\beta:]}$. Natural examples are ordinal sequence frames closed under non-empty tailhood where every sequence's domain is a successor ordinal (i.e., an ordinal with a final element), as well as ordinal sequence frames whose domain includes the empty sequence (which will be a trivial final tail of every sequence).
 \begin{theorem} \label{fmcomp}
Let \textsf{C2.FM} be \vflat plus 
McKinsey. \textsf{C2.FM} is sound and weakly complete with respect to final ordinal sequence frames.
\end{theorem}
Completeness is proved in the appendix.

We have now seen four logics which are sound and complete for different classes of ordinal sequence models: namely, \textsf{C2.F}, \textsf{C2.FS}, \textsf{C2.FM}, and \textsf{C2.FSM}. And we have only brushed the surface: for every class of ordinal sequence models, we can ask whether it corresponds to a logic, potentially revealing infinitely many new interesting conditional logics.%
\footnote{Benjamin Przybocki (p.c.) has reported interesting results along these lines on the axiomatization of classes of ordinal frames in which the length limit is some ordinal strictly between $\omega$ and $\omega^\omega$.} 

\section{Does Stalnaker's Thesis motivate \vanf?}\label{probsagain}
Since van Fraassen came up with a class of models that validates both Flatness and Sequentiality as a byproduct of trying to show the non-triviality of a restricted version of Stalnaker's Thesis, one might wonder whether there is some interesting argument from some version of that Thesis to Flattening and/or Sequentiality.%
\footnote{A more direct  kind of probabilistic argument for Flattening would follow from a probabilistic argument for the validity of IE (from which Flattening follows). However, while \citet{McGee:1989} develops a model for the probabilities of conditionals that validate IE, most have not followed McGee in taking a probabilistic argument for IE seriously, since the resulting construction requires a non-classical interpretation of conditional probability, as well as an unusual conditional logic. We are interested here in whether there is a more conservative probabilistic case for Flattening and/or Sequentiality.}

Here are some definitions which will help us to characterize the multiple versions of Stalnaker's Thesis which we will need to distinguish.

\begin{itemize}

\item  A \emph{probabilistic model} for a language $\mathcal{L}$ is a model for $\mathcal{L}$ equipped with a probability function on an algebra of the model's state space which is defined on the intension of every  $p\in \mathcal{L}$.\footnote{We write $\pi(p)$ as a shorthand for $\pi(\sem{p})$, and similarly for $\pi(q\mid p)$.}

\item In particular, a \emph{Boolean probability model}  is a probabilistic model for the Boolean language based on a classical model, and a \emph{probabilistic (sequence) order model} is a probabilistic model for the full conditional language $\mathcal{L}_>$ based on a (sequence) order model. 

    \item When $\mathcal{L}'\supset \mathcal{L}$, 
    a probabilistic model for $\mathcal{L}'$ \emph{extends} a probabilistic model for $\mathcal{L}$ iff the former agrees with the latter on the probability of every sentence of $\mathcal{L}$.
    \item
    Where $\mathcal{P},\mathcal{Q},\mathcal{R}$ are any sets of $\mathcal{L}_>$-sentences, a probabilistic order model is $(\mathcal{P}>\mathcal{Q}\mid\mathcal{R})$-Stalnaker iff for any $p\in\mathcal{P}$, $q\in\mathcal{Q}$, and $r\in\mathcal{R}$, if $\prob{p}>0$ and $\cprob{r}{p} = 1$, then $\cprob{p>q}{r} = \cprob{q}{p}$.
    \item
    When $\mathcal{P}$ and $\mathcal{Q}$ are sets of sentences and $\circ$ is a logical constant, $\mathcal{P}\circ \mathcal{Q}$ is $\{p\circ q\mid p\in\mathcal{P}, q\in\mathcal{Q}\}$.  
    \item 
    $\seq{\mathcal{P}}_\circ$ is the closure of $\mathcal{P}$ under the logical constant $\circ$.  
    
\end{itemize}
 $\mathcal{B}$ is again the set of all Boolean sentences, so $\mathcal{B}>\mathcal{B}$ is the set of ``first degree'' conditionals, with Boolean antecedents and consequents, and $\firstdegconj$ is the set of conjunctions all of whose conjuncts are first degree conditionals.

We abbreviate `$(\mathcal{P}>\mathcal{Q}\mid\{\top\})$-Stalnaker' as `$(\mathcal{P}>\mathcal{Q})$-Stalnaker'.  Discussions of Stalnaker's Thesis in the literature generally focus on this special case; however, \citet{bacon:2015} and \citet{DorrBook} discuss a range of facts about the usage of conditionals that might be explained by the extra strength that comes from allowing a non-trivial class $\mathcal{R}$ of `background conditions'.\footnote{Being $(\mathcal{P}>\mathcal{Q}\mid \mathcal{R})$-Stalnaker looks a bit like a property that \citealt{lewis:1976} showed to lead to triviality given very weak assumptions. Lewis showed that no non-trivial probabilistic model always has $\pi(p>q\mid r)=\pi(q\mid pr)$ whenever $\pi(pr)>0$, when $p\in\mathcal{P}, q\in \mathcal{Q}, r\in \mathcal{R}$, for essentially \emph{any} choice of $\mathcal{P,Q,R}$. Hence the condition that $\pi(r\mid p)=1$ in our definition of $(\mathcal{P}>\mathcal{Q}\mid \mathcal{R})$-Stalnaker is crucial for any of the properties we are interested in to be prima facie interesting.}

Using this terminology, van Fraassen's tenability result can be stated as follows:
\begin{fact}[van Fraassen] \label{vftenability}
    Every countably additive Boolean probability model can be extended to a probabilistic $\omega$-sequence model that is both $(\mathcal{B}>\firstdeg)$-Stalnaker and $(\firstdeg>\mathcal{B})$-Stalnaker.  
\end{fact}
The witnessing $\omega$-sequence model is based on the frame $P^\omega$, where $P$ is the set of worlds of the starting Boolean probability model.  The extended probability function is just the standard \emph{product measure}, corresponding to $\omega$ independent copies of the starting $\pi$.  Intuitively, this corresponds to treating each element of $\omega$ like a fresh draw of a member of $P$ from an urn, with the probabilities on each draw given by $\pi$.

We begin by showing that van Fraassen's result can be strengthened in a number of ways. 
First,  van Fraassen's proof of \autoref{vftenability} depends essentially on countable additivity, which is highly controversial as a constraint on rational credence \citep[see, e.g.][]{ArntzeniusElgaHawthorneBIDB}. So it is interesting to note that the first occurrence of `countably additive' in \cref{vftenability} can be deleted.  This is explained by the following fact:
\begin{fact} \label{finiteadd}
    Every Boolean probability model can be extended to a countably additive probability model.%
    \footnote{See \cite[62--63]{YosidaHewittFAM}. Thanks to Snow Zhang (p.c.) for telling us about, and explaining, this result.}
\end{fact}
The reason for this is that one can transfer a probability function on a given Boolean algebra to its Stone representation as an algebra of sets of ultrafilters on the original algebra, and then extend the result in a unique way to a countably additive probability function on the smallest $\sigma$-algebra containing this Stone representation.  Given \cref{finiteadd}, if we are given a merely finitely additive Boolean probability model, we can first extend it to be countably additive, and then apply \cref{vftenability} to the result.  

Van Fraassen's result can be further strengthened along several dimensions.  Here is the strongest version that we know to be true (we will presently discuss some limitations on possible further strengthenings): 
\begin{restatable}[]{theorem}{staltheorem}
\label{staltheorem}
    For any countably infinite ordinal $\alpha$, every Boolean probability model can be extended to a $(\firstdegconj>\mathcal{L}_>\mid\firstdegconj)$-Stalnaker order-model based on an $\alpha$-sequence frame.
\end{restatable}
\noindent 
This is a considerably more powerful form of Stalnaker's Thesis.  We have no restriction on the consequent of the conditional; allow any conjunction of first degree conditionals as antecedents; and also allow for a `background condition' subject to the same restriction as the antecedents.  

The generalization to sequence models with sequences longer than $\omega$ is not particularly interesting in itself, but serves to demonstrate that the feature of van Fraassen's models responsible for the validity of Sequentiality---namely, the fact that the length of the sequences is $\omega$ rather than a larger transfinite ordinal---is not doing any essential work in establishing the scope of Stalnaker's Thesis in the models, meaning that no argument for Sequentiality based on Stalnaker's Thesis is in view.  This should not seem surprising: van Fraassen's proof and ours turn on the fact that when the antecedent $p$ has positive probability and is of the right form, the set of all sequences for which $p$ is true at their $n$th tail for some finite $n$ has probability 1.  Given this, the distinctions introduced by allowing transfinite sequences can be ignored.%
\footnote{The move to transfinite sequences may however acquire new relevance if we move to a theory of primitive conditional probability like Popper's \parencite*{Popper:1959}, or allow for infinitesimal probabilities.}

This fact also shows that getting the nice behaviour described in \cref{staltheorem} will not require models that strictly validate Flattening.  For while \cref{staltheorem} involves sequence models, which are flat, we can modify the order function of one of these models so that it is no longer flat but remains $(\firstdegconj>\mathcal{L}_>\mid\firstdegconj)$-Stalnaker; indeed, we can do so in such a way that certain counterexamples to Flattening are not only true at some worlds, but have positive probability. 
Those failures, however, will have to involve conditionals with rather complex antecedents (for example, a conditional whose antecedent is an instance of Sequentiality) 
in order to force us out past the initial $\omega$-sequence of worlds in the relevant sequences.  So an argument from some restriction of Stalnaker's Thesis to the truth of some corresponding restriction of Flattening still seems like an open possibility; and if such an argument could be found, it is possible that it could provide the basis for a simplicity argument for the unrestricted validity of Flattening.%
\footnote{Note that in any $(\mathcal{P}>\mathcal{Q}\mid\mathcal{R})$-Stalnaker  model, the two sides of any instance of Flattening where $p\in \mathcal{P}\cap\mathcal{R}$, $pq \in \mathcal{P}$, $r \in \mathcal{Q}$, and $pq>r \in \mathcal{Q}$ must have the same probability, so long as $pq$ has positive probability.  For then $\prob{p > (pq > r)} = \cprob{pq > r}{p} = \cprob{r}{pq} = \prob{pq > r}$.  One might see this as securing a restricted kind of ``probabilistic validity'' for the single-premise inference rules corresponding to the two directions of Flattening.  However, there is no obvious route from any consistent restriction of Stalnaker's Thesis to the claim that the instances of Flattening have probability one, even when $p$, $q$, and $r$ are Boolean.  And indeed, the tree models introduced below show that we can have $(\mathcal{B}>\mathcal{L}_>\mid\mathcal{B})$-Stalnaker models where instances of Flattening with Boolean $p,q,r$ have probability less than one.}

As it turns out, however, there is a variant of van Fraassen's result using a different kind of model that does \emph{not} validate Flattening, or anything beyond \stal, and still delivers a pretty strong form of Stalnaker's Thesis, namely $(\mathcal{B}>\mathcal{L}_>\mid\mathcal{B})$.  The idea of the variant result is to build an order model whose worlds are not \emph{sequences} of protoworlds, but infinitely branching \emph{trees} of protoworlds---structures consisting of a ``root'' protoworld together with a countable infinity of ``branches'', each of which is itself a tree.  \Cref{fig:tree} depicts the abstract structure of such a tree: a particular tree will attach a protoworld to the root and to every branching point.  
\begin{figure}
\begin{center}
    \begin{tikzpicture}[scale=21,y=-1cm,
        label distance=-4pt,
        ]
\draw [thin, blue,decoration=arrows, decorate]
l-system [l-system={mynewtree,axiom=A
,order=10,angle=90,step=0.25cm}];
\begin{scope}[circle, inner sep=0pt,pin distance=2ex]
\node at (0,0)[pin=below right:{$\scriptstyle\langle\rangle$}] {};
\node at (0.25,0)[pin=below right:{$\scriptstyle\langle 0\rangle$}] {};
\node at (0.25,-0.15)[pin=above right:{$\scriptstyle\langle 0,0\rangle$}] {};
\node at (0.16,-0.15)[pin=above left:{$\scriptstyle\langle 0,0,0\rangle$}] {};
\node at (0.16,-0.096)[pin=below left:{$\scriptstyle\langle 0,0,0,0\rangle$}] {};
\node at (0.106,-0.15)[pin=above left:{$\scriptstyle\langle 0,0,1\rangle$}] {};
\node at (0.25,-0.24)[pin=above right:{$\scriptstyle\langle 0,1\rangle$}] {};
\node at (0.196,-0.24)[pin=157.5:{$\scriptstyle\langle 0,1,0\rangle$}] {};
\node at (0.25,-0.294)[pin=above right:{$\scriptstyle\langle 0,2\rangle$}] {};
\node at (0.4,0)[pin=below right:{$\scriptstyle\langle 1\rangle$}] {};
\node at (0.4,-0.09)[pin=above right:{$\scriptstyle\langle 1,0\rangle$}] {};
\node at (0.346,-0.09)[pin=157.5:{$\scriptstyle\langle 1,0,0\rangle$}] {};
\node at (0.4,-0.144)[pin=above right:{$\scriptstyle\langle 1,1\rangle$}] {};
\node at (0.49,0)[pin=below right:{$\scriptstyle\langle 2\rangle$}] {};
\node at (0.49,-0.054)[pin=above right:{$\scriptstyle\langle 2,0\rangle$}] {};
\node at (0.544,0)[pin=below right:{$\scriptstyle\langle 3\rangle$}] {};
\end{scope}
\end{tikzpicture}
\end{center}   
\caption{\small The set $\mathbb{N}^{<\omega}$ of lists of natural numbers, depicted as the domain of a tree.}
\label{fig:tree} 
\end{figure}
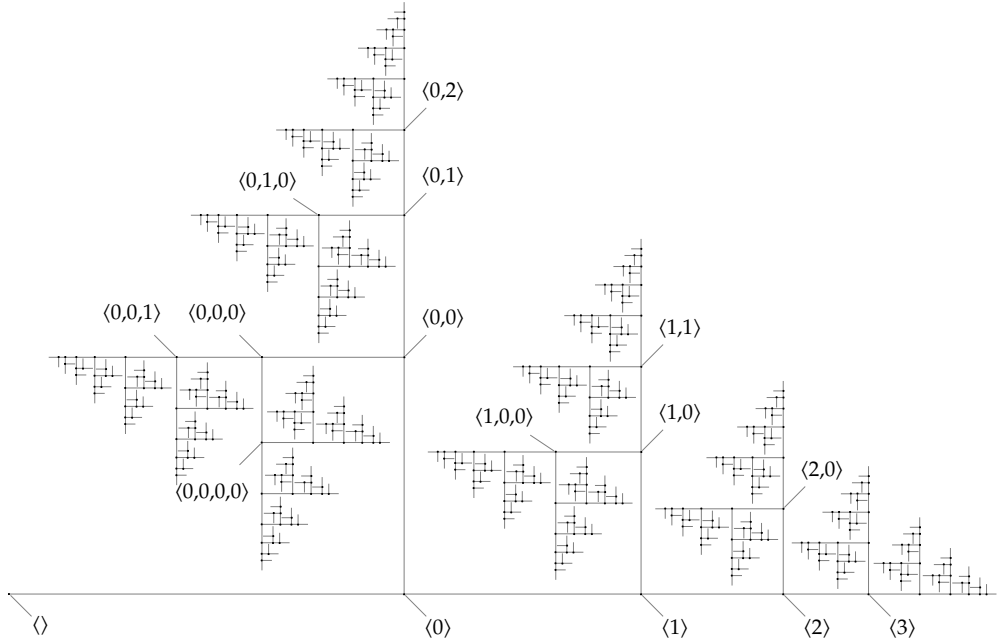

Formally, we can identify such trees with functions from \emph{lists} of natural numbers to protoworlds:
\begin{definition} \emph{Trees:}
    \begin{itemize}
        \item 
        For any set $X$, $X^{<\omega}$ is the set of lists over $X$.  
        \item 
        For any set $P$, the set of \emph{trees} over $P$ is $P^{\mathbb{N}^{<\omega}}$: the set of functions from lists of naturals to members of $P$.  
        \item        
        When $\xi$ is a tree over $P$, the \emph{root} of $\xi$ is $\xi(\seq{})$, and the \emph{$n$th branch} of $\xi$ (for any $n \in \mathbb{N}$) is the tree $\nu$ such that for any $\tau \in \mathbb{N}^{<\omega}$, $\nu(\tau) = \xi(\seq{n}+\tau)$ (where $+$ is sequence concatenation). 
        \item 
        There is a natural order function $\prec^T$ on any set $T$ of trees analogous to the tail order function on sequences: 
         $\nu \prec^T_{\xi} \omicron$ iff either  
         \begin{itemize} \item  $\nu=\xi$, $\omicron\ne\xi$, and $\omicron$ is the $n$th branch of $\xi$ for some $n$, 
         \item 
         or $\nu\ne\xi$, $\omicron\ne\xi$, and for some $n$ and $m$, $\nu$ is the $n$th branch of $\xi$,  $\omicron$ is the $m$th branch of $\xi$, and there is no $k\leq n$ such that $\omicron$ is the $k$th branch of $\xi$.  \end{itemize}
        \item 
        When $P$ is any non-empty set, a \emph{tree frame} over $P$ is an order frame whose domain is some set of trees over $P$, with the tree order function as above.
        \item 
        A \emph{full} tree frame is a tree frame whose domain consists of all trees over some $P$.
        \item
        A \emph{tree model} is an order model based on a tree frame.  
        \item 
        A \emph{categorical} tree model is a tree model where for any trees $\xi$ and $\nu$ such that $\xi(\seq{}) = \nu(\seq{})$, $\xi \in V(p)$ iff $\nu\in V(p)$ for any atom $p$.  
    \end{itemize}
\end{definition}

\begin{theorem} \label{c2trees}
    The logic of tree models is \stal.
\end{theorem}
\begin{proof}[Proof sketch]
    Given a finite order frame $\seq{W, <}$, we generate an equivalent tree frame by recursively associating each $w\in W$ with a tree $\xi_w$ over $W$,  setting $\xi_w(\seq{}) = w$ and $\xi_w(\seq{n}+\tau) = \xi_v(\tau)$, where $v$ is the world $n+1$ steps out from $w$ in $<_w$ if there is one, and otherwise, the last world in $<_w$.
    Then we can show that $u<_w v$ iff $\xi_u \prec_{\xi_w} \xi_v$.
\end{proof}
Moreover, we will have tree models where counterexamples to Flattening---even with Boolean $p,q,r$---have non-zero probability. For instance, consider a tree where $p$ is false at the root and $q$ is false at the first $p$-branch; suppose $n$ is the number of that branch. The $p,q,r$-instance of Flattening will be false if the first branch of the $n^{th}$ branch where $pq$ is true and the first branch of the starting tree where $pq$ is true do not have the same valuation for $r$. But these valuations are fully independent, and so there is non-zero probability that they indeed do differ.\footnote{Note, however, that there are \emph{some} non-theorems of \stal that always have probability 1 in tree models with product-like measures, including all the theorems of \textsf{S5}. An interesting question, which we leave open here, is how to characterize the set of sentences that always have probability 1 in such models.  Could it simply be $\mathsf{C2+S5}$?}

We can now state our tree-frame analogue of \cref{staltheorem}:
\begin{restatable}[]{theorem}{treetheorem}
\label{treetheorem}
    Every Boolean probability model can be extended to a $(\mathcal{B}>\mathcal{L}_>\mid\mathcal{B})$-Stalnaker  tree model.
\end{restatable} \noindent
In the case where the starting Boolean model $\seq{P,V,\pi}$ is countably additive, the witnessing tree model is based on the full tree frame over $P$, with the product measure (treating each node of the tree as an independent draw of a member of $P$ from an urn described by $\pi$).  

The probabilistic tree models constructed in the proof of \cref{treetheorem} are not $(\firstdegconj>\mathcal{L}_>\mid\firstdegconj)$-Stalnaker, or even $(\firstdeg>\mathcal{B})$-Stalnaker.  For example, even though $\cprob{p}{p>q}) = \prob{p\wedge p>q}/\prob{p>q} = \prob{pq}/\cprob{q}{p} = \prob{p}$ (for Boolean $p$ and $q$), $\prob{(p>q)>p}$ can be strictly less than $\prob{p}$.  We omit the tedious algebraic proof, but to get a sense for why this is, it's helpful to compare the situation with the product measure on $\omega$-sequence models, where this equality does hold (since they are $(\firstdeg>\mathcal{B})$-Stalnaker).
The $pq$ indices and the $(p>q)\wedge\negate{p}$ indices get the same probability mass in the two constructions, so we need only compare the $\neg(p>q)$ indices. The key observation is that the information that the first $(p>q)$-tail of a sequence at which $(p>q)$ is false is also a $p$-tail tells us nothing at all about the tails strictly preceding that one, beyond what we already knew from them all being $\neg(p>q)$-tails. 
By contrast, finding out that the first $p>q$ branch of a tree at which $p>q$ is false is also a $p$-branch tells us that some $p\overline q$ branch must have preceded this one (since the root verifies $\neg(p>q)$), whereas if the the first $(p>q)$-branch is a $\neg p$-branch, then the first $p\overline q$ branch of the tree can come either before or after that branch. So there are fewer $\neg(p>q)$-trees where the first $(p>q)$-branch is a $p$-branch than there are such sequences.

The behavior we just described depends crucially on the failure of Flattening in tree models.  It is natural to wonder whether this is accidental: could we construct a probabilistic order model in which Flattening is valid, and the weaker version of Stalnaker's Thesis from \cref{treetheorem} holds, but the stronger version from \cref{staltheorem} does not hold?  It turns out the answer is no:
\begin{restatable}[]{theorem}{stalextension}
    \label{stalextension}
    Suppose $\mathcal{P}$ is a class of sentences closed under Boolean operations.  Then, if a probabilistic order model is flat and $(\mathcal{P}>\mathcal{L}_>)$-Stalnaker, it is also $(\seq{\mathcal{P}>\mathcal{P}}_\wedge >\mathcal{L}_>\mid \seq{\mathcal{P}>\mathcal{P}}_\wedge)$-Stalnaker. 
\end{restatable}
\noindent
This result ties the strength of the form of Stalnaker's Thesis that van Fraassen's models validate to the extra strength of the logic that these models validate.  If one thought there were good reasons to \emph{want} a version of Stalnaker's thesis with this extra strength, that might potentially yield an interesting abductive argument for Flattening.

But it is not clear that there are good reasons for wanting this.  
The probabilistic sequence models we constructed in proving \cref{staltheorem} do not, after all, sustain the full form of Stalnaker's Thesis---they are not $(\mathcal{L}_>>\mathcal{L}_>\mid\mathcal{L}_>)$- or even $(\mathcal{L}_>>\mathcal{L}_>)$-Stalnaker. And there are some rather stringent limits on much further one could go in this direction.
\begin{restatable}[]{theorem}{limitative}
    \label{limitative}\leavevmode
   \begin{itemize}
   \item[(a)]
    No non-trivial probabilistic model interpreting $\mathcal{L}_>$ is $(\mathcal{B}>\mathcal{B}\mid\mathcal{B}\vee\firstdeg)$-Stalnaker.
 \item[(b)] \citep{StalnakerBas}
    No non-trivial probabilistic order model is $((\mathcal{B}\vee \firstdeg)>\mathcal{B})$-Stalnaker.
   \item[(c)]
    No non-trivial probabilistic flat order-model is  $((\firstdeg>\mathcal{B})>\mathcal{B})$-Stalnaker.
   \item[(d)]
    No non-trivial probabilistic $\omega$-sequence model is $((\mathcal{B}>\firstdeg)>\mathcal{B}$)-Stalnaker.\end{itemize}
\end{restatable}\noindent
Part (b) of this result is from Stalnaker's \parencite*{StalnakerBas} letter to van Fraassen; the rest is new, as far as we know.%
\footnote{We do not know whether (c) remains true if we drop the word `flat', or whether (d) remains true if we weaken `$\omega$-sequence model' to `flat order model' or just `order model'.}

The news for fans of stronger forms of Stalnaker's Thesis is not all bad. \Textcite{Fraassen:1976} already showed that one can have non-trivial $(\mathcal{L}_>>\mathcal{L}_>)$-Stalnaker models with a very weak conditional logic, and \citet{bacon:2015} strengthens this by showing that one can have non-trivial $(\mathcal{L}_>>\mathcal{L}_>\mid\mathcal{B})$-Stalnaker models with a logic encompassing all the axioms of \stal except Reciprocity.  Nevertheless, everyone will need some strategy for explaining away any \emph{prima facie} appeal of stronger versions of Stalnaker's Thesis than their preferred logic allows.  Perhaps, for example, they will appeal to some special factors that influence the resolution of context-sensitivity in such a way that conditionals embedded in the antecedents of other conditionals tend to be interpreted in some special way, maybe differently from the conditional in which they are embedded (cf. \citealt{Kaufmann:2023}).
Whatever we end up saying in response to this challenge, it is a reasonable guess that it will generalize in such a way that it could \emph{also} explain away any remaining appeal of Stalnaker's Thesis for antecedents or background conditions that are in $\firstdegconj$, but not Boolean.  

Indeed, as \citet{Kaufmann:2023} argues, it is not at all clear that sequence models make the right predictions even about conditionals with antecedents in $\firstdegconj$. Suppose you think John is very likely to go to the party, but dislikes Liam so much that he is very unlikely to go if Liam is going. Then it seems that you could reasonably think that it is \emph{also} quite unlikely that John will go to the party if Liam will go if John goes. But as we saw above, 
$(p>q)>p$ must have the same probability as $p$ in any $(\firstdeg>\mathcal{B})$-Stalnaker model (for Boolean $p$ and $q$).  So one might even see the validation of Stalnaker's Thesis in this case as a \emph{drawback} of sequence models vis-\`a-vis tree models. 

Evaluating this argument raises tricky questions about context-sensitivity which we will not try to settle here.  Our tentative conclusion is 
that the extra strength of $(\firstdegconj>\mathcal{L}_>\mid\firstdegconj)$ over $(\mathcal{B}>\mathcal{L}_>\mid\mathcal{B})$ seems unlikely to form the basis for a compelling abductive argument for Flattening.

\section{Conclusion}

Setting aside the logics of material and strict conditionals, the study of classical conditional logic has focused almost exclusively on logics of which \stal is an extension.  But we have seen in this paper that van Fraassen's models point the way to a rich array of conditional logics which properly extend \stal, without collapsing into the material conditional. Although the original motivation for these models---namely, sustaining a version of Stalnaker's Thesis---seems to do little to support the features of the models responsible for their additional logical strength, the logics they generate are nevertheless quite interesting.  Moreover, at least one of these logics, namely \vflat, enjoys considerable \emph{prima facie} plausibility for conditionals in natural language.

\newpage 

\begin{appendices}

\section{Preliminaries}
\label{app:prelim}
We begin with some further definitions that we will use throughout the appendices. 

We continue to use \emph{sequence} for any function whose domain (or `length') is an ordinal, and \emph{list} for a sequence of finite (possibly zero) length.

\begin{itemize}
\item 
$\length{\sigma}$ is the length of $\sigma$.
\item 
For $\alpha<\length{\sigma}$, $\sigma\elem{\alpha}$ is the value of $\sigma$ on $\alpha$; so e.g. $\sigma\elem{0}, \sigma\elem{1}, \sigma\elem{2}$ are the first, second, and third element of $\sigma$.

\item For $\alpha\leq\length{\sigma}$, $\sigma\slice{\alpha:}$ is the $\alpha$-tail of $\sigma$: that is, $\sigma\slice{\alpha:}{}\elem{\beta} = \sigma\elem{\alpha+\beta}$ for $\beta <\length{\sigma}-\alpha$.  

\item For $\alpha\leq\length{\sigma}$, $\sigma\slice{:\alpha}$ is the initial segment of $\sigma$ of length $\alpha$, i.e.\ the restriction of $\sigma$ to $\alpha$. 

\item For $\alpha\leq\beta \leq \length{\sigma}$, $\sigma\slice{\alpha:\beta}$ is $\sigma\slice{\alpha:}{}\slice{:\beta}$: the length $\beta-\alpha$ segment of $\sigma$ that starts at $\sigma\elem{\alpha}$. %

\item  $\sigma\elem{-1}, \sigma\elem{-2}, \ldots$ are the last, second last, \ldots elements of $\sigma$: i.e., $\sigma\elem{-n} = \sigma\elem{\alpha}$ where $\length{\sigma}=\alpha+n$, if such an $\alpha$ exists.  (It may not exist, e.g. if the length of $\sigma$ is an ordinal like $\omega$ that doesn't have a last element).  Similarly, when $\length{\sigma} =\alpha+n$, $\sigma{\slice{:-n}}$ is $\sigma\slice{:\alpha}$.  

\item $\sigma+\rho$ is the result of concatenating $\sigma$ with $\rho$, that is:
\[
(\sigma+\rho)\elem{\alpha}=
\begin{cases} 
\sigma\elem{\alpha} &  \text{when }\alpha< \length{\sigma} \\ \rho\elem{\alpha-\length{\sigma}} & \text{when }\length{\sigma}\leq\alpha< \length{\sigma}+\length{\rho}\\ \text{undefined} & \text{ otherwise} 
\end{cases}
\]
\item 
$\sigma\append{x}$ is $\sigma+\seq{x}$.  
\item $\rho$ is a \emph{segment} of $\sigma$ iff $\rho = \sigma\slice{\alpha:\beta}$ for some $\alpha$ and $\beta$; a \emph{tail} of $\sigma$ iff $\rho = \sigma\slice{\alpha:}$ for some $\alpha$; and an \emph{initial segment} of $\sigma$ iff $\rho = \sigma\slice{:\alpha}$ for some $\alpha$.  
\end{itemize}
Recall some abbreviations for our conditional language:
\begin{align*}
    \bot &\coloneqq p_0\wedge\neg p_0
    \\
    \Box p &\coloneqq \neg p > p 
    \\
    \lozenge p &\coloneqq \neg \Box \neg p\quad &&\text{(equivalent in \stal to $\neg (p>\neg p)$)} 
    \\
    p\gg q &\coloneqq \neg (p>\neg q) \quad &&\text{(equivalent in \stal to $\lozenge p\wedge p>q$)}
\end{align*}

Fix a logic \textsf{L} containing \stal; talk of consistency, entailment, equivalence, and so on throughout the appendices are relative to \textsf{L} (we will make successively stronger assumptions about \textsf{L} as we go).  

For a list of sentences $\tau$, set $\bigwedge\tau = \tau\elem{0} \wedge \cdots \wedge \tau\elem{-1}$ and $\disj\tau = \tau\elem{0} \vee \cdots \vee \tau\elem{-1}$, with $\bigwedge\seq{} = \top$ and $\disj\seq{} = \bot$.  Fix a standard ordering on $\mathcal{L}_>$, so we can extend the use of $\bigwedge$ and $\disj$ from lists to finite sets.

Our key tool throughout Appendices \ref{app:prelim}--\ref{app:mckinsey} will be a function that takes a list of sentences $\tau$ and makes a single sentence $\makesentence{\tau}$.  We define this recursively:
\begin{definition}
    The function $\makesentence{\cdot}$ is the function from lists of sentences to sentences such that:
    \begin{align*}
        \makesentence{\seq{}} &\coloneqq\top \\
        \makesentence{\tau\append{p}} &\coloneqq
        \begin{cases}
            \makesentence{\tau}\wedge (\neg\disj\tau\gg p)  &\text{if $p \neq \bot$} \\
            \makesentence{\tau}\wedge (\neg\disj\tau>p) &\text{if $p = \bot$}
        \end{cases}
    \end{align*}

\end{definition}
\begin{example}
Suppose $p, q \ne \bot$. Then:

1. $\makesentence{\seq{p}}$ is $\top \wedge (\neg\bot \gg p)$, which is \stal-equivalent to $p$.  

2. $\makesentence{\seq{p,q}}$ is $\makesentence{\seq{p}} \wedge (\neg\disj\seq{p} \gg q)$, equivalent to $p \wedge (\negate{p} > q) \wedge \Diamond\negate{p}$.  

3. $\makesentence{\seq{p,q,\bot}}$ is $\makesentence{\seq{p,q}} \wedge (\neg\disj\seq{p,q}>\bot)$, equivalent to $p \wedge (\negate{p} > q) \wedge \Diamond\negate{p} \wedge \Box(p\vee q)$.
\end{example}

$\makesentence{\cdot}$ is obviously injective since  $\makesentence{\tau\append{p}}\ne\makesentence{\seq{}}$ and if $\makesentence{\tau\append{p}} = \makesentence{\rho\append{q}}$, $\makesentence{\tau}=\makesentence{\rho}$ and $p=q$.  
Note that for $\makesentence{\tau}$ to be consistent, no element other than $\bot$ can entail any earlier element, since if $\tau \elem{k}$ entailed $\tau \elem{j}$ for $j<k$, $\neg \disj\tau\slice{:k} \wedge \tau\elem{k}$ and hence $\neg \disj\tau\slice{:k} \gg \tau\elem{k}$ would be inconsistent.  This means that for $\makesentence{\tau}$ to be consistent, $\tau$ cannot contain any inconsistent sentences other than $\bot$.  And if it does include $\bot$ at some position, every subsequent element must also be $\bot$: if $p\ne\bot$, $\makesentence{\tau+\seq{\bot,p}}$ is equivalent to $\makesentence{\tau} \wedge (\neg\disj\tau>\bot) \wedge (\neg(\disj\tau \vee \bot)\gg p)$, which is inconsistent since the second conjunct entails $\neg\disj\tau>\neg p$ while the third is equivalent to $\neg(\neg\disj\tau > \neg p)$.

The interest of this list-to-sentence operation turns on the following basic facts.
\begin{lemma}\label{basicproperty}
    If $\tau\elem{k}$ entails $pq$ and every element of $\tau\slice{:k}$ entails $\negate{p}$, then $\makesentence{\tau}$ entails $p > q$; if moreover $\tau\elem{k} \ne \bot$, $\makesentence{\tau}$ entails $p \gg q$.   
\end{lemma}
\begin{proof}
    We use the following `Cautious Monotonicity' and `Left Logical Equivalence' properties of any logic including \stal:
    \begin{align*}
    \boldtag{CMon}
    &\vdash (p > qr) \to (pq > r)
    \\
    \boldtag{CMon${\gg}$}
    &\vdash (p \gg qr) \to (pq \gg r)
    \\
    &\boldtag{LLE}
    \text{If }\vdash p\leftrightarrow q\text{ then }\vdash (p>r) \leftrightarrow (q>r)
    \\
    &\boldtag{LLE${\gg}$}
    \text{If }\vdash p\leftrightarrow q\text{ then }\vdash (p\gg r) \leftrightarrow (q\gg r)
    \end{align*}
    For the first part, note that if $\tau\elem{k}$ entails $pq$, then since $\makesentence{\tau}$ entails $\neg\disj\tau\slice{:k} > \tau\elem{k}$, it also entails $\neg\disj\tau\slice{:k} > pq$.  So by CMon, it entails $(\neg\disj\tau\slice{:k} \wedge p) > q$.  But since every member of $\tau\slice{:k}$ entails $\negate{p}$, $\neg\disj\tau\slice{:k} \wedge p$ is equivalent to $p$.  So by LLE, $\makesentence{\tau}$ entails $p>q$.  The second part is similar using CMon${\gg}$ and LLE${\gg}$.  
\end{proof}

\begin{lemma}\label{buildingup}
    If $p$ is consistent with $\makesentence{\tau}$,     and $\vdash \disj{\tau} \vee q_1\vee\cdots\vee q_n$, then either $p$ is consistent with $\makesentence{\tau\append{q_i}}$ for some $q_i$, or $p$ is consistent with $\makesentence{\tau\append{\bot}}$.
\end{lemma}
\begin{proof}
    We use the following theorem of \stal:
    \begin{equation*}
        \boldtag{$\vee$-Distribution}
        \vdash (p > (q_1\vee\cdots\vee q_n))
        \to 
        ((p > q_1) \vee \cdots \vee (p > q_n))
    \end{equation*}
    If $\vdash \disj{\tau} \vee q_1\vee\cdots\vee q_n$, then $\vdash \neg\disj\tau > (q_1\vee\cdots\vee q_n)$, so by $\vee$-Distribution, $p \wedge \makesentence{\tau}$ must be consistent with $\neg\disj\tau > q_i$ for some $q_i$.  If it is moreover consistent with $\neg\disj\tau \gg q_i$, that means that $p$ is consistent with $\makesentence{\tau\append{q_i}}$; otherwise, $p$ is consistent with $\makesentence{\tau\append{\bot}}$.    
\end{proof}
\begin{lemma} \label{putinorder}
    Suppose $X$ is a finite set of consistent sentences and $p$ and $q$ are sentences such that either $q \ne \bot$ and $p$ is consistent with $\neg\disj X \gg q$, or $q = \bot$ and $p$ is consistent with $\neg\disj X > q$.  Then there is a list $\tau$ of elements of $X$ such that $p$ is consistent with $\makesentence{\tau\append{q}}$. 
\end{lemma}
\begin{proof}
    We cover the case where $q\ne \bot$; the other case is similar.   We first show that for any list $\theta$ of members of $X$, if $p \wedge  (\neg\disj X \gg q) \wedge \makesentence{\theta}$ is consistent,
    then there is some $r$ in $X\cup\{q\}$ but not in $\theta$ such that $p \wedge  (\neg\disj X \gg q) \wedge \makesentence{\theta\append{r}}$ is consistent.  
    This follows from \cref{buildingup}, since $\bigvee(X \cup \{q, \neg q \wedge \neg\disj X\})$ is a theorem, and $\makesentence{\theta\append{r}}$ is inconsistent when $r$ is in $\theta$, while $\makesentence{\theta\append{(\neg q \wedge \neg\disj X)}}$ and $\makesentence{\theta\append{\bot}}$ are both inconsistent with $\neg\disj X\gg q$.  
    
    But if $p$ were not consistent with $\makesentence{\tau\append{q}}$ for any $\tau$ consisting entirely of members of $X$, the relevant $r$ could never be $q$, so it would have to be true that any list $\theta$ of elements of $X$ for which $p \wedge  (\neg\disj X \gg q) \wedge \makesentence{\theta}$ is consistent can be extended to a longer such list by adding some $r$ in $X$ but not in $\theta$.  This is obviously impossible, since there is at least one such list (namely the empty list), and $X$ is finite.
\end{proof}

Thanks to these nice properties, we can use our sentence-forming operation $\makesentence{\cdot}$ to define a hierarchy of `state descriptions' over a given set of atoms, where the state descriptions of a given depth $n$ consistently settle the truth value of all sentences of modal depth no greater than $n$ that can be built out of those atoms.  
\begin{definition}
    For a given logic \textsf{L} containing \stal and finite non-empty set of atoms $A$, the sets $Y_{\mathsf{L}}(A,n)$ (the ``depth-$n$ $\mathsf{L}$-state descriptions over $A$'') are defined as follows. 
    \begin{itemize}
        \item
        $Y_{\mathsf{L}}(A,0)$ is the set of all consistent conjunctions that include exactly one of $p$ and $\negate{p}$ for each atom $p\in A$.
        \item 
        $Y_{\mathsf{L}}(A,n+1)$ is the set of all consistent sentences of the form $\makesentence{\tau\append{\bot}}$, where $\tau$ is a list of elements of $Y_{\mathsf{L}}(A,n)$.
    \end{itemize}
\end{definition}
\begin{example}
    Let $p$ and $q$ be atoms.  Then $Y_{\stal}(\{p\},0)$ is $\{p,\negate{p}\}$, and $Y_{\stal}(\{p, q\},0)$ is $\{pq,p\negate{q},\negate{p}q,\negate{p}\negate{q}\}$.  
    
    $Y_{\stal}(\{p\},1)$ has four members, $\makesentence{\seq{p,\bot}}$, $\makesentence{\seq{\negate{p},\bot}}$, $\makesentence{\seq{p,\negate{p},\bot}}$, and  $\makesentence{\seq{\negate{p},p,\bot}}$, equivalent respectively to $\Box p$, $\Box\negate{p}$, $p\wedge\lozenge\negate{p}$, and $\negate{p}\wedge\lozenge p$.  
    
    $Y_{\stal}(\{p,q\},1)$ contains $\makesentence{\tau\append{\bot}}$ for each of the 64 non-empty, non-repeating lists $\tau$ of elements of $Y_{\stal}(\{p,q\},0)$.%
    \footnote{$24 = \frac{4!}{(4-4)!}$ of length 4, $24 = \frac{4!}{(4-3)!}$ of length 3, $12 = \frac{4!}{(4-2)!}$ of length 2, and $4 = \frac{4!}{(4-1)!}$ of length 1.}  
    Note that for certain logics $\mathsf{L}$ extending \stal, some of these elements would be inconsistent and hence absent from $Y_{\mathsf{L}}(\{p,q\},1)$: for example, if \textsf{L} included the axiom schema $p \to \Box p$, $Y_{\textsf{L}}(\{p,q\},1)$ would just contain the four sentences $\makesentence{\seq{pq,\bot}}$, $\makesentence{\seq{p\negate{q},\bot}}$, $\makesentence{\seq{\negate{p}q,\bot}}$, and $\makesentence{\seq{\negate{p}\negate{q},\bot}}$.

    Each member of $Y_{\stal}(\{p\},2)$ is of the form $\makesentence{\tau\append{\bot}}$ for some list $\tau$ of elements of $Y_\stal(\{p\}, 1)$; but not every non-repeating, non-empty $\tau$ can appear in this role.  If $\tau$ begins with $\makesentence{\seq{p,\bot}}$, it cannot contain any elements entailing $\negate{p}$, so the only other element that could appear is $\makesentence{\seq{p,\negate{p},\bot}}$; similarly, if $\tau$ begins with $\makesentence{\seq{\negate{p},\bot}}$, the only other element that can appear is $\makesentence{\seq{\negate{p},p,\bot}}$.  Meanwhile, if $\tau$ begins with $\makesentence{\seq{p,\negate{p},\bot}}$, at least one element entailing $\negate{p}$---either $\makesentence{\seq{\negate{p},\bot}}$ or $\makesentence{\seq{\negate{p}, p,\bot}}$---must appear later in $\tau$; similarly, if it begins with $\makesentence{\seq{\negate{p},p,\bot}}$, an element entailing $p$ must appear later.  32 lists meet these constraints.%
    \footnote{Two beginning with $\makesentence{\seq{p,\bot}}$, two beginning with $\makesentence{\seq{\negate{p},\bot}}$, 14 ($=\frac{3!}{(3-3)!} + \frac{3!}{(3-2)!} + 2$) beginning with $\makesentence{\seq{p,\negate{p},\bot}}$, and 14 beginning with $\makesentence{\seq{\negate{p},p,\bot}}$.}
        
    More generally: where $\rho = \seq{\makesentence{\tau_0},\ldots, \makesentence{\tau_n}}$ is a non-repeating list of elements of $Y_{\stal}(A,n)$, $\makesentence{\rho\append{\bot}}$ is consistent in \stal (and hence a member of $Y_{\stal}(A,n+1)$) only if $\tau_0$ is the result of deleting all but the first occurrence of each element in $\seq{\tau_0\elem{0},\ldots,\tau_n\elem{0}}$.%
    \footnote{In fact this is the only constraint.  Given disjoint pointed order models $\mathcal{M}_1\ldots \mathcal{M}_n$ for $\makesentence{\tau_1},\ldots,\makesentence{\tau_n}$, we can construct a new pointed order model by taking the union $\mathcal{M}_1\ldots \mathcal{M}_n$ together with one new world $w_0$---the distinguished world---where $<_{w_0}$ comprises $w_0$ followed by the distinguished worlds of $\mathcal{M}_1\ldots \mathcal{M}_n$, and atom $p_i$ is true at $w_0$ iff entailed by $\makesentence{\tau_0}$.  Then so long as $\rho$ obeys the given constraint, $\makesentence{\tau_0}$ and hence also $\makesentence{\rho\append{\bot}}$ are true at $w_0$.}
\end{example}
\begin{lemma} \label{statesarestates}
    If $s\in Y_{\mathsf{L}}(A,n)$ and $p$ is a sentence of modal depth $\leq n$ with atoms from $A$, then either $s$ entails $p$ in \stal or $s$ entails $\negate{p}$ in \stal. 
\end{lemma}
\begin{proof}
    By induction on $n$.  The base case holds since the elements of $Y_{\mathsf{L}}(A,0)$ settle the truth value of every atom in $A$, hence every Boolean combination of atoms in $A$.  For the induction step, it suffices to show that when $s \in Y_{\mathsf{L}}(A,n+1)$ and $p$ and $q$ have modal depth $\leq n$, $s$ entails one of $p>q$ and $\neg(p>q)$ in \stal.  Any such $s$ will be of the form $\makesentence{\tau\append{\bot}}$ where each element of $\tau$ is in $Y_{\mathsf{L}}(A,n)$.  Suppose that the first element of $\tau\append{\bot}$ that entails $p$ also entails $q$.  Then since no previous element entails $p$, all of them entail $\negate{p}$ by the induction hypothesis; so by \cref{basicproperty}, $\makesentence{\tau\append{\bot}}$ entails $p>q$.  Otherwise, the first element of $\tau\append{\bot}$ that entails $p$ does not entail $q$.  This element must be $\tau\elem{j}$ for some $j$, since $\bot$ does entail $q$.  By the induction hypothesis, $\tau\elem{j}$ entails $\negate{q}$ and all elements of $\tau\slice{:j}$ entail $\negate{p}$, so by \cref{basicproperty}, $\makesentence{\tau\append{\bot}}$ entails $p\gg \negate{q}$ and hence $\neg(p>q)$.  
\end{proof}
\begin{lemma} \label{exhaustivestates}
    If $p$ is consistent in $\mathsf{L}$, it is consistent with some element of $Y_{\mathsf{L}}(A,n)$ for every $A$ and $n$.
\end{lemma}
\begin{proof}
    By induction on $n$.  The base case holds since $\disj Y_{\mathsf{L}}(A,0)$ is a tautology.  For the induction step, we note that since $\disj Y_{\mathsf{L}}(A,n)$ is a theorem by the induction hypothesis, $p \wedge (\neg\disj Y_{\mathsf{L}}(A,n) > \bot)$ is consistent whenever $p$ is, so by \cref{putinorder}, there is a list $\tau$ of elements of $Y_{\mathsf{L}}(A,n)$ such that $p$ is consistent with $\makesentence{\tau\append{\bot}}$.  $\makesentence{\tau\append{\bot}}$ is our desired element of $Y_{\mathsf{L}}(A,n+1)$.  
\end{proof}

\Cref{statesarestates,exhaustivestates} together imply that if want to show that every \textsf{L}-consistent sentence has an order model of a certain sort, it suffices to show that every member of every $Y_{\mathsf{L}}(A,n)$ has a model of that sort.  For when $p$ has modal depth $n$ and atoms from $A$, it will be \textsf{L}-consistent with some $s\in Y_{\mathsf{L}}(A,n)$ by \cref{exhaustivestates}, hence entailed in \stal by this $s$ by \cref{statesarestates}, hence true in any model where $s$ is true (by the soundness of \stal for order models), and hence true in some model since $s$ is.

\section{\stal is weakly complete for finite order models}

We now turn to our first result: \stal is weakly complete for finite order models. While this result is not new (or at least, is part of the conditionals folklore), proving it now provides an opportunity to showcase the use of some of the definitions from the previous section, which will also be needed for the later completeness theorems for logics which strengthen \stal.

\begin{definition}
    For a given finite set of atoms $A$ and natural number $n$, we define an order model $\mathcal{M}_{A,n}=\seq{W,<,V}$: 
    \begin{itemize}
        \item 
        $W$ is $\bigcup_{m\leq n}Y_{\text{\stal}}(A,m)$: all of the \stal-state descriptions over $A$ of depth no greater than $n$.  
        \item 
        When $s$ is a depth-0 state description, $t<_su$ is never true (so $R(s) = \{s\}$).  When $s = \makesentence{\tau\append{\bot}}$ is a state description of positive depth,
        $t <_s u$ iff for some $i>0$, $u=\tau{\elem{i}}$, and either $t=s$ or $t$ is in $\tau\slice{1:i}$.
        \item  $V(p_i,w)=1$ iff $s$ \stal-entails $p_i$.  (Atoms not in $A$ are thus false everywhere.) 
    \end{itemize}
\end{definition}

\begin{lemma}\label{trueat}
    Each world in $\mathcal{M}_{A,n}$ is true at itself.  
\end{lemma}
\begin{proof}
    We prove, by induction on complexity, that for every sentence $p$ with atoms in $A$, and every state description $s$ whose depth is at least that of $p$, $p$ is true at $s$ in $M_{A,n}$ iff $s$ entails $p$.  (The claim then follows as the special case where $p=s$.)

    \bi 
    \item
    \emph{Atoms:} immediate from the definition of $\mathcal{M}_{A,n}$.

   \item \emph{Conjunction:} obvious.

   \item \emph{Negation:} If $\neg p$ is true at $s$, then $p$ is not true at $s$, so by the induction hypothesis, $s$ does not entail $p$.  Since $p$'s modal depth is no greater than the depth of $s$, it follows by \cref{statesarestates} that $s$ entails $\neg p$. Conversely, if $s$ entails $\neg p$, then since $s$ is consistent, $s$ does not entail $p$, so $p$ is false at $s$ by the induction hypothesis, so $\neg p$ is true at $s$.  

  \item  \emph{Conditional:}  Suppose $s = \makesentence{\tau\append{\bot}}$ is depth $m$ state-description, and $p>q$ is a conditional of modal depth $\leq m$, meaning that $p$ and $q$ must have modal depth $<m$.

\bi \item[(i)]  First suppose $p>q$ is false at $s$.  Then there is some $u\in R(s)$ such that $p$ and $\neg q$ are true at $u$, while $\neg p$ is true at $t$ whenever $t <_s u$.  Since every member of $R(s)$ is a depth $m$ or depth $m-1$ state description, the induction hypothesis implies that $u$ entails $p \wedge  \neg q$ while every $t$ such that $t <_s u$ entails $\neg p$.  If $u=s$, $s$ does not entail $p>q$ (since if it did it would be inconsistent, by MP).  Otherwise, $u = \tau\elem{i}$ for some $i\geq 1$, and we have that $\tau\elem{j}$ entails $\neg p$ for all $j< i$---including $j=0$, since $\tau\elem{0}$ entails all the depth $< m$ sentences $s$ entails. So by \cref{basicproperty}, $s$ entails $p\gg \neg q$.  Since $s$ is consistent, it does not entail $p>q$.  

\item[(ii)] Next, suppose $p>q$ is true at $s$.  Then there are two cases: either $p>\neg q$ is false at $s$, or $\Box \neg p$ is true at $s$.  In the former case, by part (i), $s$ does not entail $p>\neg q$.  Since $p>\neg q$ has modal depth $\leq m$, it follows by \cref{statesarestates} that $s$ entails $\neg(p>\neg q)$ and hence also $p>q$ (by CEM).  In the latter case, $\neg p$ is true at every world in $R(s)$, so by the induction hypothesis, all of these worlds entail $\neg p$.  Hence every member of $\tau$ entails $\neg p$.  ($\tau\elem{0}$ does too because it agrees with $s$ on depth $< m$ sentences.)  But $s$ entails $\neg\bigvee\tau > \bot$, so by CMon, $s$ entails $p>\bot$ and hence also $p>q$.  \qedhere
   \ei
\ei

\end{proof}
Given Lemmas \ref{statesarestates} and~\ref{exhaustivestates}, this result immediately establishes: 

\golcompleteness*

\section{Completeness for \vflat}
\label{app:flatcompleteness}
The model construction in the previous section gave a procedure for turning a \stal-state description into a pointed order model in which that state description is true.  The order models generated by this description are very far from being flat: in fact, when $w\ne v$ and $v\in R(w)$, $R(w)$ and $R(v)$ have just one common member, namely $v$. In this section, we will prove a completeness theorem for \vflat, based on a procedure for turning a \vflat-state description into a pointed ordinal sequence model (a fortiori, a flat order model) in which it is true.  This will take a bit more work, though the basic intuitions are similar.
We build on our earlier definitions, now assuming that our underlying logic \textsf{L} includes $\vflat$:
\begin{definition} 
    Where $\tau$ is any list of sentences:
    \begin{itemize}
        \item 
        $\tau$ is \emph{orderly} iff $\makesentence{\tau}$ is consistent, $\tau$ is non-empty, any two elements of $\tau$ are jointly inconsistent, and $\bot$ never occurs in $\tau$ except possibly as the final element.
        \item 
        $\tau$ is \emph{direct} iff $\tau$ is orderly and there is an orderly list of elements of $\tau$ that has the same last element as $\tau$ and includes at least two elements of $\tau$, but does not include every element of $\tau$.  
        \item 
        $\tau$ is \emph{circuitous} iff $\tau$ is orderly, not direct, and has length at least 3.
    \end{itemize}
\end{definition}

\begin{example}\label{cirqexample}
Let \textsf{L} be \vflat, let $a\coloneqq pq, b\coloneqq \negate{p}q, c \coloneqq p\negate{q}$ (where $p$ and $q$ are atoms), and consider the following length-6 list:
\begin{align*}
    \eglist &\coloneqq \seq{\underset{\eglist\elem{0}}{\makesentence{\seq{a, b, c, \bot}}}, \underset{\eglist\elem{1}}{\makesentence{\seq{b, a, c, \bot}}}, \underset{\eglist\elem{2}}{\makesentence{\seq{c, a, b, \bot}}}, \underset{\eglist\elem{3}}{\makesentence{\seq{a, c, \bot}}}, \underset{\eglist\elem{4}}{\makesentence{\seq{c, \bot}}}, \underset{\eglist\elem{5}}{\vphantom{\makesentence{\seq{}}}\bot}}
\end{align*}
$\makesentence{\eglist}$ is consistent, as can be seen by considering the ordinal sequence model whose domain comprises the following length $\omega+5$ sequence over $\{a,b,c\}$ and its non-empty tails---seven sequences in all---with the valuation where an atom is true at a sequence iff it is entailed by its first element:
\[
\seq{a,b,a,b,\ldots,c,a,b,a,c}
\]
(In fact, $\makesentence{\eglist}$ is an element of $Y_{\text{\vflat}}(\{p,q\},2)$---a depth-2 state-description over those atoms.)  Since no two elements of $\eglist$ are consistent, and $\bot$ occurs only as the final element, $\tau$ is thus orderly.  In fact $\eglist$ is \emph{direct}, since $\seq{\eglist\elem{4},\bot}$ is consistent: $\makesentence{\seq{\eglist\elem{4},\bot}}$ is true at the final tail $\seq{c}$ in the above model.  

The initial segments of $\eglist$ are also orderly, since obviously any non-empty initial segment of an orderly list is orderly.  $\eglist\slice{:5}$ is direct, since $\seq{\eglist\elem{3},\eglist\elem{4}}$ is orderly (as witnessed by the tail $\seq{a,c}$).  $\eglist\slice{:4}$ is also direct, since $\seq{\eglist\elem{1},\eglist\elem{3}}$ is orderly (as witnessed by the tail $\seq{b,a,c}$).  $\eglist\slice{:3}$, by contrast, is circuitous, since neither $\seq{\eglist\elem{0},\eglist\elem{2}}$ nor $\seq{\eglist\elem{1},\eglist\elem{2}}$ is orderly. ($\eglist\elem{0}$ entails $\negate{a}\gg b$ and hence $\negate{\eglist\elem{0}}\gg b$ (by CMon), which is inconsistent with $\negate{\eglist\elem{0}}\gg \eglist\elem{2}$, and likewise $\eglist\elem{1}$ entails $\negate{b}\gg a$ and hence $\negate{\eglist\elem{1}}\gg a$ which is inconsistent with $\negate{\eglist\elem{1}}\gg \eglist\elem{2}$.)  Finally, $\eglist\slice{:2}$ and $\eglist\slice{:1}$ are orderly but neither direct nor circuitous, since their lengths are less than 3.

\end{example}

The key new facts secured by adding Flattening are given in the following three lemmas, all of which assume that the logic \textsf{L} includes \vflat.
\begin{lemma}\label{backflattening}
    Suppose $\tau$ is orderly, and every element of  $\tau\slice{:-1}$ entails $\negate{q}$.  Then if $\tau\elem{-1}$ entails $q>r$, every element of $\tau\slice{:-1}$ is consistent with $q>r$. And if $\tau\elem{-1}$ is consistent and entails $\neg(q>r)$, every element of $\tau\slice{:-1}$ is consistent with $\neg(q>r)$.
\end{lemma}
\begin{proof}
    Suppose $k<\length{\tau}-1$, every member of $\tau\slice{:-1}$ entails $\negate{q}$, and $\tau\elem{-1}$ entails $q>r$.
    $\makesentence{\tau}$ entails 
    $\makesentence{\tau\slice{:k+1}}$ and $\neg\disj\tau\slice{:-1} > \tau\elem{-1}$, and hence $\makesentence{\tau\slice{:k+1}} \wedge (\neg\disj\tau\slice{:-1} > (q>r))$. Since $q$ entails $\neg\disj\tau\slice{:-1}$ and $\neg\disj\tau\slice{:k}$, $\neg\disj\tau\slice{:-1} > (q>r)$ and $\neg\disj\tau\slice{:k} > (q>r)$ are both equivalent to $q>r$ by the Flattening Rule, hence equivalent to each other.  Hence, $\makesentence{\tau\slice{:k+1}} \wedge (\neg\disj\tau\slice{:k} > (q>r))$ is consistent.  Since $\tau\elem{k}\ne\bot$, $\makesentence{\tau\slice{:k+1}}$ is $\makesentence{\tau\slice{:k}} \wedge (\neg\disj\tau\slice{:k} \gg \tau\elem{k})$.  So we can conclude that $\neg\disj\tau\slice{:k} \gg (\tau\elem{k} \wedge (q>r))$ is consistent.  But for this to be the case, $\tau\elem{k} \wedge (q>r)$ must also be consistent.

    The case where $\tau\elem{-1}$ is consistent and entails $\neg(q>r)$ is similar.  Then, $\makesentence{\tau}$ entails $\neg\disj\tau\slice{:-1}\gg \tau\elem{-1}$ and hence $\makesentence{\tau\slice{:k+1}} \wedge (\neg\disj\tau\slice{:-1}\gg \neg(q>r))$, i.e.\ $\makesentence{\tau\slice{:k+1}} \wedge \neg (\neg\disj\tau\slice{:-1}>(q>r))$.  By the same reasoning as above, this is equivalent to $\makesentence{\tau\slice{:k+1}} \wedge \neg (\neg\disj\tau\slice{:k}>(q>r))$, i.e.\  $\makesentence{\tau\slice{:k}}\wedge (\neg\disj\tau\slice{:k}\gg \tau\elem{k}) \wedge (\neg\disj\tau\slice{:k}\gg\neg (q>r))$.  This is consistent only if $\tau\elem{k} \wedge \neg(q>r)$ is.  
\end{proof}
\begin{lemma} \label{endapproach}
    Suppose $\tau$ is orderly.  Then for each element $\tau\elem{k} \ne \bot$, there is an orderly list $\rho$ of elements of $\tau$  such that $\rho\elem{0} = \tau\elem{k}$, $\rho\elem{-1} = \tau\elem{-1}$, and the elements of $\tau\slice{k+1:}$ all occur, in the same order, in $\rho$.
\end{lemma}
\begin{proof}
    Consider the case where $\tau\elem{-1}\ne\bot$.  Then for each $k<\length{\tau}$, the conjunction
    \[
    (\neg\disj \tau\slice{:k} \gg \tau\elem{k}) \wedge (\neg\disj \tau\slice{:k+1} \gg \tau\elem{k+1}) \wedge \cdots \wedge (\neg\disj\tau\slice{:-1} \gg \tau\elem{-1})
    \]
    is consistent, since all its conjuncts are conjuncts of $\makesentence{\tau}$.  Since $\neg\disj\tau\slice{:j} \vdash \neg\disj\tau\slice{:k}$ for all $j>k$, the $\gg$-Flattening Rule says that this is equivalent to
    \[
    \neg\disj \tau\slice{:k} \gg (\tau\elem{k} \wedge (\neg\disj\tau\slice{:k+1} \gg \tau\elem{k+1}) \wedge \cdots \wedge (\neg\disj\tau\slice{:-1} \gg \tau\elem{-1}))
    \]    
    which is thus also consistent.  Since $p\gg q$ is consistent only when $q$ is, we can conclude that 
    \[ \tag*{(*)} \label{thisisconsistent}
    \tau\elem{k} \wedge (\neg\disj\tau\slice{:k+1} \gg \tau\elem{k+1}) \wedge \cdots \wedge (\neg\disj\tau\slice{:-1} \gg \tau\elem{-1})
    \]    
    is consistent too.  But then, by \cref{putinorder} (setting $p$ in that lemma to be $ \tau\elem{k} \wedge (\neg\disj\tau\slice{:k+1} \gg \tau\elem{k+1}) \wedge \cdots \wedge (\neg\disj\tau\slice{:-2} \gg \tau\elem{-2})$, $q$ to be $\tau\elem{-1}$, and $X$ to be the set of elements of $\tau\slice{:-1}$),  there must be a list $\rho$ of elements of $\tau$, ending with $\tau\elem{-1}$, such that the conjunction of \ref{thisisconsistent} with $\makesentence{\rho}$ is consistent. Clearly, since the elements of $\tau$ are pairwise inconsistent, this conjunction can only be consistent if $\rho\elem{0}$ is $\tau\elem{k}$.  Moreover, the elements of $\tau\slice{k:}$ must occur in $\rho$ in the same order as in $\tau\slice{k:}$.  For suppose the first element of $\tau\slice{k+m:}$ in $\rho$ is $\tau\elem{k+m+j}$.  That means $\makesentence{\rho}$ has a conjunct of the form $p\gg\tau\elem{k+m+j}$, where $p$ is some disjunction of elements of $\tau\slice{:k+m}$.  By CMon, this entails $\neg\bigvee\tau\slice{:k+m}\gg\tau\elem{k+m+j}$, which is consistent with the conjunct  $\neg\bigvee\tau\slice{:k+m}\gg\tau\elem{k+m}$ of \ref{thisisconsistent} only if $j=0$.  
    
    The case where $\tau\elem{-1} = \bot$ is parallel.  
\end{proof}

\begin{lemma} \label{circuitous} 
    Suppose $\tau$ is circuitous and $t$ is in $\tau\slice{:-1}$.  Then there is an orderly list that begins with $t$, ends with $\tau\elem{-1}$, and contains all and only the elements of $\tau$.  
\end{lemma}
\begin{proof}
    By \cref{endapproach} there is an orderly list $\rho$ beginning with $t$, ending with $\tau\elem{-1}$, and containing only elements of $\tau$.  But since $\tau$ is circuitous, any such list must contain every element of $\tau$.
\end{proof}

We will now describe a function that takes any orderly list $\tau$ and returns a (possibly repeating, possibly transfinite) sequence $\makeseq{\tau}$, such that the elements of $\makeseq{\tau}$ are exactly the elements of $\tau\slice{:-1}$, and the order of their first occurrences in $\makeseq{\tau}$ is the same as their order in $\tau$.

\begin{definition}\label{makeseqdef}
    We define $\makeseq{\tau}$ recursively, based on the length of the orderly list $\tau$.  

    For the base cases, when the length of $\tau$ is 1 or 2, $\makeseq{\tau} \coloneqq\tau\slice{:-1}$: that is, $\seq{}$ if $\length{\tau}=1$ and $\seq{\tau\elem{0}}$ if $\length{\tau}=2$.
        
    For the recursion step, when $\length{\tau}>2$, there are two cases, depending on whether $\tau$ is direct or circuitous.
    \begin{itemize}
        \item 
        Case 1: $\tau$ is direct, so there is an orderly list of elements of $\tau$ with the same last element as $\tau$ and length strictly between 1 and $\length{\tau}$.  Let $j$ be the greatest number such that $\tau\elem{j}$ is the first element of such a list, and let $\rho$ be such a list beginning with $\tau\elem{j}$.  (If there are multiple appropriate lists beginning with $\tau\elem{j}$, choose $\rho$ to be the first one according to some fixed order on lists of sentences.)  Also, from \cref{endapproach}, we know that there is an orderly list of elements of $\tau$ that begins with $\tau\elem{-2}$ and ends with $\tau\elem{-1}$.  Let $\pi$ be such a list (the first one, according to our fixed order).  $\tau\elem{j}$ must occur somewhere in $\pi$, since if $j=\length{\tau}-2$ it is the first element of $\pi$, while if $j<\length{\tau}-2$, that means that $\pi$ must have the same length as $\tau$ and hence contain every element of $\tau$.  Let $\theta$ be the proper initial segment of $\pi$ up to and including the first occurrence of $\tau\elem{j}$.  Since $\tau\slice{:-1}$, $\theta$, and $\rho$ are all orderly and shorter than $\tau$, we may assume that $\makeseq{}$ is already defined on them, and define 
        \[
            \makeseq{\tau} \coloneqq\makeseq{\tau\slice{:-1}} + \makeseq{\theta} + \makeseq{\rho}
        \]
        Note that in the case where $j=\length{\tau}-2$, $\theta$ is just $\seq{\tau\elem{j}}$, so $\makeseq{\theta}$ is the empty list and $\makeseq{\tau} = \makeseq{\tau\slice{:-1}} + \makeseq{\rho}$.
        \item 
        Case 2: $\tau$ is circuitous.  Then by \cref{circuitous}, for each element $t$ of $\tau\slice{:-1}$, there is an orderly list that begins with $t$, ends with $\tau\elem{-1}$, and contains exactly the elements of $\tau$.  Define a function $\Psi$ such that for each $t$ in $\tau\slice{:-1}$, $\Psi(t)$ is such a list: $\tau$ itself if $t$ is $\tau\elem{0}$; otherwise, the first such list according to our fixed order.  Let $\chooselist(t) \coloneqq \Psi(t)\slice{:-1}$ (so each $\chooselist(t)$ is shorter than $\tau$), and $g(t) \coloneqq \Psi(t)\elem{-2}$.  Then define:
        \begin{equation*}
            \makeseq{\tau} \coloneqq
            \makeseq{\chooselist(\tau\elem{0})} 
            + \makeseq{\chooselist(g(\tau\elem{0}))}
            + \makeseq{\chooselist(g(g(\tau\elem{0})))}+ \makeseq{\chooselist(g(g(g(\tau\elem{0}))))}
            + \cdots
        \end{equation*}
    \end{itemize}
\end{definition}

\begin{example} 
    Fixing \textsf{L} as \vflat, let us compute $\makeseq{\eglist}$ where $\eglist$ is the list from \cref{cirqexample}:
    \[
        \eglist \coloneqq \seq{\underset{\eglist\elem{0}}{\makesentence{\seq{a, b, c, \bot}}}, \underset{\eglist\elem{1}}{\makesentence{\seq{b, a, c, \bot}}}, \underset{\eglist\elem{2}}{\makesentence{\seq{c, a, b, \bot}}}, \underset{\eglist\elem{3}}{\makesentence{\seq{a, c, \bot}}}, \underset{\eglist\elem{4}}{\makesentence{\seq{c, \bot}}}, \underset{\eglist\elem{5}}{\vphantom{\makesentence{\seq{}}}\bot}}
    \]
    As we noted in \cref{cirqexample}, $\eglist$ is direct, so we are in case 1.  Since $\seq{\eglist\elem{4},\bot}$ is orderly, $j=4=\length{\eglist}-2$ and $\rho = \seq{\eglist\elem{4},\bot}$  so 
    \[
        \makeseq{\eglist} = \makeseq{\eglist\slice{:5}} + \makeseq{\seq{\eglist\elem{4},\bot}}  = \makeseq{\eglist\slice{:5}} \append{\eglist\elem{4}}
    \]
    using the base case for length-2 lists for the second identity.  Proceeding to calculate $\makeseq{\eglist\slice{:5}}$, we note that since $\seq{\eglist\elem{3},\eglist\elem{4}}$ is orderly, $\eglist\slice{:5}$
    is also direct, $j=3=\length{\eglist\slice{:5}}-2$, $\rho = \seq{\eglist\elem{3},\eglist\elem{4}}$, so
    \[
        \makeseq{\eglist\slice{:5}} = \makeseq{\eglist\slice{:4}} + \makeseq{\seq{\eglist\elem{3},\eglist\elem{4}}} = \makeseq{\eglist\slice{:4}} \append{\eglist\elem{3}}
    \]
    Turning to $\eglist\slice{:4}$, we find that this is also direct, since $\seq{\eglist\elem{1},\eglist\elem{3}}$ is orderly.  There are no orderly lists of elements that begin with $\eglist\elem{2}$ and end with $\eglist\elem{3}$ and have length less than 4, since the only $i<4$ for which $\seq{\eglist\elem{2},\eglist\elem{i}}$ is orderly is 0, and the only $i<4$ for which $\seq{\eglist\elem{2},\eglist\elem{0}},\eglist\elem{i}$ is orderly is 1.  So in this case, $j=1$ and $\rho = \seq{\eglist\elem{1},\eglist\elem{3}}$. $\pi$ is the only orderly list of elements of $\eglist\slice{:4}$ beginning with $\eglist\elem{2}$ and ending with $\eglist\elem{3}$, namely $\seq{\eglist\elem{2},\eglist\elem{0},\eglist\elem{1},\eglist\elem{3}}$, and $\theta$ is thus its initial segment $\seq{\eglist\elem{2},\eglist\elem{0},\eglist\elem{1}}$.  So,
    \begin{align*}
        \makeseq{\eglist\slice{:4}} &= 
        \makeseq{\eglist\slice{:3}} + \makeseq{\seq{\eglist\elem{2},\eglist\elem{0},\eglist\elem{1}}} + \makeseq{\seq{\eglist\elem{1},\eglist\elem{3}}}
    \\
        &= 
        \makeseq{\eglist\slice{:3}} + \makeseq{\seq{\eglist\elem{2},\eglist\elem{0}}} + \makeseq{\seq{\eglist\elem{0},\eglist\elem{1}}} + \makeseq{\seq{\eglist\elem{1},\eglist\elem{3}}}
        \\
        &=
        \makeseq{\eglist\slice{:3}} + \seq{\eglist\elem{2},\eglist\elem{0},\eglist\elem{1}}
    \end{align*}
    Turning finally to computing $\makeseq{\eglist\slice{:3}}$, we already noted that $\eglist\slice{:3}$ is circuitous, so we will be in case 2.  
    The only function  meeting the requirements on $\Psi$ is as follows:
    \begin{align*}
        \Psi(\eglist\elem{0}) &= \seq{\eglist\elem{0},\eglist\elem{1},\eglist\elem{2}} \\
        \Psi(\eglist\elem{1}) &= \seq{\eglist\elem{1},\eglist\elem{0},\eglist\elem{2}}
    \end{align*}
    So, $\chooselist(\eglist\elem{0}) = \seq{\eglist\elem{0},\eglist\elem{1}}$, $\chooselist(\eglist\elem{1}) = \seq{\eglist\elem{1},\eglist\elem{0}}$, $g(\eglist\elem{0}) = \eglist\elem{1}$, $g(\eglist\elem{1}) = \eglist\elem{0}$, and 
    \begin{align*}
    \makeseq{\eglist\slice{:3}} &= \makeseq{\seq{\eglist\elem{0},\eglist\elem{1}}} + 
    \makeseq{\seq{\eglist\elem{1},\eglist\elem{0}}} + \makeseq{\seq{\eglist\elem{0},\eglist\elem{1}}} + \makeseq{\seq{\eglist\elem{1},\eglist\elem{0}}} + \cdots
    \\
    &= 
    \seq{\eglist\elem{0},\eglist\elem{1},\eglist\elem{0},\eglist\elem{1},\ldots} 
    \\\intertext{Combining all of the above, we have}
    \makeseq{\eglist} &= \seq{\eglist\elem{0},\eglist\elem{1},\eglist\elem{0},\eglist\elem{1},\ldots,\eglist\elem{2},\eglist\elem{0},\eglist\elem{1},\eglist\elem{3},\eglist\elem{4}}
    \end{align*}
    Note that if we replace each element of this $\omega+5$-sequence with the depth-0 state-description it entails, we get the sequence
    \[
    \seq{a,b,a,b,\ldots,c,a,b,a,c}
    \]
    which we used in \cref{cirqexample} as the basis for the ordinal sequence model verifying all the relevant consistency claims.  In particular, $\makesentence{\eglist}$ is true in the ordinal sequence model based on either of these sequences, with the obvious valuation.  This will be our general strategy: given a depth-$n$ state description $s =\makesentence{\tau\append{\bot}}$, we will turn $\makeseq{\tau}$ into an ordinal sequence model in the obvious way, and we will be able to show that $s$ is true in this model.
\end{example}

\begin{definition}
    Where \textsf{L} includes \vflat, $n>0$ and $s = \makesentence{\tau\append{\bot}} \in Y_{\mathsf{L}}(A,n)$ (the set of depth-$n$ state descriptions with atoms $A$), $\mathcal{M}_s$ is the ordinal sequence-model whose sequences are the non-empty tails of $\makeseq{(\tau\append{\bot})}$, with the natural valuation: atom $p_i$ is true at tail $\sigma$ iff $\sigma\elem{0}$ entails $p_i$.  
\end{definition}

The key thing we need to show is that $s$ is true in $\mathcal{M}_s$.  To get there, we will need a few more lemmas.  First, we check that $\makeseq{}$ behaves as advertised:
\begin{lemma} \label{correctorder}
    When $\tau$ is orderly, the elements of $\makeseq{\tau}$ are exactly those of $\tau\slice{:-1}$, and their first occurrences in $\makeseq{\tau}$ come in the same order as in $\tau\slice{:-1}$.  
\end{lemma}
\begin{proof}
    By induction on the length of $\tau $. 
    
    Base cases (1, 2): obvious.

    Induction step: If $\tau$ is direct, $\makeseq{\tau}$ is 
    $\makeseq{\tau\slice{:-1}} + \makeseq{\theta}+ \makeseq{\rho}$, where $\theta$ and $\rho$ are lists of elements of $\tau$ of length $<\length{\tau}$.  By the induction hypothesis, all elements of $\tau\slice{:-1}$ except $\tau\elem{-2}$ already occur in $\makeseq{\tau\slice{:-1}}$, in the same order in which they occur in $\tau\slice{:-1}$.  Moreover, $\tau\elem{-2}$ occurs later in $\makeseq{\tau}$, either as the first element of $\makeseq{\theta}$ (if $\makeseq{\theta}$ is non-empty) or else as the first element of $\makeseq{\rho}$.  And furthermore, neither $\makeseq{\theta}$ nor $\makeseq{\rho}$ has any elements not in $\theta\slice{:-1}$ or $\rho\slice{:-1}$ respectively, hence neither has any elements not in $\tau\slice{:-1}$.  So all the elements of $\tau\slice{:-1}$ occur in $\makeseq{\tau}$, in the right order.  

    If $\tau $ is circuitous, $\makeseq{\tau}$ is  
    \begin{equation*}
            \makeseq{\chooselist(\tau\elem{0})} 
            + \makeseq{\chooselist(g(\tau\elem{0}))}
            + \makeseq{\chooselist(g(g(\tau\elem{0})))}+ \makeseq{\chooselist(g(g(g(\tau\elem{0}))))}
            + \cdots
    \end{equation*}
    where $\chooselist$ and $g$ are as in \cref{makeseqdef}, and $\makeseq{\chooselist(\tau\elem{0})} = \makeseq{\tau\slice{:-1}}$.  By the induction hypothesis, $\makeseq{\tau\slice{:-1}}$ comprises exactly the elements of $\tau\slice{:-2}$, with the same order of first occurrence.  
    Meanwhile, $\tau\elem{-2}$ is $g(\tau\elem{0})$, which is the first element of $\chooselist(g(\tau\elem{0}))$ and hence of $\makeseq{\chooselist(g(\tau\elem{0}))}$, and thus also occurs in $\makeseq{\tau}$, after all elements of $\tau\slice{:-2}$.  
    And since each subsequent term in the infinite sum is derived by applying $\makeseq{}$ to a list of elements of $\tau\slice{:-1}$, nothing not in $\tau\slice{:-1}$ occurs in any of them.
\end{proof}

\begin{definition}
    For any orderly $\tau$, we define $\makelongseq{\tau}$ to be $(\makeseq{\tau}) \append{\tau\elem{-1}}$.
\end{definition}
\begin{lemma}\label{mykeylem}
    Suppose $\tau$ is an orderly list of elements of $Y_{\mathsf{L}}(A,n)$, $\sigma \append{s}$ is a segment of $\makelongseq{\tau}$, and $q, r$ are sentences of modal depth $<n$ such that every element of $\sigma$ entails $\neg q$.  Then if $s$ entails $q>r$, every element of $\sigma$ entails $q>r$; and if $s$ entails $\neg(q>r)$ and is consistent, every element of $\sigma$ entails $\neg(q>r)$.
\end{lemma}
\begin{proof}
    By induction on the length of $\tau$.  
    The base cases for 0 and 1 are trivial.  Base case for 2: $\makelongseq{\tau} = \tau$, so the only non-trivial case is where $s = \tau\elem{1}$ and $\sigma = \seq{\tau\elem{0}}$. If $\tau\elem{1}$ entails $q>r$ and $\tau\elem{0}$ entails $\neg q$, \cref{backflattening} says that $\tau\elem{0}$ is consistent with $q>r$. Since $\tau\elem{0}$ is a depth $n$ state description and $q$ and $r$ are depth $<n$, it follows by \cref{statesarestates} that it \emph{entails} $q>r$.   Similarly, if $\tau\elem{1}$ is consistent and entails $\neg(q>r)$, \cref{backflattening} implies that $\neg(q>r)$ is consistent with, and hence entailed by, $\tau\elem{0}$.
    
    For the induction step, suppose the claim holds for lists of length $\leq k$, and suppose $\tau$ is orderly and of length $k+1$.  Let $\sigma\append{s}$ be a segment of $\makelongseq{\tau}$ such that $s$ either entails $q>r$, or is consistent and entails $\neg(q>r)$.

    Case 1: $\tau $ is direct, so 
          $ \makelongseq{\tau}$ has the form  
          \[
          \makeseq{\tau\slice{:-1}} + \makeseq{\theta}+ \makelongseq{\rho}
          \] 
          with $\theta$, $t$, and $\rho$ lists of length $\leq k$, as in Case 1 of \cref{makeseqdef}. 
          If $\sigma$ is a segment of $\makeseq{\tau\slice{:-1}}$ or $\makeseq{\theta}$ or $\makeseq{\rho}$, $\sigma \append{s}$ is a segment of $\makelongseq{\tau\slice{:-1}}$, $\makelongseq{\theta}$, or $\makelongseq{\rho}$, so the claim follows from the induction hypothesis.  If $\sigma$ is of the form $\sigma_1 + \sigma_2$ where $\sigma_1$ is a tail of $\makeseq{\theta}$ and $\sigma_2$ is an initial segment of $\makeseq{\rho}$, then we first appeal to the induction hypothesis for $\rho$ to show that every element of $\sigma_2$ agrees with $s$ on $q>r$.  In particular, $\sigma_2\elem{0}$ agrees with $s$ on $q>r$.  And $\sigma_1 \append{\sigma_2\elem{0}}$ is a segment of $\makelongseq{\theta}$, so by the induction hypothesis applied to $\theta$, every element of $\sigma_1$ also agrees with $s$ on $q>r$.  Finally, if $\sigma$ is of the form $\sigma_1 + \sigma_2$ where $\sigma_1$ is a tail of $\makeseq{\tau\slice{:-1}}$ and $\sigma_2$ is an initial segment of $\makeseq{\theta} + \makelongseq{\rho}$, then every element of $\sigma_2$ agrees with $s$ on $q>r$ by what we just showed.  But $\sigma_1 \append{\sigma_2\elem{0}}$ is a segment of $\makelongseq{\tau\slice{:-1}}$.  So by the induction hypothesis applied to $\tau\slice{:-1}$, every element of $\sigma_1$ also agrees with $s$ on $q>r$.

    Case 2: $\tau $ is circuitous.  Let the functions $\chooselist$ and $g$ be as in the definition of~$\makeseq{}$.  Then there are four possible subcases:
    \begin{itemize}
        \item[(i)]
    $\sigma$ is a segment of $\makeseq{\chooselist(t)}$ for some element $t$ of $\tau\slice{:-1}$.
    \item[(ii)] 
    $\sigma$ is of the form $\sigma_1 + \sigma_2$, where for some $t$ in $\tau\slice{:-1}$, $\sigma_1$ is a tail of $\makeseq{\chooselist(t)}$ and $\sigma_2$ is an initial segment of $\makeseq{\chooselist(g(t))}$.
    \item[(iii)] 
    $\sigma$ is of the form
    \[
        \sigma_1 + 
        \makeseq{\chooselist(g(t))} + \cdots +
        \makeseq{\chooselist(g^n(t))} + \sigma_2
    \]
    where for some $t$ in $\tau\slice{:-1}$, $\sigma_1$ is a tail of $\makeseq{\chooselist(t)}$ and $\sigma_2$ is an initial segment of $\makeseq{\chooselist(g^{n+1}(t))}$.
    \item[(iv)] 
    $s$ is $\tau\elem{-1}$ and $\sigma$ is of the form 
    \[
        \sigma_1 + 
        \makeseq{\chooselist(g(t))} + 
        \makeseq{\chooselist(g(g(t)))} + \cdots
    \]
    where for some $t$ in $\tau\slice{:-1}$, $\sigma_1$ is a tail of $\makeseq{\chooselist(t)}$.
    \end{itemize}
    In subcase (i), we can appeal directly to the induction hypothesis for $\chooselist(t)$.  In subcase (ii), we first use the induction hypothesis for $\chooselist(g(t))$ to show that every element of $\sigma_2$ agrees with $s$ on $q>r$, and then use the induction hypothesis for $\chooselist(t)$ and the fact that $\sigma_1 \append{\sigma_2\elem{0}}$ is a segment of $\makelongseq{\chooselist(t)}$ to show that every element of $\sigma_1$ also agrees with $s$ on $q>r$.  
    In subcase (iii), we first use the same method to show that every element of $\makeseq{\chooselist(g^n(t))} + \sigma_2$ agrees with $s$ on $q>r$.  Since every element of $\tau\slice{:-1}$ except $g^{n+1}(t)$ occurs in $\chooselist(g^n(t))$, and $g^{n+1}(t)$ is the first element of $\sigma_2$, and every element of $\sigma$ is in $\tau\slice{:-1}$, this is already enough to show that every element of $\sigma$ agrees with $s$ on $q>r$.  
    Finally, in subcase (iv), we first reason that since every element of $\tau\slice{:-1}$ other than $g(g(t))$ occurs in $\makeseq{\chooselist(g(t))}$, and $g(g(t))$ is the first element of $\makeseq{\chooselist(g(g(t)))}$, the elements of $\sigma$ are exactly the elements of $\tau\slice{:-1}$.  Thus every element of $\tau\slice{:-1}$ entails $\neg q$, and so by \cref{backflattening}, every element of $\tau\slice{:-1}$, and hence every element of $\sigma$, agrees with $\tau\elem{-1}$ ($=s$) on $q>r$.
\end{proof}

\begin{lemma} \label{tailtruth}
    When $s = \makesentence{\tau\append{\bot}} \in Y_{\mathsf{L}}(A,n+1)$, every sentence $p$ of depth $\leq n$
    is true at a tail $\sigma$ in the model $\mathcal{M}_s$ iff $p$ is entailed by $\sigma\elem{0}$.
\end{lemma}
\begin{proof}
    By induction on the complexity of $p$.  
    \begin{itemize}
    \item 
    \emph{Atoms:} given by the valuation.  
    \item
    \emph{Conjunction:} obvious. 
    \item
    \emph{Negation:} by \cref{statesarestates} and the fact that $\sigma\elem{0}$ is a depth-$n$ state description.
    \item
    \emph{Conditional:}  Suppose $p = q>r$; then since $p$ is of depth $\leq n$, $q$ and $r$ are of depth $<n$.  
    
    Suppose first $q > r$ is \emph{not} true at $\sigma$.  Then for some $\beta$, $q \wedge  \neg r$ is true at $\sigma \slice{\beta:}$, while $q$ is not true at $\sigma \slice{\alpha:}$ for any $\alpha<\beta$.  By the induction hypothesis, $\sigma\elem{\beta}$ entails $q \wedge  \neg r$, and $\sigma\elem{\alpha}$ entails $\neg q$ for all $\alpha<\beta$.  By MP, $\sigma\elem{\beta}$ also entails $\neg(q > r)$.  Since $\sigma\elem{\beta}$ is also consistent, by \cref{mykeylem} $\sigma\elem{\alpha}$ entails $\neg(q> r)$ for all $\alpha<\beta$.  So in particular $\sigma\elem{0}$ entails $\neg(q>r)$.  And since it is consistent, that means it does not entail $q>r$.    

    Meanwhile, if $q>r$ is true at $\sigma$, there are two cases.  In the first case, $q>\neg r$ is not true at $\sigma$, in which case $\sigma\elem{0}$ entails $\neg(q>\neg r)$ by what we just proved; by CEM, it also entails $q>r$.  In the second case, $q>\bot$ is true at $\sigma$, meaning that $q$ is false at every tail of $\sigma$.  By the induction hypothesis, every element of $\sigma$ entails $\neg q$.  Since $\sigma \append{\bot}$ is a tail of of $\makelongseq{(\tau\append{\bot})}$ (i.e.\ $\makeseq(\tau\append{\bot})\append{\bot}$), and $\bot$ entails $q>r$, we can apply the first part of \cref{mykeylem} to conclude that every element of $\sigma$, and in particular $\sigma\elem{0}$, entails $q>r$. \qedhere
    \end{itemize}
\end{proof}

\begin{lemma} \label{truein}
    $s$ is true in $\mathcal{M}_s$.  
\end{lemma}
\begin{proof}
    When $s\in Y_{\mathsf{L}}(A,n+1)$, $s = \makesentence{\tau\append{\bot}}$ for some orderly list $\tau$ of elements of $Y_{\mathsf{L}}(A,n)$.  Looking back at the definition of $\makesentence{\cdot}$, it is easy to see that $s$ is true at a sequence in an ordinal sequence model iff the depth-$n$ state-descriptions true at tails of that sequence are exactly those that occur in $\tau$, and the order of their first occurrences is their order in $\tau$.  Given \cref{tailtruth}, this means it is true in $\mathcal{M}_s$ so long as the elements of $\makeseq{(\tau\append{\bot})}$ are exactly those of $\tau$ and their first occurrences are in the same order as in $\tau$.  But we already proved that this is the case, as \cref{correctorder}.
\end{proof}

This is already enough for the basic result that \vanf is complete for ordinal sequence models.  The claim we stated as \cref{ordseqcompleteness} was a bit stronger than this in two ways: we required the models to be \emph{finite}, and we also imposed a limit on the (transfinite) length of the sequences in its domain.  To secure these refinements, we can appeal to the following lemma:
\begin{lemma} \label{finitetails}
    For any orderly $\tau$ of length $k$: if $k\leq 2$ then $\makeseq{\tau}$ has length $k-1$ and $k-1$ nonempty tails; and if $k\geq 3$, then $\makeseq{\tau}$ has length at most $\omega^{k-2}$ and has at most $\frac{3}{2}(k-1)!$ non-empty tails.  
\end{lemma}

\begin{proof}
    By induction on $k$.

    Base cases: when $\tau$ has length 1, $\makeseq{\tau}$ is empty and hence has no nonempty tails; when $\tau$ has length 2, $\makeseq{\tau}$ has length 1 and is its own only non-empty tail.

    Induction step: Suppose $\tau $ is of length $k+1$ (where $k\geq 2$).  If it is direct, then $\makeseq{\tau}$ is $\makeseq{\tau \slice{:-1}} + \makeseq{\theta} + \makeseq{\rho}$, where $\tau \slice{:-1}$, $\theta $, and $\rho$ are lists of length $\leq k$.  By the induction hypothesis, each of $\makeseq{\tau \slice{:-1}}$, $\makeseq{\theta}$, and $\makeseq{\rho}$ is of length at most $\omega^{k-2}$ with at most $\frac{3}{2}(k-1)!$ non-empty tails.
    Thus the length of $\makeseq{\tau}$ is at most $\omega^{k-2}\cdot 3 \leq \omega^{k-1}$.  Also, every non-empty tail of $\makeseq{\tau}$ is either (i) a non-empty tail of $\makeseq{\rho}$, or (ii) of the form $\sigma  + \makeseq{\rho}$, where $\sigma$ is a non-empty tail of $\makeseq{\theta}$, or (iii) of the form $\sigma  + \makeseq{\theta} + \makeseq{\rho}$, where $\sigma$ is a non-empty tail of $\makeseq{\tau\slice{:-1}}$.  So if $k>2$, the number of such tails is at most $\frac{9}{2}(k-1)! \leq \frac{3}{2}k!$, while if $k=2$, the number of such tails is at most $3 = \frac{3}{2}k!$.  

    Meanwhile, if $\tau$ is circuitous, $\makeseq{\tau}$ is of the form
    \begin{equation*}
        \makeseq{\chooselist(\tau\elem{0})} + \makeseq{\chooselist(g(\tau\elem{0}))} + \makeseq{\chooselist(g(g(\tau\elem{0})))} + \cdots
    \end{equation*}
    where $g$ is a function that maps elements of $\tau\slice{:-1}$ to other elements of $\tau\slice{:-1}$, and $\chooselist$ is a function that maps each element $t$ of $\tau\slice{:-1}$ to an orderly list of length $k$.  By the induction hypothesis, each $\makeseq{\chooselist(t)}$ is of length at most $\omega^{k-2}$, so $\makeseq{\tau}$ has length at most $\omega^{k-1}$.  Also, every non-empty tail of $\makeseq{\tau}$ is of the form 
    \begin{equation*}
        \theta + \makeseq{\chooselist(g(t))} + \makeseq{\chooselist(g(g(t)))} + \cdots
    \end{equation*}
    where $\theta$ is a tail of $\makeseq{\chooselist(t)}$ for some $t$ in $\tau\slice{:-1}$.  There are only $k$ such elements $t$, and by the induction hypothesis, each  $\makeseq{\chooselist(t)}$ has at most $\frac{3}{2}(k-1)!$ tails.  So $\makeseq{\tau}$ has at most   $\frac{3}{2}k!$ tails.
\end{proof}
Since the domain of our model $\mathcal{M}_s$ comprises the nonempty tails of an orderly sequence, this gives us the refined completeness result we wanted:
\begin{theorem}[= the completeness half of \cref{ordseqcompleteness}]
    Every consistent sentence of \vflat is true in some finite ordinal sequence model, closed under non-empty tailhood, in which every sequence has length strictly between 0 and $\omega^\omega$.
\end{theorem}
\begin{proof}
    Given a consistent sentence $p$ of modal depth $n$ with atoms from $A$, by \cref{statesarestates,exhaustivestates}, there must be a depth-$n$ state description $s \in Y_{\vflat}(A,n)$ that entails $p$ in \stal.  By \cref{truein}, $s$ will be true in $\mathcal{M}_s$, which by \cref{finitetails} is finite and obeys the specified length limit.  By the soundness of \stal for order models, $p$ is true in $\mathcal{M}_s$.\end{proof}

Since ordinal sequence models are also flat order models, it follows that:
\begin{theorem}[=the completeness half of \cref{flattcom}]
    Every consistent sentence of \vflat is true in some finite flat order model.
\end{theorem}

\section{Completeness for \vanf}
\label{app:seqcompleteness}
We turn next to the stronger logic \vanf, defined as the result of adding all instances of the following schema to \vflat:
\begin{equation}
    \ourtag{Sequentiality}
    \Box(p \to \negate p > r) \wedge \Box(q \to \negate q > r) \to 
    ((p\vee q) \to \neg (p\vee q) > r).
\end{equation}
Here is the key new fact about this logic:
\begin{lemma} \label{sequentialitylemma}
    In any logic \textsf{L} including \vanf, every circuitous list ends with~$\bot$.  
\end{lemma}
\begin{proof}
Suppose for contradiction that $\tau$ is circuitous and $\tau\elem{-1} \ne 
\bot$.  Then
\begin{align} \label{consistentthing}
    \tau\elem{0} \wedge (\neg\disj\tau\slice{:-1} \gg \tau\elem{-1})
\end{align}
is consistent, since both conjuncts are equivalent to conjuncts of $\makesentence{\tau}$, which is consistent.  But
\begin{align} \label{shortinconsistent}
    \tau\elem{0} \wedge (\neg\tau\elem{0} \gg \tau\elem{-1})
\end{align}
is inconsistent, since it is equivalent to $\makesentence{\seq{\tau\elem{0},\tau\elem{-1}}}$, and the circuitousness of $\tau$ means that $\seq{\tau\elem{0},\tau\elem{-1}}$ (being a list comprising some but not all elements of $\tau$ and ending with $\tau\elem{-1}$) is inconsistent.  Likewise,
\begin{align} \label{longinconsistent}
    \disj\tau\slice{1:-1} \wedge (\neg\disj\tau\slice{1:-1} \gg \tau\elem{-1})
\end{align}
must also be inconsistent.  For supposing it were consistent, then by \cref{putinorder}, there would have to be some list $\tau'$ of elements of $\tau\slice{1:-1}$ such that $\bigvee\tau\slice{1:-1}$ is consistent with $\makesentence{\tau'\append{\tau\elem{-1}}}$.  Since  $\bigvee\tau\slice{1:-1}$ is not consistent with $\tau\elem{-1}$, this $\tau'$ must have length at least 1.  Thus, $\tau'\append{\tau\elem{-1}}$ would be an orderly list of elements of $\tau$, of length at least 2, with the same last element as $\tau$, but excluding $\tau\elem{0}$.  But the circuitousness of $\tau$ means there can be no such list.

The inconsistency of \cref{shortinconsistent} means that $\vdash \tau\elem{0} \to (\neg\tau\elem{0} > 
\neg\tau\elem{-1})$, and the inconsistency of \cref{longinconsistent} means that $\vdash \disj\tau\slice{1:-1} \to (\neg\disj\tau\slice{1:-1} > \neg\tau\elem{-1})$.  Since the logic is closed under necessitation, it follows that $\vdash \Box (\tau\elem{0} \to (\neg\tau\elem{0} > 
\neg\tau\elem{-1}))$ and $\vdash \Box(\disj\tau\slice{1:-1} \to (\neg\disj\tau\slice{1:-1} > \neg\tau\elem{-1}))$.  Noting that $\tau\elem{0} \vee \disj\tau\slice{1:-1}$ is just $\disj\tau\slice{:-1}$, we can apply Sequentiality (with $p\coloneqq \tau\elem{0}$, $q\coloneqq \disj\tau\slice{1:-1}$, $r\coloneqq \neg\tau\elem{-1}$) to conclude that $\vdash 
\disj\tau\slice{:-1} \to (\neg\disj\tau\slice{:-1} > \neg\tau\elem{-1})$.   This implies $\vdash 
\tau\elem{0} \to (\neg\disj\tau\slice{:-1} > \neg\tau\elem{-1})$, contradicting the consistency of \cref{consistentthing}.
\end{proof}

Since the definition of $\makeseq{}$ outputs a sequence of infinite length only when the input sequence is circuitous, this implies
\begin{lemma}
    In any logic including \vanf, whenever $\tau $ is orderly and does not end with $\bot$, $\makeseq{\tau}$ has finite length.
\end{lemma}
\begin{proof}
    A straightforward induction on the length of $\tau$.
\end{proof}

\begin{lemma}
    In any logic including \vanf, whenever $\tau $ is orderly, $\makeseq{\tau}$ is at most of length $\omega$.
\end{lemma}
\begin{proof}
    Given the previous lemma, it suffices to prove the result when the last element of $\tau$ is $\bot$.  We do so by induction on the length of $\tau $.  The base cases (1 and 2) are trivial.  For the first part of the induction step, suppose $\tau $ is direct, so that $\makeseq{\tau}$ is of the form $\makeseq{\tau\slice{:-1}}+\makeseq{\theta}+\makeseq{\rho}$.  Since neither $\tau\slice{:-1}$ nor $\theta$ ends with $\bot$, the first two summands are finite by the previous lemma, and the third summand is at most of length $\omega$ by the induction hypothesis, so $\makeseq{\tau}$ is at most of length $\omega$.  For the second part of the induction step, suppose $\tau $ is circuitous. Then $\makeseq{\tau}$ is  
    \begin{equation*}
        \makeseq{\chooselist(\tau\elem{0})} + \makeseq{\chooselist(g(\tau\elem{0}))} + \makeseq{\chooselist(g(g(\tau\elem{0})))} + \cdots
    \end{equation*}
    where each $\chooselist(t)$ is a non-empty list not ending in $\bot$.  By the previous lemma, all these sequences are finite; so their concatenation has order type $\omega$.  
\end{proof}

Combining this with \cref{truein}, we have:
\begin{theorem} \label{csfscompleteness}
    \vanf is complete for finite ordinal sequence models closed under non-empty tailhood in which all sequences have order-type greater than 0 and no greater than $\omega$. 
\end{theorem}

As a corollary, we have:
\goweakcomp*
\begin{proof}
    
    Let  $p$ be a sentence consistent in \vanf; then by
    \cref{csfscompleteness} it is true in a finite ordinal sequence model $\seq{\sigma,W,V}$ closed under non-empty tailhood, in which all sequences have length more than 0 and at most $\omega$.  We can assume that $W=\tails{\sigma}$ since excluding elements not accessible from the designated sequence gives an equivalent model.  If $\length{\sigma}=\omega$ then this is our desired $\omega$-sequence model. Otherwise, we note that the sequences produced via \cref{makeseqdef} never end with copies of the same element (i.e. for no such sequence $\sigma$ is $\sigma\elem{-1}=\sigma\elem{-2}$; we leave the proof to the reader).  So we can assume that $\sigma$ is a list whose last two elements are distinct.  
    Then consider the model derived from $\seq{\sigma,\tails{\sigma},V}$ by replacing each list $\rho$ in $\tails{\sigma}$ with the result of extending $\rho$ to an $\omega$-sequence by repeating its last element $\omega$ times, which we denote as $(\rho)_\omega$.  Since $\sigma$ does not end with copies, $(\cdot)_\omega$ is a bijection from $\tails{\sigma}$ to $\tails{(\sigma)_\omega}$.; Moreover, a little reflection shows that $\tau\preceq_\sigma\rho$ iff $(\tau)_\omega\preceq_{(\sigma)_\omega}(\rho)_\omega$. Where  $V'((\rho)_\omega)=V(\rho)$ for all $\rho\in\tails{\sigma}$,   $\seq{(\sigma)_\omega,\tails{(\sigma)_\omega},V'}$  is thus equivalent  to $\seq{\sigma,\tails{\sigma},V}$, and hence is our desired $\omega$-sequence model of $p$. 
    
\end{proof}
Note too that every ordinal sequence model whose sequences have length at most $\omega$ is \emph{ancestral}: every tail of every sequence can be reached by successively deleting the initial element. So we can also draw the following corollary from \cref{csfscompleteness}:
\golancestralcomp*

\section{Adding the McKinsey axiom}
\label{app:mckinsey}
In this section we consider logics \textsf{L} that extend \textsf{C2.FM}, the result of adding the McKinsey axiom to \vflat:
\begin{equation}
    \ourtag{McKinsey}
    \lozenge \Box p \vee  \lozenge \Box \neg p
\end{equation}

We'll use the following consequence of McKinsey in the context of \textsf{S4} (which \vflat includes):
\begin{equation}
   \ourtag{M*}
    \vdash \Box (p_1 \vee  \cdots \vee  p_n) \to  (\lozenge \Box p_1 \vee  \cdots \vee  \lozenge \Box p_n)
\end{equation}
\begin{lemma}\label{mstar}
    Every instance of the schema M* is a theorem of the modal logic \textsf{S4.M} (the result of adding every instance of McKinsey to \textsf{S4}).
\end{lemma}
\begin{proof}
    By induction on $n$.  Base case trivial.  Induction step: suppose $\Box (p_1 \vee  \cdots \vee  p_n)$. By McKinsey we have $\lozenge \Box \neg p_n \vee  \lozenge \Box p_{n}$, hence $\lozenge \Box (p_1 \vee  \cdots \vee  p_{n-1}) \vee  \lozenge \Box p_{n}$.  By the induction hypothesis, $\lozenge \lozenge \Box p_1 \vee  \cdots \vee  \lozenge \lozenge \Box p_{n-1} \vee  \lozenge \Box p_{n}$.  So by 4, $\lozenge \Box p_1 \vee  \cdots \vee  \lozenge \Box p_{n-1} \vee  \lozenge \Box p_{n}$.  
\end{proof}

McKinsey gives us the following counterpart to the main result (\cref{sequentialitylemma}) concerning Sequentiality:
\begin{lemma} \label{mckin}
    In any logic \textsf{L} including \textsf{C2.FM}, no  circuitous list ends with $\bot$.
\end{lemma}
\begin{proof}
    Suppose that $\tau$ is orderly and ends with $\bot$; then $\neg (\disj\tau \slice{:-1}) > \bot $, i.e.\ $\Box (\tau\elem{0} \vee \cdots \vee \tau\elem{-2})$, is consistent.  So by \cref{mstar}, $\lozenge \Box \tau\elem{0}\vee \cdots\vee \lozenge \Box \tau \elem{-2}$ is consistent, so there must be some $p$ in $\tau\slice{:-1}$ such that $\lozenge \Box p$, and hence also $\Box p$, is consistent.  In that case $\seq{p,\bot}$ is orderly, so $\tau $ is not circuitous.  
\end{proof}

Using this, we can show:
\begin{lemma}
    In any logic including \textsf{C2.FM}, whenever $\tau $ is orderly and ends with $\bot$, the length of $\makeseq{\tau}$ is a successor ordinal.
\end{lemma}
\begin{proof}
    Induction on the length of $\tau$.  The base cases (1 and 2) are trivial.  For the induction step, suppose $\tau$ is orderly and ends with $\bot$: then it is direct by \cref{mckin}, so $\makeseq{\tau}$ is of the form $\makeseq{\tau\slice{:-1}}+\makeseq{\theta}+\makeseq{\rho}$, where $\rho$ is shorter than $\tau$ and also ends with $\bot$; by the induction hypothesis, the length of $\makeseq{\rho}$ is a successor, so the length of $\makeseq{\tau}$ is a successor too.
\end{proof}

Combining this with \cref{truein}, we have:
\begin{theorem}[the completeness half of \cref{fmcomp}]
   \textsf{C2.FM} is complete for finite ordinal sequence models closed under non-empty tailhood in which the domains of all sequences are successor ordinals.
\end{theorem}
And putting together \cref{sequentialitylemma,mckin,truein}, we have
\begin{theorem}[the completeness half of \cref{FSMcomp}]
    \textsf{C2.FSM} is complete for finite list-models (i.e., ordinal sequence models closed under non-empty tailhood whose domain consists of finitely many ordinal sequences, each of finite length).
\end{theorem}

\section{Languages without left-nesting\label{appendix:lba}}
In this section, we show that all theorems of $\mathsf{C2.FSM}$ in the language $\mathcal{L}_{\mathcal{B}>}$ (in which conditionals must have Boolean antecedents) are already theorems of $\mathsf{C2.F}$. 

Given a valuation $V$ on a set $P$ (of ``protoworlds''), let $\mathcal{M}_{V,\alpha}$ be the ordinal sequence model whose domain is the set of all non-empty sequences over $P$ with length $\leq\alpha$, with the valuation given by applying $V$ to the first element of each sequence.  We define a function $h$ that takes a sequence $\sigma$ in this model's domain and a $\mathcal{L}_{\mathcal{B}>}$-sentence $p$, and returns a set $h(\sigma,p)$ of ordinals in the domain of $\sigma$---intuitively, the elements of $\sigma$ that are ``relevant'' for the truth value of $p$ at $\sigma$.  Here is the definition:
\begin{align*}
    h(\sigma,p_i) &:=  \{0\} \text{ for }p_i \text{ an atom} \\
    h(\sigma,\neg p) &:=  h(\sigma,p) \\
    h(\sigma,p\wedge q) &:=  h(\sigma,p) \cup  h(\sigma,q) \\
    h(\sigma,p>q) &:=  
    \begin{cases}
        \{0\} \cup  \{\alpha + \beta \mid \beta\in h(\sigma\slice{\alpha:},q)\} &\text{if $\sigma$ has a first $p$-tail, $\sigma\slice{\alpha:}$} \\
        \{0\} &\text{otherwise}
    \end{cases}
\end{align*}
Obviously $h(\sigma,p)$ is always a finite set of ordinals containing 0.  

Any set $X$ of ordinals is well-ordered by $\leq $, and hence there is an order-preserving bijection $f_X$ from $X$ to some ordinal $\alpha_X$.  Thus, for any sequence $\sigma$ and set $X$ of ordinals in its domain, we can construct a new sequence $\restrict{\sigma}{X}$ of length $\alpha_X$, defined by $(\restrict{\sigma}{X})\elem{\alpha}$ = $\sigma\elem{f_X^{-1}(\alpha)}$.  Note that when $\alpha \in  X$, we have $\restrict{\sigma\slice{\alpha:}}{\set{\beta:\alpha+\beta\in X}}$ = $(\restrict{\sigma}{X})\slice{f_X(\alpha):}$.  Since $\restrict{\sigma}{X}$ cannot be longer than $\sigma$ and must be nonempty, it is guaranteed to be in the domain of $\mathcal{M}_{V,\alpha}$ if $\sigma$ is. 
\begin{lemma}
    Suppose $X$ is a set of ordinals such that $h(\sigma,p)\subseteq X$.   Then, in $\mathcal{M}_{V,\alpha}$, the truth value of $p$ at $\sigma$  is the same as the truth value of $p$ at $\restrict{\sigma}{X}$.
\end{lemma}
\begin{proof}
    By induction on the complexity of $p$.  For atoms, this follows from the fact that restricting any sequence by a set of ordinals that includes 0 yields a sequence with the same first element.  For negation and conjunction it is obvious.  

    For a conditional $p>q$ (where $p$ is Boolean), note first that if no protoworld where $p$ is true occurs in $\sigma$, this will also be true of $\restrict{\sigma}{X}$.  So suppose that a $p$-protoworld occurs for the first time at position $\alpha$ in $\sigma$.  Since no $p$-protoworlds occur in $\sigma$ before position $\alpha$, none occur in $\restrict{\sigma}{X}$ before position $f_X(\alpha)$.  And of course $(\restrict{\sigma}{X})\elem{f_X(\alpha)} = \sigma\elem{\alpha}$, which is a $p$-protoworld; and so the truth value of $p>q$ at $\restrict{\sigma}{X}$ is the same as the truth value of $q$ at $(\restrict{\sigma}{X})\slice{f_X(\alpha):}$, i.e.\ at $\restrict{\sigma\slice{\alpha:}}{\set{\beta:\alpha+\beta \in  X}}$.  But $h(\sigma,p>q)$ = $\{0\} \cup \set{\alpha+\beta:\beta\in h(\sigma\slice{\alpha:},q)} \subseteq X$, so $h(\sigma\slice{\alpha:},q) \subseteq \set{\beta:\alpha+\beta \in  X}$.  So by the induction hypothesis, the truth value of $q$ at $\restrict{\sigma\slice{\alpha:}}{\set{\beta:\alpha+\beta \in  X}}$ is identical to the truth value of $q$ at $\sigma\slice{\alpha:}$, which is in turn identical to the truth value of $p>q$ at $\sigma$.
\end{proof}

Taking $X = h(\sigma,p)$, we have:
\begin{corollary}
    The truth value of any $\mathcal{L}_{\mathcal{B}>}$-sentence $p$ at $\sigma$ in $\mathcal{M}_{V,\alpha}$ is the same as its truth value at the list $\restrict{\sigma}{h(\sigma,p)}$.  
\end{corollary}
Finally, since throwing all sequences other than the list $\restrict{\sigma}{h(\sigma,p)}$ and its non-empty tails out of the domain of $\mathcal{M}_{V,\alpha}$ will not affect the truth value of any sentence at $\restrict{\sigma}{h(\sigma,p)}$, we can derive:
\begin{lemma}
    Every $\mathcal{L}_{\mathcal{B}>}$ sentence that is consistent in \vflat is true in some finite list-model.  
\end{lemma}
Given the soundness of \textsf{C2.FSM} for list models, this immediately implies
\begin{theorem}
    When $p \in \mathcal{L}_{\mathcal{B}>}$, $\vdash_{\vflat}p$ iff $\vdash_{\textsf{C2.FSM}}p$.
\end{theorem}
Since \textsf{C2.FSM} includes \vanf, this immediately implies \cref{noleftnest} (which was the same but with \vanf in place of \vflatp).

\section{Equivalence of different axiomatizations of \vanf} 
\setcounter{equation}{0}
\label{app:restrictedseq}
In this section, we show that the following axiom schemas are equivalent over \vflat. The third was not mentioned in the main text, but is useful for proving the equivalence of the other two, and yields an alternate axiomatization which is simpler, as measured by number of symbols in the official language, albeit somewhat harder to grok.
\begin{align}
    \ourtag{Sequentiality}
    \Box(p \to (\negate p > r)) \wedge \Box(q \to (\negate q > r)) &\to ((p\vee q) \to (\neg(p \vee q) > r))
    \\
    \ourtag{Restricted Sequentiality}
    \Box(p \to (\negate p > q)) \wedge \Box(q \to (\negate q > p)) &\to ((p\vee q) \to \Box(p\vee q))
    \\ 
    \ourtag{Conditional Sequentiality}
    ((\negate p > \negate q) > \negate p) \wedge ((\negate q > \negate p) > \negate q) &\to ((p\vee q) \to \Box(p \vee q))
\end{align}

\begin{theorem}
    For any logic \textsf{L} containing \vflat, the results of adding Sequentiality, Restricted Sequentiality, and Conditional Sequentiality are all identical.
\end{theorem}

\begin{proof}
    For the purposes of this proof, let $s \coloneqq (\negate{p} > \negate{q}) > \negate{p}$ and $t\coloneqq (\negate{q} > \negate{p}) > \negate{q}$, so Conditional Sequentiality is $st \to ((p\vee q) \to \Box(p\vee q))$.    

    (a) To derive Restricted Sequentiality from Sequentiality, just let $r\coloneqq p\vee q$ in Sequentiality.

    (b) To derive Conditional Sequentiality from Restricted Sequentiality, it suffices to derive
    \begin{align}
        & \vdash pst \to \neg pst > qst 
        \label{thingweneed}\\
        & \vdash qst \to \neg qst > pst
        \label{parallelthing}
    \end{align}
    For given these, we can apply necessitation and then Restricted Sequentiality to derive
    \begin{align}
        & \vdash (pst \vee qst) \to \Box(pst \vee qst)
    \\\intertext{which implies}
        & \vdash st \to ((p \vee q) \to \Box(p\vee q))
    \end{align}
So, let's see how to establish \cref{thingweneed}; the proof of \cref{parallelthing} will be parallel.   
    For contradiction, assume $p$, $s$, $t$, and $\neg(\neg pst > qst)$.  By CEM we have $\neg pst > \neg qst$ and hence (i) $\neg pst > (st\to\negate{p}\snegate{q})$.  The converse (ii) $(st\to \negate{p}\snegate{q}) > \neg pst$ is trivially true.  Noting that the antecedents $\negate{p}>\negate{q}$ and $\negate{q}>\negate{p}$ of $s$ and $t$ each entail $st\to \negate{p}\snegate{q}$, we can apply the Flattening (or Cautious Exportation) Rule to our assumptions $s$ and $t$ to derive $(st\to\negate{p}\snegate{q}) > s$ and $(st\to \negate{p}\snegate{q}) > t$, and hence (iii) $(st\to \negate{p}\snegate{q}) > \negate{p}\snegate{q}$. 
    By Reciprocity, (i)--(iii) jointly imply $\neg pst > \negate{p}\snegate{q}$, and hence $\negate{p} > \negate{q}$ by CMon.  So by $s$ and MP we have $\negate{p}$, contradicting our assumption $p$.

    (c) To derive Sequentiality from Conditional Sequentiality, assume $\Box(p \to (\negate p > r))$, $\Box(q \to (\negate q > r))$, and $p\vee q$; we want to show $\negate{p}\snegate{q} > r$.  If $\Box(p\vee q)$ this is vacuously true, so we can further suppose $\Diamond(\negate{p}\,\negate{q})$.  We can then apply Conditional Sequentiality to conclude that at least one of $s$ and $t$ is false.  Suppose without loss of generality that it's $s$.  Then by CEM, $(\negate p > \negate q) > p$.  So by MOD and the first assumption, $(\negate p > \negate q) > (\negate p > r)$, hence $(\negate p > \negate q) > (\negate p > \negate{q}r)$.  By CMon, this implies $(\negate p > \negate q) > (\negate p\, \negate{q} > r)$.  Since the antecedent of the right-nested conditional entails the first antecedent $(\negate{p}>\negate{q})$, we can apply the Flattening (or Cautious Importation) Rule to infer that $\negate p\,\negate{q} > r$.
\end{proof}

\section{Probabilities}
This appendix proves facts stated in \cref{probsagain}.  We are interested in the question whether non-trivial probabilistic order models can be $(\mathcal{P}>\mathcal{Q}\mid\mathcal{R})$-Stalnaker, for various classes of sentences $\mathcal{P},\mathcal{Q},\mathcal{R}$ (always containing $\top$ and $\bot$). It will be useful to have several equivalent ways of formulating these properties:
\begin{lemma}
    \label{stalequivalents}
     Properties (1--8) of a probabilistic order model are equivalent for any $\mathcal{P,Q,R}$ containing $\top$ and $\bot$:
    \begin{itemize}
        \item[(1)]
        The model is $(\mathcal{P}>\mathcal{Q}\mid\mathcal{R})$-Stalnaker: that is, $\cprob{p>q}{r} = \cprob{q}{p}$ when $p\in\mathcal{P},q\in\mathcal{Q}, r\in\mathcal{R}$, $\prob{p}>0$, and $\cprob{r}{p} = 1$.  
        \item[(2)]
        $\cprob{p\gg q}{r} = \cprob{q}{p}$, under the same condition.
        \item[(3)]
        $\cprob{p>q}{r} = \prob{p>q}$, weakening the condition $r\in \mathcal{R}$ to $r\in\mathcal{R}\cup\{p\}$.
        \item[(4)]
        $\cprob{p\gg q}{r} = \prob{p \gg q}$ under the same condition.
        \item[(5)] $\cprob{p> q}{\negate{r}}$, with the further condition that $\prob{r}<1$.
        \item[(6)] $\cprob{p> q}{\negate{r}}$, under the same conditions.
        \item[(7)] $\cprob{p> q}{\negate{r}} = \prob{p>q}$, under the same conditions.
        \item[(8)] $\cprob{p\gg q}{\negate{r}} = \prob{p\gg q}$, under the same condition.
    \end{itemize}
    Moreover, provided that $\mathcal{P}\wedge\mathcal{R} \subseteq \mathcal{P}$ (the conjunction of any element of $\mathcal{P}$ with an element of $\mathcal{R}$ is in $\mathcal{P}$), (1--8) are equivalent to (1$'$--8$'$), where these are the same as (1--8) but with the condition `$\cprob{r}{p} = 1$' replaced by `$\sem{p}\subseteq\sem{r}$'.
\end{lemma}
\begin{proof}
    $1\Leftrightarrow 2$: $\cprob{p> q}{r} = \cprob{(p\gg q)\vee (p>\bot)}{r} = \cprob{p\gg q}{r} + \cprob{p>\bot}{r} = \cprob{p\gg q}{r} + \cprob{\bot}{p} = \cprob{p\gg q}{r}$.

    $1\Rightarrow 3$: $\prob{p>q} = \cprob{p>q}{\top} = \cprob{q}{p} = \cprob{p>q}{r}$ for all $r\in \mathcal{R}$; also $\cprob{q}{p} = \cprob{p>q}{p}$ by MP and And-to-if. 

    $3\Rightarrow 1$: $\cprob{q}{p} = \cprob{p>q}{p}$ (by MP and And-to-if) ${} = \prob{p>q}$ (since $p\in\mathcal{R}\cup\set{p}$) $= \cprob{p>q}{r}$.  

    $1\Rightarrow 5$: 
    Given $\prob{r}<1$, $\cprob{q}{p} = \cprob{p>q}{\top} = \prob{p>q} = \cprob{p>q}{r}\prob{r} + \cprob{p>q}{\negate{r}}\pi(\negate{r}) = \cprob{q}{p}\prob{r} + \cprob{p>q}{\negate{r}}\pi(\negate{r})$, ando $\cprob{q}{p}(1-\prob{r)} = \cprob{p>q}{\negate{r}}\pi(\negate{r})$. But $\prob{\negate{r}} = 1-\prob{r} \mathrel{>} 0$, and hence $\cprob{q}{p} = \cprob{p>q}{\negate{r}}$.

    $5\Rightarrow 1$: 
    Since we require $\bot\in\mathcal{R}$, $\prob{p>q} = \cprob{p>q}{\neg\bot} = \cprob{q}{p}$.  If $\prob{r} = 1$, then $\cprob{p>q}{r} = \prob{p>q} = \cprob{q}{p}$.  Otherwise, $\cprob{q}{p} = \prob{p>q} = \cprob{p>q}{r}\prob{r}+\cprob{p>q}{\negate{r}}\pi(\negate{r}) = \cprob{p>q}{r}\prob{r}+\cprob{q}{p}\pi(\negate{r})$, and so $\cprob{q}{p}(1-\prob{\negate{r}}) = \cprob{p>q}{r}\prob{r}$.  But $\prob{r} = 1-\prob{\negate{r}} \mathrel{>} 0$, and hence $\cprob{q}{p} = \cprob{p>q}{r}$.  

    $n+4 \Leftrightarrow m+4$: same as $n\Leftrightarrow m$.  

    $1 \Rightarrow 1'$: 
    Trivial, since if $\sem{p}\subseteq\sem{r}$, $\cprob{r}{p} = 1$.  

    $1'\Rightarrow 1$: Suppose $\cprob{r}{p} = 1$.  Then $\cprob{p>r}{\top} = 1$, so $\prob{p>pr \wedge pr>p} = 1$, so by Reciprocity, $p$ and $pr$ are substitutable in the antecedents of conditionals, with probability 1.  So $\cprob{p>q}{r} = \cprob{pr>q}{r} = \cprob{q}{pr}$ (since $pr \in \mathcal{P}$)${}= \cprob{qr}{p}/\cprob{r}{p} = \cprob{q}{p}$.  

    $m' \Leftrightarrow n'$: Same as $m\Leftrightarrow n$.
\end{proof}

Before turning to our `tenability' results proper, we will prove a result specific to \emph{flat} probabilistic order-models, taking us from weaker to stronger forms of Stalnaker's Thesis.  We start with an easy result:
\begin{lemma} \label{flatbackground}
    Suppose $\mathcal{P}$ is a set of sentences closed under conjunction.  If a probabilistic order model is flat and $(\mathcal{P}>\mathcal{L}_>)$-Stalnaker, it is $(\mathcal{P}>\mathcal{L}_>\mid\mathcal{P})$-Stalnaker.
\end{lemma}
\begin{proof}
    By \cref{stalequivalents} (3$'$), it suffices to show that when $p,r\in \mathcal{P}$ and $\sem{p}\subseteq\sem{r}$,
    $\cprob{p>q}{r} = \prob{p>q}$.  Since the model is $(\mathcal{P}>\mathcal{L}_>)$-Stalnaker, $\cprob{p>q}{r} = \prob{r>(p>q)}$.  But since the model is flat and $\sem{p} = \sem{pr}$, $\prob{r>(p>q)} = \prob{r>(pr>q)} = \prob{pr>q} = \prob{p>q}$.  
\end{proof}
Using this, we can establish
\stalextension*
\begin{proof}
    Fix a flat, $(\mathcal{P}>\mathcal{L}_>)$-Stalnaker order model $\seq{W, \prec, V, \pi}$.  By \cref{flatbackground}, we will be done if we can show that the model is $(\seq{\mathcal{P}>\mathcal{P}}_{\wedge}>\mathcal{L}_>)$-Stalnaker: that is, for any sentences $p$ and $q$ such that $p\in \seq{\mathcal{P}>\mathcal{P}}_{\wedge}$ and $\prob{p}>0$, $\pi(p>q) = \pi(q\mid p)$.  Any such $p$ is of the form
    \begin{align*}
        p &\coloneqq \bigwedge_{i\in I} (r_i>s_i)
    \end{align*}
    where $I = \{k:k < n\}$ (for $n>0$), and each $r_i$ and $s_i$ is in $\mathcal{P}$.  
    
    \newcommand{\alltrue}[1]{t_{#1}}
    \newcommand{\allfalse}[1]{f_{#1}}
    \newcommand{\sometrue}[1]{\negate{f_{#1}}}
    For any $X\subseteq I$, define:
    \begin{align*}
        \alltrue{X} &\coloneqq \bigwedge_{i\in X} r_i
        &
        \allfalse{X} &\coloneqq 
        \bigwedge_{i\in X} \negate{r_i}
        &
        b_X &\coloneqq 
        \bigwedge_{i\in X}r_i\negate{s_i} 
        &
        c_X &\coloneqq
        \bigwedge_{i\in X}(r_i\gg \negate{s_i})
    \end{align*}

    Here is a useful general property of probabilistic independence: when $p$ is probabilistically independent of $q$, $r$, and $qr$, it is also probabilistically independent of $q\vee r$.%
    \footnote{$\pi(p \wedge (q\vee r)) = \pi(pq \vee pr) 
    = \pi(pq) + \pi(pr) - \pi(pqr) = \pi(p)\pi(q) + \pi(p)\pi(r) - \pi(p)\pi(qr) = \pi(p)(\pi(q) + \pi(r) - \pi(qr)) = \pi(p)\pi(q\vee r)$.}  
    This straightforwardly generalizes to show that when a proposition is probabilistically independent of all the members of some finite family of propositions closed under conjunction, it is also probabilistically independent of the disjunction of all those propositions.  It is easy to see that $\negate{p}$ is equivalent to the disjunction of the set of propositions $c_X$ where $X$ is nonempty (as well as to the disjunction of all such propositions where $|X|=1$).  Since this set is closed under conjunction, we can show that $p>q$ is probabilistically independent of $\negate{p}$ by showing that it is probabilistically independent of each $c_X$.  That will give us what we want, since anything independent of $\negate{p}$ is independent of $p$, and $\pi(p>q\mid p) = \pi(q\mid p)$ (by MP and And-to-if).

    Each $c_X$ is in turn equivalent to the disjunction of the pairwise inconsistent propositions $c_X\alltrue{X\setminus Y}\allfalse{Y}$, or equivalently $b_{X\setminus Y}c_{Y}\allfalse{Y}$, for all $Y\subseteq X$.  So we can show that $p>q$ is probabilistically independent of each $c_X$ by showing that it is probabilistically independent of $b_{X\setminus Y}c_{Y}\allfalse{Y}$ whenever $Y\subseteq X \subseteq I$ and $X$ is nonempty.  
    
    We will prove this by induction on the size of $Y$.  
    
    \emph{Base case:} $Y = \varnothing$, so what we need to show here is that $p>q$ is probabilistically independent of $b_X$ for all nonempty $X$.  Since $p\vdash \negate{b_X}$, Flattening implies $p>q \dashv\vdash \negate{b_X} > (p > q)$.  But $\negate{b_X}\in\mathcal{P}$ and has positive probability, and the model is $(\mathcal{P}>\mathcal{L}_>)$-Stalnaker; so $\prob{p>q} = \prob{\negate{b_X} > (p > q)} = \cprob{p > q}{\negate{b_X}}$.

    \emph{Induction step:} suppose the independence claim holds whenever $Y$ has $n$ elements or fewer, and consider some $X$ and  $Y$ where $|Y| = n+1$ and $Y\subseteq X$.  We can assume $\prob{b_{X\setminus Y}c_Y\allfalse{Y}}>0$, since otherwise independence holds trivially.  Note that (by $\vee$-Distribution), $\bigvee_{Z\subsetneq Y}(\sometrue{Y} \gg \alltrue{Y\setminus Z}\allfalse{Z})$ is equivalent to $\Diamond\sometrue{Y}$, which follows from $c_Y$.  So $b_{X\setminus Y}c_Y\allfalse{Y}$ is equivalent to the disjunction of $b_{X\setminus Y}c_Y\allfalse{Y} \wedge \sometrue{Y} \gg \alltrue{Y\setminus Z}\allfalse{Z}$ for all proper subsets $Z$ of $Y$.  We will show that $p>q$ is probabilistically  independent of the former by showing that it is probabilistically independent of each of the latter.%
    \footnote{Since a proposition is probabilistically independent of a disjunction of pairwise incompatible propositions whenever it is probabilistically independent of each of them.  Since everything is trivially probabilistically independent of $\bot$, this is a special case of the previously noted generalization about families closed under conjunction.}
    Since $r_i \vdash \sometrue{Y}$ for all $i\in Y$, each $r_i\gg \negate{s_i}$ is equivalent to $\sometrue{Y} \gg (r_i \gg \negate{s_i})$ by $\gg$-Flattening, and hence
    \setcounter{equation}{0}
    \renewcommand{\theequation}{\roman{equation}}
    \begin{align} \label{firstequiv}
        c_Y &\dashv\vdash \sometrue{Y} \gg c_Y
        \\
        \intertext{We can similarly show in \vflat that}
        p>q \wedge c_Y &\dashv\vdash \sometrue{Y} \gg (p>q \wedge c_Y) \label{secondequiv}
    \end{align}
    For since $p\vdash \sometrue{Y} \vee \negate{c_Y}$, we have $p>q \dashv\vdash (\sometrue{Y} \vee \negate{c_Y}) > (p > q)$ by Flattening; likewise, since $r_i\vdash \sometrue{Y} \vee \negate{c_Y}$ for all $i\in Y$, we have $c_Y \dashv\vdash (\sometrue{Y} \vee \negate{c_Y}) \gg c_Y$, or equivalently  $c_Y \dashv\vdash (\sometrue{Y} \vee \negate{c_Y}) \gg \sometrue{Y}$.  By Reciprocity, this implies that $p>q$ is equivalent, modulo $c_Y$, to $\sometrue{Y} \gg (p > q)$.      

    Since $c_Y\alltrue{Y\setminus Z}\allfalse{Z}$ is equivalent to $b_{Y\setminus Z}c_{Z}\allfalse{Z}$,
    \eqref{firstequiv} and \eqref{secondequiv} respectively imply:
    \begin{align} \label{thirdequiv}
            c_Y \wedge \sometrue{Y} \gg \alltrue{Y\setminus Z}\allfalse{Z} &\dashv\vdash \sometrue{Y} \gg b_{Y\setminus Z}c_Z\allfalse{Z}
            \\ \label{fourthequiv}
            p>q  \wedge c_Y \wedge \sometrue{Y} \gg \alltrue{Y\setminus Z}\allfalse{Z} &\dashv\vdash \sometrue{Y} \gg (p>q \wedge b_{Y\setminus Z}c_Z\allfalse{Z})
    \end{align}
    We can then reason equationally:
    \begin{align*}
        \prob{p>q \wedge b_{X\setminus Y}c_Y\allfalse{Y} \wedge \sometrue{Y} \gg \alltrue{Y\setminus Z}\allfalse{Z}}\hspace{-8em}& \\
        &= 
        \cprob{p>q \wedge c_Y \wedge \sometrue{Y} \gg \alltrue{Y\setminus Z}\allfalse{Z}}{b_{X\setminus Y}\allfalse{Y}}\prob{b_{X\setminus Y}\allfalse{Y}}
        \\
        &=
        \cprob{\sometrue{Y} \gg (p>q \wedge b_{Y\setminus Z}c_Z\allfalse{Z})}{b_{X\setminus Y}\allfalse{Y}}\prob{b_{X\setminus Y}\allfalse{Y}} 
        \\
        \intertext{by equivalence \eqref{fourthequiv}}
        &=
        \cprob{p>q \wedge b_{Y\setminus Z}c_Z\allfalse{Z}}{\sometrue{Y}}\prob{b_{X\setminus Y}\allfalse{Y}} 
        \\
        \intertext{by \cref{stalequivalents} (4$'$), since the model is $(\mathcal{P}>\mathcal{L}_>\mid\mathcal{P})$-Stalnaker (as we know from \cref{flatbackground}), $b_{X\setminus Y}f_Y$ and $\negate{f_Y}$ are incompatible and in $\mathcal{P}$, with $\prob{\negate{f_Y}}>0$}
        &=
        \cprob{p>q}{b_{Y\setminus Z}c_Z\allfalse{Z}}\cprob{b_{Y\setminus Z}c_Z\allfalse{Z}}{\sometrue{Y}}\prob{b_{X\setminus Y}\allfalse{Y}} \\
        \intertext{since $b_{Y\setminus Z}\allfalse{Z}$ entails $\negate{f_Y}$}
        &=
        \cprob{p>q}{b_{Y\setminus Z}c_Z\allfalse{Z}}\cprob{\sometrue{Y} \gg b_{Y\setminus Z}c_Z\allfalse{Z}}{b_{X\setminus Y}\allfalse{Y}}\prob{b_{X\setminus Y}\allfalse{Y}} 
        \\
        \intertext{by the same reasoning 
        }
        &=
        \prob{p>q}\cprob{\sometrue{Y} \gg b_{Y\setminus Z}c_Z\allfalse{Z}}{b_{X\setminus Y}\allfalse{Y}}\prob{b_{X\setminus Y}\allfalse{Y}} 
        \\ 
        \intertext{by the induction hypothesis, since $Z\subsetneq Y$ and hence $|Z| < |Y|$}
        &=
        \prob{p>q}\cprob{c_Y \wedge \sometrue{Y} \gg \alltrue{Y\setminus Z}\allfalse{Z}}{b_{X\setminus Y}\allfalse{Y}}\prob{b_{X\setminus Y}\allfalse{Y}}
        \\
        \intertext{by equivalence \eqref{thirdequiv}}
        &=
        \prob{p>q}\prob{b_{X\setminus Y}c_Y\allfalse{Y} \wedge \sometrue{Y} \gg \alltrue{Y\setminus Z}\allfalse{Z}}.
    \end{align*}
Since the equation we just established holds for all $Z\subsetneq Y$, we can infer
    \begin{align*}
        \prob{p>q \wedge b_{X\setminus Y}c_Y\allfalse{Y}} = \prob{p>q}\prob{b_{X\setminus Y}c_Y\allfalse{Y}}
    \end{align*}
    which concludes the induction.
\end{proof}
We can now proceed to the tenability results for sequence and tree models.  Our basic tool will be the standard concept of \emph{product measure} (see \cite[\S 3.5]{BogachevMT} for definitions and proofs).  When $\pi$ is a countably additive probability function on a $\sigma$-algebra of subsets of $P$, and $I$ is an arbitrary index set, the product $\pi^I$ is a countably additive probability function on a $\sigma$-algebra of subsets of the function space $P^I$, which agrees with $\pi$ in the sense that when $\pi$ is defined on $X\subseteq P$ and $i\in I$, $\pi^I(\{\xi: \xi(i)\in X\})$ is defined and equal to $\pi(X)$.  The product measure also has two other key properties of \emph{independence} and \emph{invariance}, which we define as follows:

\begin{itemize}
    \item
    $A\subseteq P^I$ \emph{supervenes} on $X\subseteq I$ iff whenever $\xi$ and $\nu$ agree on $X$ (i.e.\ $\xi(i) = \nu(i)$ for all $i\in X$)  
    $\xi\in A$ iff $\nu \in A$.

    \item 
    $A$ and $B$ are \emph{orthogonal} iff for some $X\subseteq I$, $A$ supervenes on $X$ and $B$ supervenes on $I\setminus X$.
    \item 
    When $f:I\to I$, the \emph{lift} $f^*$ is the function $\mathcal{P}(P^I)\to \mathcal{P}(P^I)$ defined by $f^*(A) = \set{\xi \mid \xi\circ f \in A}$.  
    \item 
    Where $\pi$ is a probability function on $P^I$, $\pi$ has the \emph{independence property} iff for any $X,Y$ in the domain of $\pi$, if $X$ and $Y$ are orthogonal, then $\pi(X\cap Y) = \pi(X)\pi(Y)$.  
    \item
    $\pi$ has the \emph{invariance property} iff for any injective $f:I\to I$, $\pi(f^*(X)) = \pi(X)$, if the latter is defined.    
    \item
    $\pi$ is \emph{product-like} iff it has both the independence and invariance properties.  
\end{itemize}
If we pick a particular $@\in I$, we can use the product measure to turn a Boolean probability model $\seq{P,\pi,V}$ into a Boolean probability model $\seq{P^I, \pi^I, V_@}$ that extends it, by defining $V_@(\xi) \coloneqq V(\xi(@))$.  The two cases of this extension operation that matter for our purposes are (i) $I$ is some countably infinite ordinal, with $@=0$, and (ii) $I = \mathbb{N}^{<\omega}$, with $@=\seq{}$.  In both cases, we can turn the Boolean probability model $\seq{P^I, \pi^I, V_@}$ into a probabilistic order model, by endowing $P^I$ with the standard sequence or tree order function.  

To show that these are in fact probabilistic order models, we must confirm that the domain of $\pi^I$ contains the denotations of all sentences.  For this, it is sufficient to show that the domain is closed under the conditional-forming operation.  But this follows from its being closed under countable unions and intersections. We can express this uniformly for both sequence and tree models by disjunctively defining a family of functions $s_\beta$ on $I$.  When $I$ is a countably infiniteordinal $\alpha$ and $\beta$ has the property that $\beta+\gamma < \alpha$ whenever $\gamma<\alpha$, let $s_\beta(\gamma) \coloneqq \beta+\gamma$ (for all $\gamma < \alpha$).%
\footnote{This property of $\beta$ is equivalent to its being less than the first term in the Cantor Normal Form of $\alpha$, and to its being the case that whenever $\sigma$ is a sequence of length $\alpha$, $\sigma\slice{\beta:}$ is too.} 
When $I$ is $\mathbb{N}^{<\omega}$ and $n<\omega$, define $s_n$ by $s_0(\tau) \coloneqq \tau$ and $s_{n+1}(\tau) \coloneqq \seq{n}+\tau$.  Thus, when $X$ is a set of $\alpha$-sequences, $\shift{\beta}(X)$ is the set of $\alpha$-sequences $\sigma$ whose $\beta$th tail $\sigma\slice{\beta:}$ is in $X$.  And when $X$ is a set of trees, $\shift{n+1}(X)$ is the set of trees whose $n$th branch is in $X$, while $\shift{0}(X)=X$.  In each case, the conditional-forming operation can be defined as follows (with complementation  relative to $P^I$):
\begin{align*}
    A\setcond B \coloneqq 
    \bigcap_\alpha \shift{\alpha}(\complement{A})
    \cup
    \bigcup_\alpha \big(\shift{\alpha}(A\cap B) \cap \bigcap_{\beta<\alpha}\shift{\beta}(\complement{A})\big)
\end{align*}
Since all unions and intersections are countable, the domain of $\pi^I$ contains $A\setcond B$ whenever it contains $A$ and $B$.

We will start with the tenability result for sequence models.
\begin{lemma} \label{sequencefact}
    If $\seq{P^\alpha, \prec, V, \pi}$ is a full, categorical, probabilistic sequence model and $\pi$ is product-like, then $\seq{P^\alpha, \prec, V, \pi}$ is $(\firstdegconj>\mathcal{L}_>\mid\firstdegconj)$-Stalnaker.  
\end{lemma}
\begin{proof}
    By \cref{stalextension}, it is sufficient to show that any such model is $(\mathcal{B}>\mathcal{L}_>)$-Stalnaker.

    By \cref{stalequivalents} (7), this is equivalent to the claim that when $p$ is Boolean  and $\prob{p}<1$, $\cprob{p>q}{\negate{p}} = \prob{p>q}$.%
    \footnote{Since whenever $r\in\set{\top}\cup\set{p}$ and $\prob{r}<1$, $r=p$.}
    
    By the definition of $\setcond$, we have
    \begin{align*}
        \sem{p>q} &= \sem{pq} \cup (\sem{\negate{p}}\cap \shift{1}\sem{p>q})
        \\
    \intertext{and hence}
        \cprob{p>q}{\negate{p}}
        &=
        \pi(\shift{1}\sem{p>q}\mid \negate{p}) 
    \end{align*}
    But by the independence property of $\pi$, we have
    \begin{align*}
        \pi(\shift{1}\sem{p>q} \mid \negate{p}) = \pi(\shift{1}\sem{p>q})
    \end{align*}
    since $\sem{\negate{p}}$ supervenes on $\set{0}$, while $\shift{1}\sem{p>q}$
    supervenes on $\{\beta:0<\beta<\alpha\}$.  

    And by the invariance property of $\pi$, we have 
    \begin{gather*}
        \pi(\shift{1}\sem{p>q}) = \prob{p>q}
    \end{gather*}
    Combining the last three equations gives $\cprob{p>q}{\negate{p}} = \prob{p>q}$.
\end{proof}
As an immediate consequence of this, we can derive
\staltheorem*
\begin{proof}
    If the given Boolean probability model is not already countably additive, first extend it to a countably additive $\seq{P,\pi,V}$.  Then take the product over $\alpha$ to form the probabilistic sequence-model $\seq{P^{\alpha}, \prec\,, \pi^{\alpha}, V_0}$.  This extends $\seq{P,\pi,V}$, is full and categorical, and $\pi^{\alpha}$ is product-like.  So by \cref{sequencefact}, it is $(\firstdegconj>\mathcal{L}_>\mid\firstdegconj)$-Stalnaker.  
\end{proof}
    
The proof for tree models is similar, though a bit more complex, because there is no `shift' operation $f$ such that $\sem{p>q}$ agrees with $f^*\sem{p>q}$ when $p$ is false. To get around this, we will focus initially on a variant conditional-like operation $\setcond^*$, where $\zeta \in  A\setcond^* B$ iff either there is no \emph{positive} $n$ such that $\zeta\in \shift{n}(A)$, or $\zeta \in \shift{n}(B)$ where $n$ is the least such number.
That is:
\begin{align*}
    A\setcond^* B \coloneqq 
    \bigcap_{n>0} \shift{n}(\complement{A})
    \cup
    \bigcup_{n>0} \big(\shift{n}(A\cap B) \cap \bigcap_{\mathclap{\scriptstyle 0<m<n}}\shift{m}(\complement{A})\big)
\end{align*}
For $\setcond^*$, there \emph{is} an appropriate shift function.  Let $t:\mathbb{N}^{<\omega}\to\mathbb{N}^{<\omega}$ be the function that maps the empty list to $\seq{0}$ and adds 1 to the first element of any nonempty list of numbers: 
\begin{align*}
    t(\seq{}) &\coloneqq \seq{0} 
    \\t(\seq{n}+\tau) &\coloneqq \seq{n+1}+\tau
\end{align*}
Looking at \cref{fig:tree}, $t$ maps the whole fractal structure to the smaller copy extending rightwards from $\seq{0}$.  Thus, when $X$ is a set of trees, tree $\xi$ belongs to $t^*(X)$ just in case chopping off its root and its zeroth branch turns $\xi$ into a member of $X$.  This operation interacts appropriately with $\setcond^*$:
\[
    A\setcond^* C = \shift{1}(A\cap C) \cup (\shift{1}(\complement{A}) \cap t^*(A\setcond^* C))
\]
That is: a tree whose zeroth branch is in $A$ is in $A\setcond^* C$ iff its zeroth branch is in $C$, while a tree whose zeroth branch is not in $A$ is in $A\setcond^* C$ iff the result of chopping off its root and zeroth branch is in $A\setcond^* C$.  

We can use this to first prove an analogue of \autoref{sequencefact} for $\setcond^*$, which we can then carry back to $\setcond$, using the fact that $\sem{p > q}$ coincides with $\sem{p}\setcond^*\sem{q}$ in the case where $p$ is false.

\begin{lemma} \label{treefact}
    If $\seq{P^{\mathbb{N}^{<\omega}}, \prec, V, \pi}$ is a categorical, probabilistic tree model and $\pi$ is product-like, $\seq{P^\alpha, \prec, V, \pi}$ is $(\mathcal{B}>\mathcal{L}_>\mid\mathcal{B})$-Stalnaker.  
\end{lemma}
\begin{proof}
    For any $p$ and $q$, we have
    \begin{align*}
        \pi(\sem{p}\setcond^* \sem{q})
        &= 
        \pi(\shift{1}\sem{pq}) +\pi(\shift{1}\sem{\negate{p}}\cap t^*(\sem{p}\setcond^* \sem{q}))
        \\
        \intertext{by the equation above;}
        &= 
        \pi(\shift{1}\sem{pq}) +
        \pi(\shift{1}\sem{\negate{p}})
        \pi(t^*(\sem{p}\setcond^* \sem{q}))
        \\
        \intertext{by independence, since $\shift{1}\sem{\negate{p}}$ supervenes on $\{\tau:\tau\elem{0} = 0\}$ (branch 0), while $t^*(\sem{p}\setcond^* \sem{q})$ supervenes on $\{\tau:\tau\ne\seq{}\text{ and }\tau\elem{0}\ne 0\}$, since $\sem{p}\setcond^*\sem{q}$ supervenes on $\{\tau:\tau\ne\seq{}\}$;}
        &=
        \prob{pq} + \prob{\negate{p}}\pi(\sem{p}\setcond^* \sem{q})
    \end{align*}
    by invariance, since $s_1$ and $t$ are both injective.
    
    Rearranging the terms in this equation, we get
    \begin{align*}
        \prob{pq} = \pi(\sem{p}\setcond^* \sem{q})(1 - \prob{\negate{p}}) = \pi(\sem{p}\setcond^* \sem{q})\prob{p}
    \end{align*}
    and thus when $\pi(p)>0$,
    \begin{align*}
        \pi(\sem{p}\setcond^* \sem{q}) = \cprob{q}{p}.
    \end{align*}

    As we noted above, the ordinary conditional-forming operation $\setcond$ can be defined in terms of $\setcond^*$: $\sem{p>q} = \sem{pq}\cup (\sem{\negate{p}}\cap (\sem{p} \setcond^* \sem{q}))$.  Thus, when $\sem{p}\subseteq \sem{r}$, $\sem{p>q} \cap \sem{\negate{r}} = (\sem{p}\setcond^* \sem{q})\cap\sem{\negate{r}}$.  So, when $\sem{p}\subseteq \sem{r}$ and $r$ is Boolean:
    \begin{align*}
        \cprob{p>q}{\negate{r}}
        &=
        \pi(\sem{p}\setcond^* \sem{q} \mid \negate{r}) 
        \\
        &=
        \pi(\sem{p}\setcond^* \sem{q}) 
        \\
        &=
        \cprob{q}{p}
    \end{align*}
    Here the final identity is justified by what we just showed. The previous identity is justified by independence: since $r$ is Boolean and the model is categorical, $\sem{\negate{r}}$ supervenes on $\{\seq{}\}$, while $\sem{p}\setcond^* \sem{q}$ supervenes on its complement.

    By \cref{stalequivalents} (5$'$), this suffices for the model to be  $(\mathcal{B}>\mathcal{L}_>\mid\mathcal{B})$-Stalnaker.
\end{proof}

As an immediate consequence of \cref{treefact}, we can prove
\treetheorem*
\begin{proof}
    If the given Boolean probability model is not already countably additive, extend it to a countably additive $\seq{P,\pi,V}$.  Then take the product over $\mathbb{N}^{<\omega}$ to form the probabilistic tree-model $\seq{P^{\mathbb{N}^{<\omega}}, \prec\,, \pi^{\mathbb{N}^{<\omega}}, V_{\seq{}}}$, which extends $\seq{P,\pi,V}$.  Since this tree-model is categorical and $\pi^{\mathbb{N}^{<\omega}}$ is product-like, by \cref{treefact}, it is $(\mathcal{B}>\mathcal{L}_>\mid\mathcal{B})$-Stalnaker.  
\end{proof}

Finally, we prove a suite of limitative results.  Here, a `non-trivial' probability model is one with three pairwise inconsistent, positive probability Boolean sentences, one of which has probability greater than $1/2$.%
\footnote{Part (c) also goes through on a weaker definition of `nontrivial' that merely requires one of the sentences to have probability greater than $1/3$; parts (b) and (d) go through on an even weaker definition that requires only that all three have positive probability.}
\limitative*
\begin{proof}
    (a)
    Given non-triviality, we could find Boolean $p$ and $q$ with $0 \mathrel{<} \prob{p} \mathrel{<} 1/2$ and $0 < \prob{q\mid p} \leq 1/2$.  But if the model were $(\mathcal{B}>\mathcal{B}\mid\mathcal{B}\vee(\mathcal{B}>\mathcal{B})$-Stalnaker, we would have $\prob{p>q} = \cprob{q}{p} = \cprob{p>q}{p \vee (p>q)} = \prob{p>q}/\prob{p\vee (p>q)}$.  Given that $\prob{p>q} = \cprob{q}{p} > 0$, this implies $\prob{p\vee (p>q)} = 1$, hence $\prob{p>q} \geq 1-\prob{p}$, hence $\cprob{q}{p} \geq \prob{\negate{p}}$, contradicting our assumptions.

    For parts (b--d), we use the following fact about \stal (and other logics containing And-to-If and Modus Ponens): if $p>q$ entails $p$, then $p>q$ is equivalent to $pq$, and hence $\prob{p>q}=\cprob{q}{p}$ only if $\prob{pq}=\cprob{q}{p}$, which can happen only if either $\prob{pq}=0$ or $\prob{p}=1$.  
    (Similarly if $p\gg q$ entails $p$.)
    
    (b) By CMon, $(p\vee(p>q))> pq$ \stal-entails $p>q$ and hence its own antecedent $p\vee(p>q)$.  But given non-triviality, there would be Boolean $p$ and $q$ with $\prob{pq}$, $\prob{p\negate{q}}$, and $\prob{\negate{p}}$ all positive. Then $\prob{(p\vee(p>q)) \wedge pq} = \prob{pq} > 0$.  And also, $\cprob{p>q}{\negate{p}} = \cprob{q}{p} < 1$, which implies that $\prob{(p>q)\wedge \negate{p}} \mathrel{<} \prob{\negate{p}}$, and hence    $\prob{p\vee (p>q)} = \prob{p} + \prob{(p>q) \wedge \negate{p}} < \prob{p} + \prob{\negate{p}} = 1$.

    (c) In \vflat, $((p> \negate{q})>p) \gg pq$ entails its antecedent.  This can be seen by considering sequence models (appealing to \cref{ordseqcompleteness}). 
    Suppose the antecedent is true for the first time at $\sigma\slice{\alpha:}$, and $pq$ is also true there.  Then $p\negate{q}$ cannot be true at $\sigma\slice{\beta:}$ for any $\beta<\alpha$, since it entails the antecedent.  Hence, if $p>\negate{q}$ is true at any $\sigma\slice{\gamma:}$, the least such $\gamma$ is greater than $\alpha$.  But in that case, $p$ must be true at $\sigma\slice{\gamma:}$, implying that $(p>\negate{q})>p$ is true at $\sigma\slice{\beta:}$ for all $\beta<\alpha$. Since by assumption it is true for the first time at $\sigma\slice{\alpha:}$, $\alpha=0$.
    
    Given non-triviality, we could find Boolean $p$ and $q$ with $0 < \prob{\negate{p}} \leq \prob{p\negate{q}} \leq \prob{pq}$ and $\prob{pq}>1/3$.  If the model is $((\firstdeg>\mathcal{B})>\mathcal{B})$-Stalnaker, it is also $(\firstdeg>\mathcal{B})$- and $(\mathcal{B}>\mathcal{B})$-Stalnaker, so 
    $\prob{(p>\negate{q})>p} = \cprob{p}{p>\negate{q}} = \prob{p\negate{q}}/\prob{p>\negate{q}} = \prob{p\negate{q}}/\cprob{\negate{q}}{p} = \prob{p}$, which rules out the possibility that  $\prob{(p>\negate{q})>p} = 1$.  And since $\prob{p}\geq 2/3$ and $\prob{pq}\mathrel{>} 1/3$, $\prob{(p>\negate{q})>p} + \prob{pq} = \prob{p} + \prob{pq} \mathrel{>} 1$, which rules out the possibility that $\prob{(p>\negate{q})>p \wedge pq} = 0$.

    (d)
    $((pq\vee\negate{p})>(p>q))>pq$ \vanf-entails its antecedent.  For suppose $(pq\vee\negate{p})>(p>q)$ is true for the first time at $\sigma\slice{n+1:}$, and $pq$ is also true there.  Then $(pq\vee\negate{p})>(p>q)$ must be false at $\sigma\slice{n:}$.  Since $pq$ is true at $\sigma\slice{n+1:}$, that could only happen if $pq\vee\negate{p}$ is true at $\sigma\slice{n:}$ while $p>q$ is false there.  But for $p>q$ to be false there, $p\negate{q}$ must be true there: contradiction.  
    
    Given non-triviality, we could find Boolean $p$ and $q$ with $\prob{pq}$, $\prob{p\negate{q}}$, and $\prob{\negate{p}}$ all positive. Then $0 < \prob{pq} = \prob{(pq\vee\negate{p}) > (p>q))\wedge pq}$.  
    Moreover, since $\cprob{q}{p} < 1$,  $1 > \cprob{pq}{pq\vee\negate{p}} + \cprob{q}{p}\cprob{\negate{p}}{pq\vee\negate{p}}) = \cprob{p>q}{pq}\cprob{pq}{pq\vee\negate{p}} + \cprob{p>q}{\negate{p}}\cprob{\negate{p}}{pq\vee\negate{p}} = \cprob{p>q}{pq\vee\negate{p}} = \prob{(pq\vee\negate{p})>(p>q)}$.%
   
    \footnote{Note that while the entailment in part (d) is not valid in \vflat, the corresponding material conditional will still receive probability 1 (for Boolean $p$ and $q$) in any categorical sequence model equipped with a product measure. So if it is possible for a non-trivial flat probabilistic order model to be $((\mathcal{B}>\firstdeg)>\mathcal{B})$-Stalnaker, showing this will require a construction quite different from van Fraassen's.}

\end{proof}

\end{appendices}

\clearpage

\printbibliography
\end{document}